\documentclass[a4paper,11pt, oneside, reqno]{amsart}

\usepackage{inputenc,fontenc}
\usepackage{geometry}

\usepackage{tikz-cd}
\usetikzlibrary{babel}
\usepackage{hyperref}
\usepackage{amsthm}
\usepackage{amssymb}
\usepackage{amsmath}
\usepackage{enumitem}
\usepackage{mathrsfs}
\usepackage{aligned-overset}
\usepackage{scalerel}

\newcommand{\Z}{\mathbb{Z}}

\newcommand{\R}{\mathbb{R}}
\newcommand{\C}{\mathbb{C}}

\newcommand{\cotimes}[1]{\hspace{1pt}\scaleobj{0.8}{\square}_{#1}\hspace{1pt}}

\newcommand{\dgCoalg}{\mathbf{dgCoalg}}
\newcommand{\corad}{\mathrm{coRad}}

\newcommand{\End}{\mathrm{End}}
\newcommand{\Focc}{\mathbf{FOCC}}

\newcommand{\id}{\mathrm{id}}
\newcommand{\im}{\mathrm{im}}

\newcommand{\coinv}[1]{{}^{\mathrm{co}\hspace{2pt}#1}}

\theoremstyle{definition}
\newtheorem{definition}{Definition}[section]
\newtheorem{example}[definition]{Example}
\newtheorem{remark}[definition]{Remark}

\theoremstyle{plain}
\newtheorem{lemma}[definition]{Lemma}
\newtheorem{theorem}[definition]{Theorem}
\newtheorem*{theorem*}{Theorem}
\newtheorem{proposition}[definition]{Proposition}
\newtheorem{corollary}[definition]{Corollary}

\title{Codifferential Calculi on Quantum Homogeneous Spaces}

\author{Julius Benner}

\address{Mathematical Institute of Charles University, Sokolovsk\'a 83, Prague, Czech Republic}

\email{julius.benner@matfyz.cuni.cz}

\thanks{The author acknowledges support by MSCA-DN CaLiForNIA - 101119552, and is grateful for the hospitality offered by the Korteweg de-Vries Institute of Mathematics of the University of Amsterdam.}

\begin{document}
    
    \begin{abstract}
        We develop the theory of first- and higher-order codifferential calculi over coalgebras $C$ over fields $k$ with characteristic $\mathrm{char}(k)\neq 2$. For a given first-order codifferential calculus, we introduce its maximal prolongation by means of an explicit construction that associates to it a differential graded coalgebra, satisfying a universal property. For module coalgebras over a Hopf algebra $U$, we introduce the notion of an equivariant codifferential calculus. If $C$ is of the form $U\otimes_H k$ for a Hopf algebra $U$ and a right coideal subalgebra $H$ such that $U$ is faithfully flat as a left- and right $H$-module, we show that equivariant first-order codifferential calculi correspond to certain right coideals $T\subseteq \ker(\varepsilon\colon C\rightarrow k)$ called quantum tangent spaces. If $H$ is a sub bialgebra and the right $C$-coaction on $T$ is trivial, then the maximal prolongation is described in terms of a quadratic coalgebra. We further relate codifferential calculi to differential calculi and Cartan pairs over the dual algebra $C^\ast$, or more generally subalgebras thereof. We explicitly compute codifferential calculi on the coalgebra pre duals of the Podle\'s sphere and the quantized projective spaces. As an application, we give a new proof that the antiholomorphic Heckenberger--Kolb calculi on quantized projective spaces have classical dimension.
    \end{abstract}

    \maketitle

    \tableofcontents

    \section{Introduction}

    Noncommutative geometry studies noncommutative algebras $A$ by thinking of them as algebras of functions or observables on noncommutative spaces $X$. A smooth structure on $X$ corresponds to a choice of first-order differential calculus on $A$. First-order differential calculi and their higher-order counterparts are structures that have been intensively studied over the past decades. Notably, they are one of the central themes in the recent monograph \cite{BM20}. Differential calculi appear in areas as far-reaching as deep learning \cite{FZ23} and stochastics through the It\^{o} and Stratonovich calculi \cite[Section 1.9]{BM20}. Perhaps the field that has seen the most fruitful application of differential calculi, and which is at the same time the main motivation for this article, has been that of quantum groups and quantized flag manifolds. Besides the pioneering work of Woronowicz \cite{Wor89}, one of the main achievements in this area is the construction of direct $q$-deformations of the classical algebraic de Rham complexes on quantized irreducible flag manifolds. These deformations are named Heckenberger--Kolb differential calculi, after the authors of the celebrated article \cite{HK06}, and they share many important features with their classical counterpart. To add to the significance of this result, outside the irreducible setting, deformations like this have only recently appeared in the literature for the $A$-series full flag manifolds via a construction using Lusztig's root vectors \cite{BS25}. Differential calculi on quantized flag manifolds have further been used in the construction of Dolbeault--Dirac spectral triples \cite{DD10, BOBS20, WGOB22}, demonstrating that certain quantized flag manifolds fit into the framework of noncommutative geometry in the sense of Connes \cite{Con94}.

    If one takes the idea that a noncommutative manifold is a pair consisting of a noncommutative algebra $A$ and a first-order differential calculus over $A$ seriously, then there should also be the notion of noncommutative vector fields over $A$. Indeed, there is the notion of a Cartan pair introduced by Borowiec \cite{Bor96, Bor97, Bor24} which serves as a noncommutative algebraic substitute for vector fields. Unfortunately, Cartan pairs have seen very little attention elsewhere in the literature, and thus only few developments have been made. For instance, it is not known if combining a first-order differential calculus and a Cartan pair gives rise to a noncommutative version of a Cartan calculus (for an account of Cartan calculi for classical manifolds see \cite[Chapter 2.4]{AM78}).

    The passage from a first-order differential calculus to a Cartan pair is a dualization at the level of modules, see for instance the main theorem in \cite{Bor96}. If one turns to quantum groups, then there is another way of dualizing: Many examples of quantum groups --- for instance Drinfeld--Jimbo quantum groups --- arise naturally as a pair of dually paired Hopf algebras $(A,U)$. If one takes a first-order differential on $A$, then the corresponding dual structure on $U$ is neither a first-order differential calculus nor a Cartan pair, but a \emph{first-order codifferential calculus}.

    The definition of a first-order codifferential calculus essentially arises as a naive translation of the axioms of a first-order differential calculus into coalgebra language. More precisely, a first-order codifferential calculus over a coalgebra $C$ over a field $k$ of characteristic not equal to $2$ consists of a $C$-bicomodule $\mathscr{W}_1$ together with a linear map $\delta\colon \mathscr{W}_1\rightarrow C$, meaning that
    \[
        \Delta\circ \delta = (\id\otimes \delta)\circ {}_{\mathscr{W}_1}\Delta + (\delta\otimes \id)\circ \Delta_{\mathscr{W}_1}
    \]
    where ${}_{\mathscr{W}_1}\Delta$ and $\Delta_{\mathscr{W}_1}$ denote the left- and right coactions of $\mathscr{W}_1$ respectively, such that the map
    \[
        (\id\otimes \delta)\circ {}_{\mathscr{W}_1}\Delta\colon\mathscr{W}_1\rightarrow C\otimes C
    \]
    is injective (see Definition \ref{DefFOCC}).
    
    Somewhat surprisingly, the first appearance of a first-order codifferential calculus in the literature predates the work of Woronowicz \cite{Wor89}, namely in \cite{Doi81}, Doi introduced what is now called the universal first-order codifferential calculus. Other than this work, not many results about codifferential calculi can be found in the literature. The first definition of a first-order codifferential calculus appeared in \cite{BV00}, which however omits the injectivity axiom. The only account of a higher-order codifferential calculus appears in \cite{Kha97}, where --- in our terminology --- the maximal prolongation of the universal first-order codifferential calculus is constructed. The only more comprehensive study of first-order codifferential calculi, which uses a definition equivalent to the one in this article, was carried out by Borowiec and Mieszkalski in the recent preprint \cite{BM26}.

    The goal of this article is to expand the theory of codifferential calculi, by proving analogs of several results that are known for differential calculi, and in addition relate codifferential calculi to differential calculi and Cartan pairs. Firstly, we introduce higher-order codifferential calculi, and the maximal prolongation of a first-order codifferential calculus. The maximal prolongation is best understood by comparing it to the process of taking the 1-forms of a manifold, and extending them to the de Rham complex. The corresponding construction for differential calculi goes back to Schauenburg \cite{Schau96}. The precise result we obtain is Theorem \ref{ThmMaxProlongationUniversalProperty} which states that the maximal prolongation determines a functor, which is part of an adjunction between the category of first-order codifferential calculi and the category of dg coalgebras.
    
    Secondly, we investigate into cocalculi in the (left-) equivariant setting. More precisely, we consider first-order codifferential calculi $\mathscr{W}_1$ over coalgebras of the form $C = U/UH^+$, where $U$ is a Hopf algebra, $H$ is a right coideal subalgebra and $H^+$ denotes the kernel of the counit of $H$. Here, equivariance means that $\mathscr{W}_1$ is equipped with a compatible left $U$-action (in a sense made precise in Section \ref{SecEquivariantCocalculi}), and that the coderivation $\delta$ is $U$-linear. We show that if $U$ is faithfully flat over $H$, then $U$-equivariant first-order codifferential calculi on $C$ correspond to quantum tangent spaces, that is, subspaces $T\subseteq C^+ := \ker(\varepsilon\colon C\rightarrow k)$ such that
    \[
        \Delta(T) \subseteq (T\oplus k[1])\otimes C, \quad HT\subseteq T.
    \]
    The precise statement is found in Theorem \ref{ThmClassificationEquivariantFOCC}. A coalgebra as above can be thought of as a coalgebraic counterpart of a quantum homogeneous space, which for us are certain coideal subalgebras of Hopf algebras. We elucidate this interpretation in Section \ref{SecInducedQHS}. An analog of this result for a quantum homogeneous space $B\subseteq A$ goes back to Hermission \cite[Theorem 2]{Her02}. Notably, left covariant first-order differential calculi over $B$ also correspond to quantum tangent space, as shown by Heckenberger and Kolb \cite{HK03}. However, in this case one considers quantum tangent spaces inside the finitary dual $B^\circ$. If we let $H = k$, we obtain a first-order codifferential calculus-analog of a theorem by Woronowicz \cite[Theorem 1.6]{Wor89}. In the case of codifferential calculi over Hopf algebras, even more is known. See for instance the classification of bicovariant first-order codifferential calculi over Hopf algebras due to Borowiec and Mieszkalski \cite[Proposition 16]{BM26}.

    For an equivariant first-order codifferential calculus over a coalgebra $C$ as above, we further obtain a description of its maximal prolongation. Namely, if $H$ is a sub bialgebra, and the induced right $C$-coaction on the corresponding quantum tangent (see \eqref{eq:ReducedRightCoaction}) space is trivial, then its maximal prolongation can be written as $U\otimes_H \mathfrak{C}(T,\check{R})$. Here, $\mathfrak{C}(T,\check{R})$ is a quadratic coalgebra (see Appendix \ref{SecQuadraticCoalgebras}) where $\check{R} \subseteq T\otimes T$ is given in Proposition \ref{PropDegTwoRelationsSupZero}.
    
    It was pointed out in \cite{BM26} that a first-order codifferential calculus over $C$ induces a first-order differential calculus on the dual of $C$. We extend this observation to higher orders by showing that the maximal prolongation of the first-order codifferential calculus determines a prolongation of the dual calculus, which we call codifferential prolongation \ref{DefCodifferentialProlongation}. For a class of quantum homogeneous spaces, which quantized irreducible flag manifolds belong to, we show that the codifferential prolongation is isomorphic to the maximal prolongation (Theorem \ref{ThmCodiffProlongationisMaximalProlongation}), and from the proof of this statement, it follows that the maximal prolongation is determined by the quadratic dual algebra of $\mathfrak{C}(T,\check{R})$ (Corollary \ref{PorMaxProlongationQuadraticAlgebras}). To put this observation into context, this means that to compute the relations for the maximal prolongation of the dual first-order differential calculus, one might as well compute $\check{R}$. We demonstrate this for the antiholomorphic Heckenberger--Kolb calculus on $\mathcal{O}_q(\C P^n)$, which leads to a new proof that the latter has classical dimension.

    For the Podle\'s sphere $\mathcal{O}_q(S^2)$, we explicitly compute the maximal prolongation of a two-dimensional first-order codifferential calculus, and obtain a sequence of Verma-like modules for $U_q(\mathfrak{sl}_2)$. An identical sequence appeared in the introduction of \cite{HK072}, where it was noted that by taking finitary duals one recovers the unique covariant differential calculus of classical dimension on $\mathcal{O}_q(S^2)$ discovered by Podle\'s \cite{Pod92}. In \cite{HK072} it is more generally shown that the Heckenberger--Kolb calculi on arbitrary irreducible quantized flag manifolds can be recovered as finitary duals of sequences derived from quantum generalized parabolic Bernstein--Gelfand--Gelfand (BGG) resolutions introduced in their earlier paper \cite{HK071}. This marks the first indication, that codifferential calculi and BGG resolutions are related. In fact, the initial motivation for this article was to understand the construction in \cite{HK072} in a more conceptual way, and adapt it to more general quantum flag manifolds. The precise relation between codifferential calculi and BGG resolutions --- including the classical case --- is unknown up to this point and will be addressed in future work. We note that BGG resolutions have been used in the construction of spectral triples \cite{WGOB22}, and the noncommutative geometric counterparts of BGG resolutions --- BGG sequences --- have been used in the proof of the Baum--Connes conjecture for $SU_q(3)$ \cite{VY15}. Moreover, BGG sequences are one of the cornerstones of classical parabolic geometry \cite{CSS01}.\newline\phantom{a}

    \emph{Summary of the Paper.} The paper is organized as follows: In Section \ref{SecPrelim} we mainly fix notation and recall elementary facts about coalgebras and comodules over coalgebras. In addition, we recall the notion of a cotensor coalgebra of a bicomodule $M$ over a coalgebra $C$. The latter serves as a foundation for the construction of the maximal prolongation of a first-order codifferential calculus.

    In Section \ref{SecCodifferentialCalculi} we first recall the notion of a first-order codifferential calculus. Our definition slightly differs from the one given in \cite{BM26}, but we show that there are indeed equivalent (Proposition \ref{PropBicomodQuotients}). After a brief account on codifferentiable maps, we introduce the notion of a higher-order codifferential calculus, and detail the construction of the maximal prolongation of a first-order codifferential calculus.

    Section \ref{SecQHS} specializes to quantum homogeneous spaces. We recall two equivalences between categories of relative Hopf modules --- Takeuchi's equivalence and its dual counterpart named Schneider's equivalence --- as well as monoidal versions thereof. We show how a quotient coalgebra of the form $C = U/UH^+$ dualizes to a quantum homogeneous space, and end the section with an account of admissible pairings between relative Hopf modules.

    Section \ref{SecEquivariantCocalculi} contains the main results on codifferential calculi on quantum homogeneous spaces, which heavily rely on the result of Section \ref{SecSchneider}.

    In Section \ref{SecRelFOCCFODCFOVC} we construct differential calculi and Cartan pairs from first- and higher-order codifferential calculi, where we first consider the general case and then the case of quantum homogeneous spaces.

    In Section \ref{SecCodifferentialCalculionQuantizedFlags} we demonstrate our results by explicitly constructing codifferential calculi on the Podle\'s sphere and the quantized projective spaces.\newline\phantom{a}

    \emph{Acknowledgements.} I would like to thank Arnab Bhattacharjee, Andrzej Borowiec, Antonio Del Donno, Giovanni Gava, Keegan Flood, Mauro Mantegazza, and Jasper Stokman for helpful discussions. I also thank Niels Kowalzig for pointing me to the work of Khalkhali \cite{Kha97} and Ulrich Krähmer for suggesting me the paper of Masuoka \cite{Mas91}. In addition, I thank Thomas Weber for his course on Takeuchi's categorical equivalence, which has been a big source of inspiration, as well as helpful discussions. Many thanks to Réamonn \'O Buachalla for a number of discussions and guidance in writing this paper, including comments on earlier versions.

    \section{Preliminaries on coalgebras}\label{SecPrelim}

    For the entirety of this paper, $k$ denotes a field of characteristic $\mathrm{char}(k)\neq 2$, and every vector space will be considered over $k$. We start by recalling some general facts about coalgebras and comodules over coalgebras, while also laying out our notation. For the purpose of this section, we fix a coalgebra $C$.

    \subsection{Bicomodules over coalgebras}

    If $M$ is a left, respectively right comodule over $C$, we will write ${}_M\Delta$ respectively $\Delta_M$ for the left- and right $C$-coactions of $M$. We will make heavy use of Sweedler notation, that is we write $\Delta(c) = c_{(1)}\otimes c_{(2)}$ for the coproduct of an element $c \in C$, and ${}_M\Delta(m) =: m_{(-1)}\otimes m_{(0)}$, respectively $\Delta_M(m) =: m_{(0)}\otimes m_{(1)}$ for an element $m\in M$. For a detailed explanation of Sweedler notation, we refer the reader to Chapters 2 and 3 in \cite{Rad12}.

    \begin{definition}
        Let $C$ and $D$ be coalgebras. A \textit{$C$-$D$-bicomodule} is a $k$-vector space $M$ together with a left $C$-coaction and a right $D$-coaction such that
        \[
            m_{(-1)}\otimes m_{(0)(0)} \otimes m_{(0)(1)} = m_{(0)(-1)}\otimes m_{(0)(0)}\otimes m_{(1)}
        \]
        for all $m \in M$.
    \end{definition}
    
    Thus, in a $C$-$D$-bicomodule $M$, we can unambiguously write $m_{(-1)}\otimes m_{(0)}\otimes m_{(1)}$. We denote by $\mathcal{M}^C, {}^C\mathcal{M}$ and ${}^C\mathcal{M}^D$ the categories of left $C$-comodules, right $C$-comodules and $C$-$D$-bicomodules, with morphisms given by left or right $C$-colinear, respectively $C$-$D$-bicolinear maps. If $C = D$, we write \textit{$C$-bicomodule} and \textit{$C$-bicolinear} in place of $C$-$C$-bicomodule and $C$-$C$-bicolinear. We denote by ${}^C\mathrm{Hom}^D(M,N)$ the space of $C$-$D$-bicolinear maps between two $C$-$D$-bicomodules $M$ and $N$. Let $C^{op}$ be the opposite coalgebra of $C$. A $C$-$D$-bicomodule can be regarded as a right $D\otimes C^{op}$-comodule and vice versa. In one direction, this works as follows. For a $C$-$D$-bicomodule $M$ we define a right $D\otimes C^{op}$-coaction via $m \mapsto m_{(0)}\otimes m_{(1)}\otimes m_{(-1)}$.

    Next we recall the definition of a cotensor product \cite{Doi81, EM66}. For a right $C$-comodule $M$ and a left $C$-comodule $N$ the cotensor product $M\cotimes{C} N$ is defined as the kernel of the map
    \[
        M\otimes N \rightarrow M\otimes C \otimes N, \quad m\otimes n \mapsto m_{(0)}\otimes m_{(1)}\otimes n - m\otimes n_{(-1)}\otimes n_{(0)}.
    \]
    We denote elements of $M\cotimes{C} N$ symbolically by $m\cotimes{C} n$, where $m\cotimes{C} n$ is in general not an elementary tensor. If $D, E$ are other coalgebras, and if in addition $M$ is a left $D$-comodule and $N$ is a right $E$-comodule, then $M\cotimes{C} N$ is naturally a $D$-$E$-bicomodule with left $D$-coaction and right $E$-coactions given by $m\cotimes{C} n\mapsto m_{(-1)}\otimes m_{(0)}\cotimes{C} n$ and $m\cotimes{C} n \mapsto m\cotimes{C} n_{(0)}\otimes n_{(1)}$. Moreover, we obtain a functor
    \[
        -\cotimes{C}-\colon {}^D\mathcal{M}^C\times {}^C\mathcal{M}^E\rightarrow {}^D\mathcal{M}^E.
    \]
    Thus, for morphisms $f\in {}^D\mathcal{M}^C, g\in {}^C\mathcal{M}^E$, we can sense of the morphisms $f\cotimes{C}g$, which are defined as $m\cotimes{C} n\mapsto f(m)\cotimes{C} g(n)$.
    If $D = E = C$, this gives rise to a monoidal structure on ${}^C\mathcal{M}^C$, which is remarked in \cite[Section 2]{AMS07} in a more general setting (for more background on monoidal categories consult standard texts in category theory \cite{EGNO15, Mac98}). Here, the components of the left- and right unitors in ${}^C\mathcal{M}^C$ at a $C$-bicomodule $M$ are given as follows:
    \[
        C\cotimes{C} M\rightarrow M,\quad c\cotimes{C} m\mapsto \varepsilon(c)m,\quad M\cotimes{C} C\rightarrow M, \quad m\cotimes{C}c\mapsto \varepsilon(c)m.
    \]
    The inverses of these maps are given by the left- and right coactions. If we are given a morphism of coalgebras $\theta\colon D\rightarrow C$, then we get a \textit{corestriction functor} $\mathrm{coRes}_C^D\colon{}^D\mathcal{M}^D\rightarrow {}^C\mathcal{M}^C$ which sends a $D$-bicomodule $M$ to the same underlying vector space $M$ with left- and right $C$ actions given by $m \mapsto \theta(m_{(-1)})\otimes m_{(0)}$ and $m\mapsto m_{(0)}\otimes \theta(m_{(1)})$.

    Since we work over a field, we have the following well-known result at our disposal.

    \begin{lemma}
        \label{LemSubcomodGenByElt}
        Let $M$ be a right $C$-comodule, let $\{c_i\}_{i\in I}$ be a vector space basis of $C$ and let $m \in M$. Then there are unique elements $m_i \in M, i\in I$ (where $m_i = 0$ for all but finitely many $i\in I$) such that
        \[
            \Delta_M(m) = \sum_{i\in I} m_i\otimes c_i.
        \]
        Moreover, the subspace $\mathrm{span}_k\{m_i\mid i\in I\}$ is a $C$-subcomodule of $M$ containing $m$. In particular if $\Delta(m)\subseteq V\otimes C$ for some subspace $V\subseteq M$, then $V$ contains a $C$-subcomodule of $M$ containing $m$.
    \end{lemma}

    \subsection{Reduced coproducts}

    Recall that for any coalgebra $D$ that admits a coalgebra section $\iota\colon k\rightarrow D$ of its counit, we can define a reduced coproduct on $D^+ = \ker(\varepsilon)$ via
    \[
        \overline{\Delta}(d) = d_{(1)}\otimes d_{(2)} - \iota(1)\otimes d - d\otimes \iota(1)
    \]
    which remains coassociative.
    The following lemma shows that this observation generalizes to the case where we replace $\varepsilon$ by a morphism of coalgebras $\theta\colon D\rightarrow C$ and $\iota$ by a section of $\theta$.  We will use the reduced coproduct mainly in the proof of Theorem \ref{ThmMaxProlongationUniversalProperty} to simplify notation. 

    \begin{lemma}
        \label{LemRedCoprod}
        Let $\theta\colon D\rightarrow C$ be a morphism of coalgebras, let $\iota\colon C\rightarrow D$ be a morphism of coalgebras such that $\iota\circ \theta = \id_C$ and write $D^+ = \ker(\theta)$. Then the following assertions hold:
        \begin{enumerate}[label = (\arabic*)]
            \item For an element $d \in D$ write $d^+ := d - \iota(\theta(d))$. Then the subspace $ D^+$ is a $D$-bicomodule via
            \[
                \Delta_{D^+}(d) := d_{(1)}^+\otimes d_{(2)}, \quad {}_{D^+}\Delta(d) := d_{(1)}\otimes d_{(2)}^+.
            \]
            In particular $D^+$ is also a $C$-bicomodule, via $\mathrm{coRes}_D^C$. Moreover, the maps $D\rightarrow D^+, d\mapsto d^+$ and $D^+\rightarrow D, d\mapsto d$ are $C$-bicolinear.
            \item The map
            \[
                \overline{\Delta}\colon D^+\rightarrow  D^+\otimes  D^+, \quad d \mapsto d_{(1)}\otimes d_{(2)} - \iota(\theta(d_{(1)}))\otimes d_{(2)} - d_{(1)}\otimes \iota(\theta(d_{(2)}))
            \]
            turns $D^+$ into a non-counital coalgebra in $({}^C\mathcal{M}^C, \cotimes{C}, C)$. In addition, for all $d\in D^+$ it holds that
            \begin{equation}\label{eq:CoprodOfRedElt}
                \Delta(d) = d_{(-1)}\otimes d_{(0)} + d_{(0)}\otimes d_{(1)} + \overline{\Delta}(d)
            \end{equation}
            where the Sweedler notation refers to the left- and right $C$-actions on $D^+$.
        \end{enumerate}
    \end{lemma}

    \begin{definition}
        \label{DefRedCoprod}
        In the situation of Lemma \ref{LemRedCoprod}, we call $\overline{\Delta}$ the \textit{reduced coproduct} and write $d_{[1]}\otimes d_{[2]} := \overline{\Delta}(d)$ for $d \in D^+$.
    \end{definition}

    \subsection{Cotensor coalgebras}

    Cotensor coalgebras were introduced by Nichols \cite{Nich78} and later generalized by Ardizzoni, Menini and \c{S}tefan to arbitrary coalgebras in Abelian monoidal categories \cite{AMS07}. Let us recall their definition.

    \begin{definition}
        \label{DefRelTensorCoalg}
        Let $M$ be a $C$-bicomodule. The \textit{cotensor coalgebra} over $C$ of $M$ is the graded vector space
        \[
            T_C^c(M) := C \oplus M \oplus \bigoplus_{n\ge 2}M^{\cotimes{C} n}, \quad \text{where } M^{\cotimes{C} n} = \underbrace{M\cotimes{C} \dots \cotimes{C} M}_{n\text{ times}}
        \]
        together with the following coproduct $\Delta$: For $c \in C$ the coproduct is given by the coproduct of $C$, for $m \in M$, we set $\Delta(m) = m_{(-1)}\otimes m_{(0)} + m_{(0)}\otimes m_{(1)}$ and for $\underline{m} = m^1\cotimes{C} \dots \cotimes{C} m^n \in M^{\cotimes{C} n}$ where $n\ge 2$, we set
        \[
            \Delta(\underline{m}) = \underline{m}_{(-1)}\otimes \underline{m}_{(0)} + \underline{m}_{(0)}\otimes \underline{m}_{(1)} + \sum_{i=1}^{n-1} (m^1\cotimes{C}\dots\cotimes{C}m^{i})\otimes(m^{i+1}\cotimes{C}\dots\cotimes{C}m^n). 
        \]
    \end{definition}

    Recall that the coradical $\corad(D)$ of a coalgebra $D$ is given by the sum of all simple subcoalgebras of $D$ \cite[Definition 3.4.1]{Rad12}. The cotensor coalgebra has the following mapping property \cite[Proposition 1.4.2]{Nich78}.

    \begin{proposition}[Nichols '78]
        \label{PropCotensorCoalgUniversal}
        Let $M$ be a $C$-bicomodule. If $D$ and $C'$ are coalgebras, $q\colon D\rightarrow C', r\colon C'\rightarrow C$ are morphisms of coalgebras and $h\colon D \rightarrow M$ is a $C$-bicolinear map such that $h(\mathrm{coRad}(D)) = 0$ then there exists a unique morphism of coalgebras $h^\flat\colon D \rightarrow T_C^c(M)$ such that the diagrams
        \[\begin{tikzcd}
	    {T_C^c(M)} & D & {T_C^c(M)} \\
	    C & {C'} & M & D
	    \arrow[from=1-1, to=2-1]
	    \arrow["{h^\flat}"', dashed, from=1-2, to=1-1]
	    \arrow["q", from=1-2, to=2-2]
	    \arrow[from=1-3, to=2-3]
	    \arrow["r"', from=2-2, to=2-1]
	    \arrow["{h^\flat}"', dashed, from=2-4, to=1-3]
	    \arrow["h"', from=2-4, to=2-3]
        \end{tikzcd}\]
        commute (here $T_C^c(M)\rightarrow C$ and $T_C^c(M)\rightarrow M$ denote the projections onto the degree 0 and 1 part). Explicitly, the map $h^\flat$ is given by
        \[
            h^\flat(d) = r(q(d)) + \sum_{n\ge 1}h(d_{(1)})\cotimes{C}\dots\cotimes{C}h(d_{(n)}).
        \]
    \end{proposition}

    We are mostly interested in the case where $D$ is a graded coalgebra.

    \begin{example}
        \label{ExplGrcoalg}
        Recall that a (non-negatively) graded coalgebra consists of a coalgebra $D_\bullet$ together with a decomposition $D_\bullet = \bigoplus_{n\ge 0} D_n$ such that
        \begin{equation}\label{eq:grcoalg}
            \Delta(D_n) \subseteq \bigoplus_{i=0}^n D_{n-i}\otimes D_i
        \end{equation}
        for all $n\ge 0$. This condition implies that $D_0$ is a subcoalgebra of $D_\bullet$. Moreover, the projection $\pi_0\colon D_\bullet \rightarrow D_0$ is a morphism of coalgebras, such that $\pi_0\circ \iota_0 = \id_{D_0}$, where $\iota_0\colon D_0\rightarrow D_\bullet$ is the inclusion. In the notation of Lemma \ref{LemRedCoprod} we get $D_\bullet^+ = \bigoplus_{n\ge 1} D_n$. Here, the left- and right $D_0$-coactions on $D_\bullet^+$ are given by $d \mapsto \pi_0(d_{(1)})\otimes d_{(2)}^+$ and $d\mapsto d_{(1)}^+\otimes \pi_0(d_{(2)})$, and it follows that every homogeneous component $D_n$ with $n\ge 0$ is a sub $D_0$-bicomodule of $D_\bullet$. Now consider the projection $\pi_1\colon D_\bullet\rightarrow D_1$. This map is $D_0$-bicolinear. We claim that $\pi_1(\corad (D_\bullet)) = 0$. Indeed, consider the quotient map $p\colon D_\bullet\rightarrow D_\bullet/D_0$. It follows from \eqref{eq:grcoalg} that
        \begin{equation}\label{eq:GrFiltration}
            D_\bullet = \bigcup_{n\ge 1}\ker(p^{\otimes n}\circ \Delta^{(n-1)})
        \end{equation}
        (here $\Delta^{(n-1)}$ denotes the iterated coproduct). If $S$ is a simple subcoalgebra of $C$, then either $S\subseteq D_0$ or $S\cap D_0 = 0$. Suppose we have $S\cap D_0 = 0$, then $(p^{\otimes n}\circ \Delta^{(n-1)})\vert_S$ is injective for all $n\ge 1$. However, by \eqref{eq:GrFiltration} and since simple coalgebras are finite dimensional we also have $S\subseteq \ker(p^{\otimes n}\circ \Delta^{(n-1)})$ for some $n\ge 1$ which is a contradiction. Thus, we have $S\subseteq D_0$ and the claim follows. This observation allows us to construct maps from $D_\bullet$ into cotensor coalgebras. Let $M$ be a $C$-bicomodule and let $r\colon D_0\rightarrow C$ be a morphism of coalgebras. If $h\colon D_\bullet\rightarrow M$ is a $C$-bicolinear map such that $h = h'\circ \pi_1$ for some $C$-bicolinear map $h'\colon D_1 \rightarrow M$, then $h$ induces a unique morphism of coalgebras $h^\flat\colon D_\bullet \rightarrow T_C^c(M)$ by Proposition \ref{PropCotensorCoalgUniversal}.
    \end{example}

    \section{Codifferential calculi}\label{SecCodifferentialCalculi}

    We now introduce the main objects of this article, namely codifferential calculi. We begin by discussing the first order.

    \subsection{First-order codifferential calculi}\label{SecFirstOrderCoddiferentialCalculi}

    Recall that a \textit{first-order differential calculus} \cite[Definition 1.1]{Wor89} over a unital associative $k$-algebra $A$ consists of an $A$-bimodule $\Omega^1$ together with a derivation $d\colon A\rightarrow \Omega^1$ such that $\Omega^1$ is generated as a left $A$-module by $\{d(a)\mid a \in A\}$. The latter condition is equivalent to the assertion that the map
    \[
        A\otimes A \rightarrow \Omega^1, \quad a\otimes b \mapsto ad(b)
    \]
    is surjective.

    We arrive at the definition of a first-order codifferential calculus by naively dualizing the above to coalgebras. For instance, instead of a derivation, we have a coderivation \cite[p. 44]{Doi81}.

    \begin{definition}
        Let $C$ be a coalgebra and let $W$ be a $C$-bicomodule. A \textit{coderivation} is a linear map $\delta\colon W\rightarrow C$ such that
        \[
            \Delta(\delta(w)) = \delta(w_{(0)})\otimes w_{(1)} + w_{(-1)}\otimes \delta(w_{(0)})
        \]
        for all $w \in W$. We denote by $\mathrm{Coder}(W, C)$ the $k$-vector space of all coderivations on $W$ with values in $C$.
    \end{definition}

    Since $\mathrm{char}(k) \neq 2$ we have the following simple-to-prove, but essential property for any coderivation $\delta$:
    \begin{equation}\label{eq:CounitOnCoderivation}
        \varepsilon\circ \delta = 0.
    \end{equation}

    For first-order codifferential calculi we demand that an appropriate dualization of the surjectivity condition of differential calculi holds.

    \begin{definition}
        \label{DefFOCC}
        Let $C$ be a coalgebra. A \textit{first-order codifferential calculus (FOCC)} over $C$ is a $C$-bicomodule $\mathscr{W}_1$ together with a coderivation $\delta\colon \mathscr{W}_1\rightarrow C$, such that the map
        \[
            \mathscr{W}_1\rightarrow C\otimes C, \quad w\mapsto w_{(-1)}\otimes \delta(w_{(0)})
        \]
        is injective.
    \end{definition}

    As a first example in this exposition, we will consider the universal first-order codifferential calculus. To motivate it from a categorical point of view, note that $\mathrm{Coder}(-,C)$ gives rise to a contravariant functor from ${}^C\mathcal{M}^C$ to the category of vector spaces. On morphisms $\varphi\colon V\rightarrow W$, this functor acts as follows:
    \[
        \varphi^\ast\colon\mathrm{Coder}(W,C)\rightarrow \mathrm{Coder}(V,C), \quad \delta \mapsto \delta\circ \varphi.
    \]
    In fact, as shown by Doi \cite{Doi81}, the functor $\mathrm{Coder}(-,C)$ is representable, and the representing object is called the universal first-order codifferential calculus.

    \begin{proposition}
        \label{PropUniversalFOCC}
        Let $C$ be a coalgebra and let $\mathscr{W}_1^u := C\otimes C/\im(\Delta)$, with $C$-bicomodule structure given by
        \[
            [c\otimes d] \mapsto c_{(1)}\otimes [c_{(2)}\otimes d], \quad [c\otimes d]\mapsto [c\otimes d_{(1)}]\otimes d_{(2)}
        \]
        and let $\delta^u$ be the map
        \[
            \mathscr{W}_1^u \rightarrow C, \quad [c\otimes d]\mapsto \varepsilon(c)d-\varepsilon(d)c.
        \]
        Then the pair $(\mathscr{W}_1^u, \delta^u)$ is a first-order codifferential calculus, and the following universal property holds: For any $C$-bicomodule $W$ and any coderivation $\delta\colon W \rightarrow C$, there exists a unique $C$-bicolinear map $\varphi\colon W \rightarrow \mathscr{W}_1^u$, such that the diagram
        \[\begin{tikzcd}
	    W & {\mathscr{W}_1^u} \\
	    & C
	    \arrow["{\varphi}", dashed, from=1-1, to=1-2]
	    \arrow["\delta"', from=1-1, to=2-2]
	    \arrow["{\delta^u}", from=1-2, to=2-2]
        \end{tikzcd}\]
        commutes. In particular, the natural bijection
        \[
            {}^C\mathrm{Hom}^C(W, \mathscr{W}_1^u) \rightarrow \mathrm{Coder}(W, C), \quad \varphi\mapsto \delta^u\circ \varphi
        \]
        exhibits $\mathscr{W}_1^u$ as a representing object of $\mathrm{Coder}(-,C)$.
    \end{proposition}

    \begin{proof}
        Universality of the pair $(\mathscr{W}_1^u, \delta^u)$ follows from \cite[Proposition 13]{Doi81}. Thus, we only need to show that $[c\otimes d] \mapsto c_{(1)}\otimes \delta^u([c_{(2)}\otimes d])$ is injective. Observe that for all $c, d\in C$ we have
        \begin{align*}
            [c_{(1)}\otimes \delta^u([c_{(2)}\otimes d])] &= [\varepsilon(c_{(2)})c_{(1)}\otimes d - \varepsilon(d)c_{(1)}\otimes c_{(2)}]\\
            &= [c\otimes d].
        \end{align*}
        Thus, the map $[c\otimes d]\mapsto c_{(1)}\otimes \delta^u([c_{(2)}\otimes d])$ is a section of the projection $C\otimes C\rightarrow \mathscr{W}_1^u$. In particular, it is injective.
    \end{proof}

    If $\delta\colon W\rightarrow C$, is a coderivation, then the unique map $\varphi\colon W\rightarrow \mathscr{W}_1^u$ such that $\delta^u\circ \varphi = \delta$ is given by
    \begin{equation}\label{eq:CanMapToUniversalFOCC}
        \varphi(w) = [w_{(-1)}\otimes \delta(w_{(0)})]
    \end{equation}
    for all $w \in W$ (see for instance the proof of \cite[Proposition 13]{Doi81}).

    \begin{remark}
        \label{RemUniversalFOCCKhalkhali}
        The universal first-order codifferential calculus has an alternative description due to Khalkhali \cite{Kha97}: As the underlying $C$-bicomodule consider $C\otimes C^+$, where $C^+ = \ker(\varepsilon)$ with left- and right $C$-coactions
        \begin{gather*}
            {}_{C\otimes C^+}\Delta\colon C\otimes C^+ \rightarrow C\otimes C\otimes C^+, \quad c\otimes d \mapsto c_{(1)}\otimes c_{(2)}\otimes d\\
            \Delta_{C\otimes C^+}\colon C\otimes C^+\rightarrow C\otimes C^+\otimes C, \quad c\otimes d \mapsto c_{(1)}\otimes c_{(2)}\otimes d - c\otimes d_{(1)}\otimes d_{(2)}.
        \end{gather*}
        For the coderivation, we take
        \[
            \delta\colon C\otimes C^+\rightarrow C, \quad c\otimes d \mapsto \varepsilon(c)d.
        \]
    \end{remark}

    The universal first-order codifferential calculus gives us an alternative way to phrase the injectivity condition, which is the one used in \cite{BM26}.

    \begin{proposition}
        \label{PropBicomodQuotients}
        Let $\delta\colon W \rightarrow C$ be a coderivation, and let $\varphi\colon W \rightarrow \mathscr{W}_1^u$ be the unique $C$-bicolinear map as in Proposition \ref{PropUniversalFOCC}. Then $W/\ker(\varphi)$ together with the induced coderivation $\overline{\delta}\colon W/\ker(\varphi) \rightarrow C, [w]\mapsto \delta(w)$ is a first-order codifferential calculus. Moreover, the pair $(W,\delta)$ is a first-order codifferential calculus if and only if $\varphi$ is injective. 
    \end{proposition}

    \begin{proof}
        We have $\overline{\delta} = \delta^u\circ \overline{\varphi}$, where $\overline{\varphi}\colon W/\ker(\varphi)\rightarrow \mathscr{W}_1^u, [w] \mapsto \varphi(w)$ is the $C$-bicolinear map induced by $\varphi$. Thus, this map is a well-defined coderivation. Next observe that for any pair $(W, \delta)$ consisting of a $C$-bicomodule and a coderivation, the diagram
        \[\begin{tikzcd}
	    & {C\otimes C} \\
	    W & {\mathscr{W}_1^u}
	    \arrow[from=2-1, to=1-2]
	    \arrow["\varphi"', from=2-1, to=2-2]
	    \arrow[from=2-2, to=1-2]
        \end{tikzcd}\]
        commutes, where $\varphi$ is again the unique $C$-bicolinear map as in Proposition \ref{PropUniversalFOCC} and the vertical maps are as in Definition \ref{DefFOCC}. This shows at once that $W/\ker(\varphi)$ is a first-order codifferential calculus and that $(W,\delta)$ is a first-order codifferential calculus if and only $\varphi$ is injective.
    \end{proof}

    \begin{remark}
        As remarked in the introduction, codifferential calculi have not been thoroughly studied until recently. The work on this article started some time before the preprint \cite{BM26} appeared, after noticing that the degree 1 differential in the sequence of $U_q(\mathfrak{sl}_2)$-modules in the introduction of \cite{HK072} is a coderivation. This is why our approach and notation differ from the one used in \cite{BM26}.
    \end{remark}

    \subsection{Codifferentiable maps}\label{SecCodiffMap}

    In this section, we want to investigate how first-order codifferential calculi interact with morphisms of coalgebras. More generally, one can consider morphisms of coderivations \cite[Definition 1]{BM26}.

    \begin{definition}
        Let $C$ and $D$ be coalgebras, and let $(\mathscr{W}_1^C, \delta^C)$ and $(\mathscr{W}_1^D, \delta^D)$ be first-order codifferential calculi over $C$ and $D$ respectively. A morphism of coalgebras $\varphi\colon C\rightarrow D$ is \textit{codifferentiable} if there exists a $D$-bicolinear map $\varphi^\delta\colon \mathscr{W}_1^C \rightarrow \mathscr{W}_1^D$ such that the diagram
        \begin{equation}\label{eq:Codifferentiability}\begin{tikzcd}[column sep=large]
	    {\mathscr{W}_1^C} & {\mathscr{W}_1^D} \\
	    C & D
	    \arrow["{\varphi^\delta}", from=1-1, to=1-2]
	    \arrow["{\delta^C}"', from=1-1, to=2-1]
	    \arrow["{\delta^D}", from=1-2, to=2-2]
	    \arrow["\varphi"', from=2-1, to=2-2]
        \end{tikzcd}\end{equation}
        commutes. We refer to $\varphi^\delta$ as a \textit{codifferential} of $\varphi$
    \end{definition}

    If a codifferential exists, then it is necessarily unique.

    \begin{proposition}
        Let $\varphi\colon C\rightarrow D$ be a morphism of coalgebras, and let $(\mathscr{W}_1^C, \delta^C)$ and $(\mathscr{W}_1^D, \delta^D)$ be first-order codifferential calculi over $C$ and $D$ respectively. If $\varphi$ is codifferentiable, then there exists at most one codifferential of $\varphi$.
    \end{proposition}

    \begin{proof}
        Let $\varphi^\delta$ be a codifferential of $\varphi$. Using \eqref{eq:Codifferentiability} and the fact that $\varphi^\delta$ is $D$-bicolinear, we get
        \[
            \varphi^\delta(w)_{(-1)}\otimes \delta^D(\varphi^\delta(w)_{(0)}) = \varphi(w_{(-1)})\otimes \varphi(\delta^C(w_{(0)}))
        \]
        for all $w \in \mathscr{W}_1^C$. The expression on the left-hand side is simply the canonical embedding $\mathscr{W}_1^D\hookrightarrow D\otimes D$ (see Definition \ref{DefFOCC}) applied to $\varphi^\delta(w)$. Since the right-hand side does not involve $\varphi^\delta$, uniqueness follows.
    \end{proof}

    Note that for this argument to apply, the injectivity axiom for first-order codifferential calculi is essential. We also observe that any coalgebra map $\varphi\colon C\rightarrow D$ is codifferentiable if we equip $C$ and $D$ with their corresponding universal first-order codifferential calculus. The codifferential of $\varphi$ is in this case given by $[c\otimes d]\mapsto [\varphi(c)\otimes \varphi(d)]$.

    \begin{definition}
        We denote by $\mathbf{FOCC}$ the category consisting of objects given by coalgebras $C$ together with a first-order codifferential calculus and morphisms given by codifferential maps. Let $U\colon \mathbf{FOCC}\rightarrow \mathbf{Coalg}_k$ be the functor that takes a first-order codifferential calculus $(\mathscr{W}_1, C)$ to the coalgebra $C$, and acts as the identity on morphisms. For a fixed coalgebra $C$ we write $\mathbf{FOCC}_C$ for the fiber category $U^{-1}(\{C\})$, that is, the category of first-order codifferential calculi over $C$, where a morphism $\mathscr{W}_1 \rightarrow \tilde{\mathscr{W}}_1$ is given by the unique codifferential of the identity $\id_C$.
    \end{definition}

    \begin{remark}
        It follows by Proposition \ref{PropUniversalFOCC}, that for any coalgebra $C$, the category $\mathbf{FOCC}_C$ is equivalent poset category of sub-bicomodules of the universal first-order codifferential calculus $\mathscr{W}_1^u$ over $C$. This classification of first-order codifferential calculi is stricter than the one in \cite{BM26}. The reason is that the category $\mathbf{FOCC}_C$ excludes isomorphisms in $\mathbf{FOCC}$ between two first-order codifferential calculi over $C$, that are given by coalgebra automorphisms $C\rightarrow C$.
    \end{remark}

    \subsection{Higher-order codifferential calculi}\label{SecDGCoalg}

    In this section we extend the notion of a first-order codifferential calculi to higher orders. First, recall the definition of a differential graded coalgebra.    

    \begin{definition}
        A \textit{differential graded coalgebra (dg coalgebra)} is a (non-negatively) graded coalgebra (see Example \ref{ExplGrcoalg}) $D_\bullet$ together with a \textit{differential} $\delta$, that is a linear map of degree -1 such that $\delta^2 = 0$ and
        \begin{equation}\label{eq:dgcoalg}
            \Delta(\delta(c)) = \delta(c_{(1)})\otimes c_{(2)} + (-1)^{|c_{(1)}|}c_{(1)}\otimes \delta(c_{(2)})
        \end{equation}
        for all homogeneous elements $c \in D_\bullet$.
    \end{definition}

    One can equivalently describe a dg coalgebra as a comonoid in the category of (non-negatively graded) chain complexes. With this description \eqref{eq:dgcoalg} corresponds to the condition $\Delta(\delta(d)) = \delta^{D_\bullet\otimes D_\bullet}(\Delta(d))$ where $\delta^{D_\bullet\otimes D_\bullet}$ is the differential of the tensor product of chain complexes $D_\bullet\otimes D_\bullet$.
    
    Let us now discuss how to verify \eqref{eq:dgcoalg} in practice, using this equivalent comonoidal description. Fix $d\in D_n$ for some $n \ge 1$, then both $\Delta(\delta(d))$ and $\delta^{D_\bullet\otimes D_\bullet}(\Delta(d))$ lie in $\bigoplus_{i=0}^{n-1}D_i\otimes D_{n-1-i}$. Thus, to show that they are equal, one has to compare their projections onto each $D_i\otimes D_{n-1-i}$ for $0\le i\le n-1$. Write
    \[
        \Delta(\delta(d)) = \sum_{i=0}^{n-1} \delta(d)_{(1)}^i\otimes \delta(d)_{(2)}^{n-1-i}, \quad \Delta(d) = \sum_{j=0}^n d_{(1)}^j\otimes d_{(2)}^{n-j}
    \]
    such that $\delta(d)_{(1)}^i\otimes \delta(d)_{(2)}^{n-1-i} \in D_i\otimes D_{n-1-i}$ for all $0\leq i \leq n-1$ and $d_{(1)}^j\otimes d_{(2)}^{n-j} \in D_j\otimes D_{n-j}$ for all $0\leq j \leq n$. In order to show that \eqref{eq:dgcoalg} holds, it is thus sufficient and necessary to show that
    \begin{equation}\label{eq:dgcoalgII}
        \delta(d)_{(1)}^i\otimes \delta(d)_{(2)}^{n-1-i} = \delta(d_{(1)}^{i+1}) \otimes d_{(2)}^{n-(i+1)} + (-1)^id_{(1)}^i\otimes \delta(d_{(2)}^{n-i})
    \end{equation}
    for all $d \in D_n$ with $n\ge 1$ and $0\le i \le n-1$.

    Write $C := D_0$ as noted in Example \ref{ExplGrcoalg}, the latter is a subcoalgebra of $D_\bullet$ and every homogeneous component $D_i$ for $i\ge 0$ is a $C$-bicomodule. If we explicitly write down the compatibility condition \eqref{eq:dgcoalgII} for an element $d \in D_1$, we get
    \[
        \Delta(\delta_1(d)) = \delta_1(d_{(0)})\otimes d_{(1)} + d_{(-1)}\otimes \delta_1(d_{(0)}).
    \]
    That is, $\delta_1$ is a coderivation, and thus by Proposition \ref{PropBicomodQuotients}, the quotient $D_1/\ker(\varphi)$, where $\varphi\colon D_1 \rightarrow \mathscr{W}_1^u$ is the unique $C$-bicolinear map from Proposition \ref{PropUniversalFOCC}, is a first-order codifferential calculus over $D_0$, which we denote by $\underline{\mathscr{W}_1}(D_\bullet)$. This construction behaves well with morphisms of dg coalgebras.
    
    \begin{proposition}
        \label{PropContractionII}
        Let $f \colon D_\bullet \rightarrow \tilde{D}_\bullet$ be a morphism of dg coalgebras. Then the restriction $f_0\colon D_0 \rightarrow \tilde{D}_0$ is codifferentiable with respect to the codifferential calculi $\underline{\mathscr{W}_1}(D_\bullet)$ and $\underline{\mathscr{W}_1}(\tilde{D}_\bullet)$. In particular, we get a functor
        \[
            \underline{\mathscr{W}_1}(-)\colon \dgCoalg \rightarrow \Focc.
        \]
    \end{proposition}

    \begin{proof}
        Since morphisms of dg coalgebras commute with the respective differentials, the assertion follows directly by checking that $f_1 \colon D_1 \rightarrow \tilde{D}_1$ is $\tilde{D}_0$-bicolinear, and that it descends to a map $\underline{\mathscr{W}_1}(D_\bullet) \rightarrow \underline{\mathscr{W}_1}(\tilde{D}_\bullet)$.
    \end{proof}

    A (higher-order) codifferential calculus is a dg coalgebra $D_\bullet$, which embeds into the cotensor coalgebra $T_{D_0}^c(\underline{\mathscr{W}_1}(D_\bullet))$. In more detail:

    \begin{definition}
        \label{DefCodifferentialCalculus}
        A dg coalgebra $D_\bullet$ is called \textit{codifferential calculus} if the map
        \[
            D_\bullet \rightarrow T_{D_0}^c(\underline{\mathscr{W}_1}(D_\bullet))
        \]
        induced by the projection $D_\bullet\overset{\pi_1}{\rightarrow} D_1 \rightarrow \underline{\mathscr{W}_1}(D_\bullet)$ (see Example \ref{ExplGrcoalg}) is injective.
    \end{definition}

    \subsection{The maximal prolongation}\label{SecMaximalProlongation}

    For this section we fix a coalgebra $C$ and a first-order codifferential calculus $\mathscr{W}_1$ over $C$. In addition, we write $\varphi$ for the canonical map $\mathscr{W}_1\rightarrow \mathscr{W}_1^u$ from Proposition \ref{PropUniversalFOCC}. Note that since $\mathscr{W}_1$ is a first-order codifferential calculus, the map $\varphi$ is injective by Proposition \ref{PropBicomodQuotients}. The goal is to universally construct a codifferential calculus $\mathscr{W}_\bullet^{\mathrm{max}}$ such that $\underline{\mathscr{W}_1}(\mathscr{W}_\bullet^{\mathrm{max}}) = \mathscr{W}_1$. The key step for this construction is the following definition.
    
    \begin{definition}
        \label{DefDeg2Coforms}
        Consider the following map
        \[
            \check{\delta}_2\colon \mathscr{W}_1\cotimes{C} \mathscr{W}_1\rightarrow \mathscr{W}_1^u, \quad w_1\cotimes{C} w_2\mapsto [\delta(w_1)\otimes \delta(w_2)].
        \]
        We let $\mathscr{W}_2^\mathrm{max}$ be the largest sub $C$-bicomodule of $\mathscr{W}_1\cotimes{C} \mathscr{W}_1$ such that $\check{\delta}_2(\mathscr{W}_2^\mathrm{max})\subseteq \varphi(\mathscr{W}_1)$. Moreover, we let $\delta_2\colon \mathscr{W}_2^\mathrm{max}\rightarrow \mathscr{W}_1$ be the unique map such that the diagram
        \[\begin{tikzcd}\label{eq:DefDeg2Coforms}
	    {\mathscr{W}_1\cotimes{C}\mathscr{W}_1} & {\mathscr{W}_1^u} \\
	    {\mathscr{W}_2^\mathrm{max}} & {\mathscr{W}_1}
	    \arrow["{\check{\delta}_2}", from=1-1, to=1-2]
	    \arrow[hook, from=2-1, to=1-1]
	    \arrow["{\delta_2}", dashed, from=2-1, to=2-2]
	    \arrow["{\varphi}"', hook, from=2-2, to=1-2]
        \end{tikzcd}\]
        commutes.
    \end{definition}

    From the definition above it is not obvious which elements of $\mathscr{W}_1\cotimes{C}\mathscr{W}_1$ belong to $\mathscr{W}_2^\mathrm{max}$. However, we have the following sufficient condition.

    \begin{lemma}
        \label{LemEltsOfW2}
        If $w_1\cotimes{C}w_2 \in \mathscr{W}_1\cotimes{C}\mathscr{W}_1$ such that $\check{\delta}(w_1\cotimes{C}w_2) = 0$, then $w_1\cotimes{C}w_2 \in \mathscr{W}_2^\mathrm{max}$.
    \end{lemma}

    \begin{proof}
        We claim that $(\id\otimes \check{\delta}\otimes \id)(w_{1(-1)}\otimes w_{1(0)}\cotimes{C} w_{2(0)}\otimes w_{2(1)}) \in C\otimes \varphi(\mathscr{W}_1)\otimes C$. Indeed, using the coderivation property of $\delta$ gives
        \begin{align*}
            &\phantom{==} w_{1(-1)}\otimes[\delta(w_{1(0)})\otimes \delta(w_{2(0)})]\otimes w_{2(1)}\\
            &= \delta(w_{1})_{(1)}\otimes[\delta(w_1)_{(2)}\otimes \delta(w_{2(0)})]\otimes w_{2(1)} - \delta(w_{1(0)})\otimes [w_{1(1)}\otimes \delta(w_{2(0)})]\otimes w_{2(1)}\\
            &= \delta(w_1)_{(1)}\otimes[\delta(w_1)_{(2)}\otimes \delta(w_2)_{(1)}]\otimes \delta(w_2)_{(2)} - \delta(w_1)_{(1)}\otimes[\delta(w_1)_{(2)}\otimes w_{2(-1)}]\otimes \delta(w_{2(0)})\\
            &\phantom{==} - \delta(w_{1(0)})\otimes [w_{1(1)}\otimes \delta(w_2)_{(1)}] \otimes \delta(w_2)_{(2)} + \delta(w_{1(0)})\otimes [w_{1(1)}\otimes w_{2(-1)}]\otimes \delta(w_{2(0)}).
        \end{align*}
        Here, the first summand is zero, since $[\delta(w_1)\otimes \delta(w_2)] = 0$ by assumption, the second summand is equal to
        \begin{align*}
            -({}_{\mathscr{W}_1^u}\Delta\otimes \id)([\delta(w_1)\otimes w_{2(-1)}]\otimes \delta(w_{2(0)})) &= -({}_{\mathscr{W}_1^u}\Delta\otimes \id)([\delta(w_{1(0)})\otimes w_{1(1)}]\otimes \delta(w_2))\\
            &= ({}_{\mathscr{W}_1^u}\Delta\otimes \id)([w_{1(-1)}\otimes \delta(w_{1(0)})]\otimes \delta(w_2))\\
            &= ({}_{\mathscr{W}_1^u}\Delta\otimes \id)(\varphi(w_1)\otimes \delta(w_2))
        \end{align*}
        where we used that $w_1\otimes w_2 \in \mathscr{W}_1\cotimes{C}\mathscr{W}_1$, and the last step uses \eqref{eq:CanMapToUniversalFOCC}. Similarly, the third summand equals to
        \begin{align*}
            -(\id\otimes \Delta_{\mathscr{W}_1^u})(\delta(w_{1(0)})\otimes [w_{1(1)}\otimes \delta(w_2)]) = -(\id\otimes \Delta_{\mathscr{W}_1^u})(\delta(w_1)\otimes \varphi(w_2))
        \end{align*}
        and finally the last summand vanishes since
        \begin{align*}
            w_{1(0)}\otimes w_{1(1)}\otimes w_{2(-1)}\otimes w_{2(0)} &= w_{1(0)}\otimes w_{1(1)}\otimes {}_{\mathscr{W}_1}\Delta(w_2)\\
            &= w_1\otimes w_{2(-1)}\otimes {}_{\mathscr{W}_1}\Delta(w_{2(0)})\\
            &= w_1 \otimes w_{2(-1)(1)}\otimes w_{2(-1)(2)}\otimes w_{2(0)}.
        \end{align*}
        Thus, we conclude that $w_1\cotimes{C}w_2 \in \mathscr{W}_2^\mathrm{max}$ by Lemma \ref{LemSubcomodGenByElt}.
    \end{proof}

    The following assertion shows that $\delta_2$ has the properties of a differential of a dg coalgebra in degree 2.

    \begin{lemma}
        \label{LemDeg2Differential}
        The map $\delta_2$ from Definition \ref{DefDeg2Coforms} satisfies $\delta\circ \delta_2 = 0$. In addition, the identities
        \begin{align}
            {}_{\mathscr{W}_1}\Delta(\delta_2(w_1\cotimes{C}w_2)) &= w_{1(-1)}\otimes\delta_2(w_{1(0)}\cotimes{C}w_2) + \delta(w_1)\otimes w_2\\
            \Delta_{\mathscr{W}_1}(\delta_2(w_1\cotimes{C}w_2)) &= \delta_2(w_1\cotimes{C}w_{2(0)})\otimes w_{2(1)} - w_1\otimes \delta(w_2)
        \end{align}
        hold for every $w_1\cotimes{C}w_2 \in \mathscr{W}_2^\mathrm{max}$.
    \end{lemma}

    \begin{proof}
        Let $w_1\cotimes{C}w_2 \in \mathscr{W}_2^\mathrm{max}$, then
        \[
            \delta(\delta_2(w_1\cotimes{C}w_2)) = \delta^u(\varphi(\delta_2(w_1\cotimes{C}w_2))) = \delta^u([\delta(w_1)\otimes\delta(w_2)]) = 0
        \]
        where the last equality follows from the definition of $\delta^u$ and \eqref{eq:CounitOnCoderivation}. This shows the first assertion. For the second assertion we compute
        \begin{align*}
            (\id\otimes\varphi)({}_{\mathscr{W}_1}\Delta(\delta_2(w_1\cotimes{C}w_2))) &= {}_{\mathscr{W}_1^u}\Delta(\varphi(\delta_2(w_1\cotimes{C}w_2))) = \delta(w_1)_{(1)}\otimes [\delta(w_1)_{(2)}\otimes\delta(w_2)]\\
            &= w_{1(-1)}\otimes[\delta(w_{1(0)})\otimes\delta(w_2)] + \delta(w_{1(0)})\otimes [w_{1(1)}\otimes \delta(w_2)]\\
            &= w_{1(-1)}\otimes [\delta(w_{1(0)})\otimes\delta(w_2)] + \delta(w_{1}) \otimes [w_{2(-1)}\otimes \delta(w_{2(0)})]\\
            &= (\id\otimes \varphi)(w_{1(-1)}\otimes \delta_2(w_{1(0)}\cotimes{C} w_2) + \delta(w_1)\otimes w_2)
        \end{align*}
        where we used \eqref{eq:CanMapToUniversalFOCC} in the last step. Thus, the first identity holds by injectivity of $(\id\otimes \varphi)$. Similarly, we compute
        \begin{align*}
            (\varphi\otimes \id)(\Delta_{\mathscr{W}_1}(\delta_2(w_1\cotimes{C}w_2))) &= [\delta(w_1)\otimes \delta(w_{2(0)})]\otimes w_{2(1)} + [\delta(w_{1(0)})\otimes w_{1(1)}]\otimes \delta(w_2)\\
            &= [\delta(w_1)\otimes \delta(w_{2(0)})]\otimes w_{2(1)}\\
            &\phantom{==} + [\delta(w_1)_{(1)}\otimes\delta(w_1)_{(2)}-w_{1(-1)}\otimes \delta(w_{1(0)})]\otimes \delta(w_2)\\
            &= (\varphi\otimes \id)(\delta_2(w_1\cotimes{C}w_{2(0)})\otimes w_{2(1)} - w_1\otimes \delta(w_2)).
        \end{align*}
        Hence, also the second identity holds.
    \end{proof}

    We now extend the above to a full dg coalgebra. For $n \ge 2$ we define the degree $n$ component as
    \begin{equation}\label{eq:MaximalProlongationDegn}
        \mathscr{W}_n^\mathrm{max} := \bigcap_{i+j+2 = n} \mathscr{W}_1^{\cotimes{C} i}\cotimes{C}\mathscr{W}_2^\mathrm{max}\cotimes{C}\mathscr{W}_1^{\cotimes{C} j}.
    \end{equation}
    By setting $\mathscr{W}_0^\mathrm{max} := C$ it follows that $\mathscr{W}_\bullet^\mathrm{max}$ is a subcoalgebra of $T_C^c(\mathscr{W}_1)$. We define a differential on $\mathscr{W}_\bullet^\mathrm{max}$ by setting $\delta_1 := \delta$ and letting
    \begin{equation}\label{eq:MaximalProlongationDifferential}
        \delta_n(w_1\cotimes{C}\dots \cotimes{C} w_n) := \sum_{i=1}^{n-1}(-1)^{i-1}w_1\otimes \dots \otimes \delta_2(w_i\cotimes{C}w_{i+1})\otimes \dots \otimes w_n
    \end{equation}
    where $w_1\cotimes{C}\dots\cotimes{C}w_n\in \mathscr{W}_n^\mathrm{max}$ for $n\ge 2$.

    \begin{proposition}
        \label{PropMaximalProlongation}
        The pair $(\mathscr{W}_\bullet^\mathrm{max}, \delta_\bullet)$ is codifferential calculus with $\underline{\mathscr{W}_1}(\mathscr{W}_\bullet^\mathrm{max}) = \mathscr{W}_1$.
    \end{proposition}

    \begin{proof}
        If we can show that $(\mathscr{W}_\bullet^\mathrm{max},\delta_\bullet)$ is a dg coalgebra, then it is clear by construction that $\underline{\mathscr{W}_1}(\mathscr{W}_\bullet^\mathrm{max}) = \mathscr{W}_1$ and that $\mathscr{W}_\bullet^\mathrm{max}\rightarrow T_C^c(\mathscr{W}_1)$ is an embedding. Thus, it is sufficient to prove the first assertion. Firstly, we show that $\delta_\bullet$ is well-defined. For $n = 2$ this follows by definition. If $n\ge 3$ we can write
        \begin{align*}
            \delta_n(w_1\cotimes{C}\dots \cotimes{C}w_n) &= \delta_i(w_1\cotimes{C}\dots\cotimes{C}w_i)\otimes (w_{i+1}\otimes w_{i+2}) \otimes \dots \otimes w_n\\
            &\phantom{==} + (-1)^{i-1}w_1\otimes \dots \otimes \delta_3(w_i\cotimes{C}w_{i+1}\cotimes{C}w_{i+2})\otimes \dots \otimes w_n\\
            &\phantom{==} + (-1)^{i+1}w_1\otimes\dots\otimes (w_i\otimes w_{i+1})\otimes \delta_{n-i-1}(w_{i+2}\cotimes{C}\dots\cotimes{C}w_n)
        \end{align*}
        for all $1\le i \le n-2$ and $w_1\cotimes{C}\dots\cotimes{C} w_n \in \mathscr{W}_n^\mathrm{max}$, where we emphasized the $i,i+1$ component. Thus, well-defined ness reduces to the case $n = 3$.
        Firstly, for any element $w_1\cotimes{C}w_2\cotimes{C}w_3 \in \mathscr{W}_3^\mathrm{max}$ we have $\delta_3(w_1\cotimes{C}w_2\cotimes{C}w_3) \in \mathscr{W}_1\cotimes{C}\mathscr{W}_1$ by applying Lemma \ref{LemDeg2Differential}:
        \begin{align*}
            &\phantom{==} (\Delta_{\mathscr{W}_1}\otimes \id)(\delta_2(w_1\cotimes{C} w_2)\otimes w_3 - w_1 \otimes \delta_2(w_2\cotimes{C} w_3))\\
            &= \delta_2(w_1\cotimes{C}w_{2(0)})\otimes w_{2(1)}\otimes w_3 - w_1 \otimes \delta(w_2)\otimes w_3 - w_{1(0)}\otimes w_{1(1)}\otimes \delta_2(w_2\cotimes{C}w_3)\\
            &= \delta_2(w_1\cotimes{C} w_2)\otimes w_{3(-1)}\otimes w_{3(0)} - w_1\otimes \delta(w_2)\otimes w_3 - w_1 \otimes w_{2(-1)}\otimes \delta_2(w_{2(0)}\cotimes{C}w_3)\\
            &= (\id \otimes {}_{\mathscr{W}_1}\Delta)(\delta_2(w_1\cotimes{C}w_2)\otimes w_3 - w_1\otimes \delta_2(w_2\cotimes{C}w_3)).
        \end{align*}
        Next, we also have $\delta_3(w_1\cotimes{C}w_2\cotimes{C}w_3) \in \mathscr{W}_2^\mathrm{max}$ as a consequence of Lemma \ref{LemEltsOfW2}, since
        \begin{equation}\label{eq:Deg3Differential}
            \check{\delta}(\delta_3(w_1\cotimes{C}w_2\cotimes{C}w_3)) = [\delta(\delta_2(w_1\cotimes{C} w_2)\otimes \delta(w_3))] - [\delta(w_1)\otimes \delta(\delta_2(w_2\cotimes{C}w_3))] = 0
        \end{equation}
        using the first assertion of Lemma \ref{LemDeg2Differential}. We proceed by showing that $\delta_{n-1}\circ \delta_n = 0$ for all $n\ge 2$. For $n = 2$ this follows from Lemma \ref{LemDeg2Differential} and for $n = 3$ this follows from \eqref{eq:Deg3Differential}. For $n \ge 4$ we get
        \begin{align*}
            &\phantom{==} \delta_{n-1}(\delta_n(w_1\cotimes{C}\dots \cotimes{C}w_n))\\
            &= \sum_{1\leq j<i-1\leq n}(-1)^{i+j}w_1\otimes \dots \otimes \delta_2(w_j\cotimes{C}w_{j+1})\otimes \dots \otimes \delta_2(w_i\cotimes{C}w_{i+1})\otimes \dots \otimes w_n\\
            &\phantom{==} + \sum_{1\le i < j\le n-2}(-1)^{i+j}w_1\otimes \dots \otimes \delta_2(w_i\cotimes{C}w_{i+1})\otimes \dots \otimes \delta_2(w_{j+1}\cotimes{C} w_{j+2}) \otimes \dots \otimes w_n\\
            &\phantom{==} + \sum_{i=1}^{n-2}w_1 \otimes \dots \otimes\delta_2(\delta_3(w_i\cotimes{C}w_{i+1}\cotimes{C}w_{i+2}))\otimes \dots \otimes w_n.
        \end{align*}
        The first two summands cancel after applying an index shift and the last summand is zero by the case $n = 3$.
        
        We are left to show that $\delta_\bullet$ satisfies \eqref{eq:dgcoalg}. For this fix $\underline{w} := w_1\cotimes{C}\dots\cotimes{C} w_n \in \mathscr{W}_n^\mathrm{max}$. As we noted previously, it is sufficient to show \eqref{eq:dgcoalgII} for all $0\le i\le n-1$. Consider first the case $i\neq 0,n-1$. In this case left-hand side of \eqref{eq:dgcoalgII} evaluates to
        \begin{gather*}
            \delta_{i+1}(w_1\cotimes{C}\dots\cotimes{C}w_{i+1})\otimes (w_{i+2}\cotimes{C} \dots \cotimes{C} w_n)\\ + (-1)^k(w_1\cotimes{C}\dots\cotimes{C}w_i)\otimes \delta_{n-i}(w_{i+1}\cotimes{C}\dots\cotimes{C}w_n).
        \end{gather*}
        It is also immediate that the right-hand side gives the same, since the coproduct is given by deconcatenation. Now let $i = 0$ the by using Lemma \ref{LemDeg2Differential} the left-hand side of \eqref{eq:dgcoalgII} becomes
        \begin{align*}
            &\delta_2(w_1\cotimes{C}w_2)_{(-1)}\otimes (\delta_2(w_1\cotimes{C}w_2)_{(0)}\otimes w_3\otimes \dots\otimes w_n)\\
            &\phantom{==} - w_{1(-1)}\otimes (w_{1(0)}\otimes \delta_{n-1}(w_2\cotimes{C}\dots \cotimes{C} w_n))\\
            &= w_{1(-1)}\otimes (\delta_2(w_{1(0)}\cotimes{C}w_2)\otimes w_3\otimes \dots \otimes w_n) + \delta(w_1)\otimes (w_2\cotimes{C}\dots\cotimes{C}w_n)\\
            &\phantom{==} - w_{1(-1)}\otimes (w_{1(0)}\otimes \delta_{n-1}(w_2\cotimes{C}\dots\cotimes{C}w_n))\\
            &= \delta(w_1)\otimes (w_2\cotimes{C}\dots\cotimes{C}w_n) + w_{1(-1)}\otimes\delta_n(w_{1(0)}\cotimes{C}w_2\cotimes{C}\dots\cotimes{C}w_n).
        \end{align*}
        Again, by definition of the coproduct on $T_C^c(\mathscr{W}_1)$ it is immediate that the right-hand side of \eqref{eq:dgcoalgII} is the same in this situation. The case $i = n-1$ is analogous.
    \end{proof}
    
    \begin{definition}
        \label{DefMaxProlongation}
        Let $\mathscr{W}_1$ be a first-order codifferential calculus. We call $(\mathscr{W}_\bullet^\mathrm{max}, \delta_\bullet)$ the \textit{maximal prolongation} of $\mathscr{W}_1$.
    \end{definition}

    This name is justified by the result below.

    \begin{theorem}
        \label{ThmMaxProlongationUniversalProperty}
        Let $D_\bullet$ be a dg coalgebra, and let $\mathscr{W}_1$ be a first-order codifferential calculus over $C$. Write $D := D_0$, such that $\underline{\mathscr{W}_1}(D_\bullet)$ is a first-order codifferential calculus over $D$. Then any codifferentiable map $f\colon D\rightarrow C$ (with respect to $\underline{\mathscr{W}_1}(D_\bullet)$ and $\mathscr{W}_1$) extends uniquely to a morphism of dg coalgebras $f^\flat\colon D_\bullet \rightarrow \mathscr{W}_\bullet^\mathrm{max}$. In particular, the maximal prolongation determines a functor $\mathbf{FOCC} \rightarrow \mathbf{dgCoalg}$ which is right-adjoint to $\underline{\mathscr{W}_1}(-)$.
    \end{theorem}

    \begin{proof}
        Let $f\colon D \rightarrow C$ be codifferentiable.
        We start by showing uniqueness. Let $f^\flat\colon D_\bullet \rightarrow \mathscr{W}_\bullet^\mathrm{max}$ be a morphism of dg coalgebras which extends $f$. By the same logic as in the proof of Proposition \ref{PropContractionII}, the degree 1 component $f_1 \colon D_1 \rightarrow \mathscr{W}_1$ descends to a $C$-bicolinear map $\underline{\mathscr{W}_1}(D_\bullet)\rightarrow \mathscr{W}_1$, and this map is precisely the codifferential $f^\delta$ of $f$. Next view $f^\flat$ as a morphism of coalgebras $D_\bullet \rightarrow T_C^c(\mathscr{W}_1)$ --- which we can do since $\mathscr{W}_\bullet^\mathrm{max}$ is a subcoalgebra of $T_C^c(\mathscr{W}_1)$ --- and observe that the diagrams
        \begin{equation}\label{eq:ExtensionCodifferentiableMap}
        \begin{tikzcd}
	    {T_C^c(\mathscr{W}_1)} & {D_\bullet} \\
	    C & D
	    \arrow[two heads, from=1-1, to=2-1]
	    \arrow["{f^\flat}"', from=1-2, to=1-1]
	    \arrow["{\pi_0}", two heads, from=1-2, to=2-2]
	    \arrow["f"', from=2-2, to=2-1]
        \end{tikzcd}\quad
        \begin{tikzcd}
	    {T_C^c(\mathscr{W}_1)}\\
	    {\mathscr{W}_1} & {D_\bullet}
	    \arrow[two heads, from=1-1, to=2-1]
	    \arrow["{f^\flat}"', from=2-2, to=1-1]
	    \arrow["h"', from=2-2, to=2-1]
        \end{tikzcd}
        \end{equation}
        commute, where $h$ is given by the projection $D_\bullet \rightarrow \underline{\mathscr{W}_1}(D_\bullet)$ followed by $f^\delta$. As in Example \ref{ExplGrcoalg}, the coradical of $D_\bullet$ gets mapped to zero by $h$, thus it follows that $f^\flat$ is unique by \ref{PropCotensorCoalgUniversal}. For existence, note that the same result also gives existence of a morphism of coalgebras $f^\flat\colon D_\bullet\rightarrow T_C^c(\mathscr{W}_1)$ such that the diagrams in \eqref{eq:ExtensionCodifferentiableMap} commute. Thus, we only need to show that $f^\flat$ restricts to a morphism of dg coalgebras to $\mathscr{W}_\bullet^\mathrm{max}$. Recall that $f^\flat$ is given by the formula
        \begin{equation}\label{eq:Deffflat}
            f^\flat(d) = f(\pi_0(d)) + \sum_{n\ge 0}h(d_{(1)})\cotimes{C} \dots \cotimes{C} h(d_{(n)}), \quad d\in D_\bullet.
        \end{equation}
        It follows that $f^\flat(d) \in \mathscr{W}_\bullet^\mathrm{max}$ for every $d \in D_n$ if $n = 0,1$. Also note that by definition of $h$ we have $\delta(h(d)) = f(\delta_1(d))$ for all $d \in D_1$, and that we can replace the coproduct in \eqref{eq:Deffflat} by the reduced coproduct since $h(d) = 0$ for all $d\in D$. We first show that $f^\flat(D_2)\subseteq \mathscr{W}_2^\mathrm{max}$. Since $f^\flat$ is $C$-bicolinear it suffices to show that for all $d\in D_2$ we have $\check{\delta}(f^\flat(d)) = \check{\delta}(h(d_{(1)})\cotimes{C}h(d_{(2)})) = [\delta(h(d_{(1)}))\otimes \delta(h(d_{(2)}))] = [w_{(-1)}\otimes \delta(w_{(0)})]$ for some $w \in \mathscr{W}_1$. We claim that the latter equality holds for $w = h(\delta_2(d))$. Indeed, by combining \eqref{eq:CoprodOfRedElt} and \eqref{eq:dgcoalgII} we have
        \[
            \delta_2(d)_{(-1)}\otimes \delta_2(d)_{(0)} = d_{(-1)}\otimes \delta_2(d_{(0)}) + \delta_1(d_{[1]})\otimes d_{[2]}.
        \]
        Because $h$ is left $C$-colinear we obtain
        \begin{align*}
            w_{(-1)}\otimes w_{(0)} &= h(\delta_2(d))_{(-1)}\otimes h(\delta_2(d)_{(0)}) = f(\delta_2(d)_{(-1)})\otimes h(\delta_2(d)_{(0)})\\
            &= f(d_{(-1)})\otimes h(\delta_2(d_{(0)})) + f(\delta_1(d_{[1]}))\otimes h(d_{[2]})\\
            &= f(d_{(-1)})\otimes h(\delta_2(d_{(0)})) + \delta(h(d_{[1]}))\otimes h(d_{[2]})
        \end{align*}
        and consequently
        \begin{align*}
            w_{(-1)}\otimes \delta(w_{(0)}) &= f(d_{(-1)})\otimes \delta(h(\delta_2(d_{(0)}))) + \delta(h(d_{[1]}))\otimes \delta(h(d_{[2]}))\\
            &= f(d_{(-1)})\otimes f(\delta_1(\delta_2(d_{(0)}))) + \delta(h(d_{[1]}))\otimes \delta(h(d_{[2]}))\\
            &= \delta(h(d_{[1]}))\otimes\delta(h(d_{[2]})) = \delta(h(d_{(1)}))\otimes \delta(h(d_{(2)})).
        \end{align*}
        It follows that $f^\flat(d) \in \mathscr{W}_2^\mathrm{max}$. Moreover, for the reason that $h(d) = 0$ if $d \in D_n$ with $n\neq 1$, it also holds that $h(d_{(1)})\cotimes{C} h(d_{(2)})\in \mathscr{W}_2^\mathrm{max}$ for any $d\in D_\bullet$, and thus $f^\flat(d) \in \mathscr{W}_n^\mathrm{max}$ for any $d \in D_\bullet$, since we can write
        \[
            h(d_{(1)})\cotimes{C} \dots \cotimes{C} h(d_{(n)}) = h(d_{(1)})\cotimes{C} \dots \cotimes{C}h(d_{(i)(1)})\cotimes{C} h(d_{(i)(2)})\cotimes{C} \dots \cotimes{C} h(d_{(i-1)})
        \]
        for any $n\ge 2$ and $1\le i\le n-1$. We are left with showing that $f^\flat(\delta_n(d)) = \delta_n(f^\flat(d))$ for all $n\ge 1$ and $d\in D_n$. For $n = 1$ this is equivalent to $f(\delta_1(d)) = \delta(h(d))$, which holds as already remarked. For $n = 2$ we have $\delta_2(h(d_{(1)})\cotimes{C}h(d_{(2)})) = w = h(\delta_2(d))$ in view of \ref{eq:DefDeg2Coforms}. For $n\ge 3$, note that the differential of $D_\bullet$ is also compatible with the reduced coproduct in the sense that
        \[
            \overline{\Delta}(\delta(d)) = \delta(d_{[1]})\otimes d_{[2]} + (-1)^{|d_{[1]}|}d_{[1]}\otimes \delta(d_{[2]}).
        \]
        By induction on $n$ it follows that
        \[
            \overline{\Delta}^{(n-2)}(\delta_n(d)) = \sum_{i=1}^{n-1}(-1)^{i-1}d_{[1]}\otimes \dots \otimes \delta(d_{[i]})\otimes \dots \otimes d_{[n-1]}.
        \]
        Therefore, we can write
        \begin{align*}
            f^\flat(\delta_n(d)) &= h(\delta_n(d)_{(1)})\cotimes{C} \dots \cotimes{C} h(\delta_n(d)_{(n-1)})\\
            &= \sum_{i=1}^{n-1}(-1)^{i-1}h(d_{(1)})\otimes \dots \otimes h(\delta(d_{(i)}))\otimes \dots \otimes h(d_{(n-1)})\\
            &= \sum_{i=1}^{n-1}(-1)^{i-1}h(d_{(1)})\otimes \dots \otimes\delta_2(h(d_{(i)})\cotimes{C} h(d_{(i+1)}))\otimes \dots \otimes h(d_{(n)})\\
            &= \delta_n(f^\flat(d)).
        \end{align*}
    \end{proof}

    The case where $D_\bullet$ is itself a codifferential calculus with $\underline{\mathscr{W}_1}(D_\bullet) = \mathscr{W}_1$ is of most interest to us.

    \begin{corollary}
        \label{CorMaximalProlongationMaximal}
        Let $\mathscr{W}_1$ be a first-order codifferential calculus over $C$. If $D_\bullet$ is a codifferential calculus such that $\underline{\mathscr{W}_1}(D_\bullet) = \mathscr{W}_1$ (here this also means that $D_0 = C$) then the unique morphism of dg coalgebras $D_\bullet \rightarrow \mathscr{W}_\bullet^\mathrm{max}$ extending $\id_C\colon D_0 = C \rightarrow C$ is an embedding.
    \end{corollary}

    \begin{proof}
        Since $D_\bullet$ is a codifferential calculus we know that $D_\bullet\rightarrow T_C^c(\mathscr{W}_1)$ is an embedding. However, we know from the proof of Theorem \ref{ThmMaxProlongationUniversalProperty} that this map restricts to $\mathscr{W}_\bullet^\mathrm{max}$, and that it is the unique extension of $\id_C$ to a morphism of dg coalgebras $D_\bullet \rightarrow \mathscr{W}_\bullet^\mathrm{max}$.
    \end{proof}

    \section{Quantum homogeneous spaces}\label{SecQHS}

    In the present section, we will make use of notation for relative Hopf modules. Let $A$ be a bialgebra, and recall that a left $A$-comodule algebra $B$ is a monoid in the category of left $A$-comodules, which amounts to the assertion that multiplication of $B$ is $A$-colinear. We denote by ${}_B^A\mathcal{M}$ the category of left $B$-modules and left $A$-comodules $\mathcal{F}$ such that
    \[
        {}_A\Delta(b\triangleright f) = b_{(-1)}f_{(-1)}\otimes b_{(0)}\triangleright f_{(0)}, \quad f\in \mathcal{F}, b\in B
    \]
    with morphisms given by left $B$-linear and left $A$-colinear maps. We write ${}^A\mathcal{M}_B$ and ${}_B^A\mathcal{M}_B$ for analogously defined categories, where in the last category, the objects are also assumed to be $B$-bimodules. Dually, for a left $A$-module coalgebra $C$, we can define categories ${}_A^C\mathcal{M}, {}_A\mathcal{M}^C, {}_A^C\mathcal{M}^C$, where the compatibility condition for an object $M$ in the first category reads as
    \[
        {}_M\Delta(a\triangleright m) = a_{(1)}\triangleright m_{(-1)}\otimes a_{(2)}\triangleright m_{(0)}, \quad a\in A, m\in M.        
    \]
    Lastly, for a left $A$-comodule algebra $B$ and a \textbf{right} $A$-module coalgebra $C$, we write ${}^C\mathcal{M}_B$ for the category consisting of left $C$-comodules and right $B$-modules $V$ such that
    \[
        {}_V\Delta(vb) = v_{(-1)}\triangleleft b_{(-1)}\otimes v_{(0)}\triangleleft b_{(0)}, \quad v \in V, b\in B
    \]
    and for a \textbf{right} $A$-comodule algebra $B$ and a left $A$-module coalgebra $C$, we write ${}_B\mathcal{M}^C$ for the analogous category.

    \subsection{Takeuchi's categorical equivalence}\label{SecTakeuchi}

    For this section, fix a Hopf algebra $A$ and a left coideal subalgebra $B\subseteq A$, that is, $B$ is a subalgebra of $A$ such that $\Delta(B) \subseteq A\otimes B$. Note that $B$ is naturally a left $A$-comodule algebra with coaction given by the coproduct. We consider the quotient coalgebra $A/B^+A$ together with the canonical projection $\pi_B\colon A\rightarrow A/B^+A$, and denote this quotient from now on by $\pi_B(A)$. Observe that the latter is a right $A$-module coalgebra with right $A$-action induced by the multiplication of $A$.

    We now recall the Takeuchi functors between ${}_B^A\mathcal{M}_B$ and ${}^{\pi_B}\mathcal{M}_B := {}^{\pi_B(A)}\mathcal{M}_B$. The first functor $\Phi\colon {}_B^A\mathcal{M}_B\rightarrow {}^{\pi_B}\mathcal{M}_B$ sends an object $\mathcal{F}\in {}_B^A\mathcal{M}_B$ to $\mathcal{F}/B^+\mathcal{F}$, where the right $B$-action is the one induced by the right $B$-action on $\mathcal{F}$, and the left $\pi_B(A)$-coaction is given by $[f]\mapsto \pi_B(f_{(-1)})\otimes [f_{(0)}]$. On morphisms the functor acts by taking induced maps on quotients. The second functor $\Psi \colon {}^{\pi_B}\mathcal{M}_B\rightarrow {}_B^A\mathcal{M}_B$ sends an object $V\in {}^{\pi_B}\mathcal{M}_B$ to $A\cotimes{\pi_B} V$ (here we write $\cotimes{\pi_B} := \cotimes{\pi_B(A)}$), where the left $B$-action is given by left multiplication on the first factor, the left $A$-coaction is given by applying the coproduct to the first factor, and the right $B$-action is given by $(a\cotimes{\pi_B}v)\triangleleft b = ab_{(1)}\cotimes{\pi_B} vb_{(2)}$. The functor $\Psi$ sends a morphism $\varphi\in {}^{\pi_B}\mathcal{M}_B$ to $\id_A\cotimes{\pi_B}\varphi$.

    \begin{theorem}
        \label{ThmTakeuchi}
        The functor $\Phi$ is left adjoint to $\Psi$. In more detail, the unit of the adjunction is given by
        \[
            \eta_{\mathcal{F}}\colon\mathcal{F} \rightarrow A\cotimes{\pi_B}\mathcal{F}/B^+\mathcal{F}, \quad f \mapsto f_{(-1)}\cotimes{\pi_B} [f_{(0)}]
        \]
        and the counit by
        \[
            \varepsilon_V\colon (A\cotimes{\pi_B}V)/B^+(A\cotimes{\pi_B}V) \rightarrow V,\quad [a\cotimes{\pi_B} v]\mapsto \varepsilon(a)v.
        \]
        If $A$ is faithfully flat as a right $B$-module, then this adjunction is an equivalence.
    \end{theorem}

    As a consequence of the above Theorem, we can write $B$ as the coinvariants under the right $\pi_B(A)$-coaction on $A$.

    \begin{corollary}
        \label{CorQHS}
        Suppose that $A$ is faithfully flat as a right $B$-module. Then
        \[
            B = A\coinv{\pi_B(A)} = \{a \in A \mid a_{(1)}\otimes \pi_B(a_{(2)}) = a\otimes \pi_B(1)\}.
        \]
    \end{corollary}

    The latter justifies the following terminology.

    \begin{definition}
        \label{DefQHS}
        A left coideal subalgebra $B\subseteq A$ such that $A$ is faithfully flat as a right $B$-module is called \textit{quantum homogeneous space}.
    \end{definition}

    Next let us recall that the category ${}_B^A\mathcal{M}_B$ is monoidal with monoidal product given by $-\otimes_B-$ and monoidal unit $B$ (consult \cite[Section 4.1]{Bua16} for more details on this monoidal structure). Thus, the category ${}^{\pi_B}\mathcal{M}_B$ inherits a monoidal structure through the equivalence in Theorem \ref{ThmTakeuchi}. In general however, there is no known description of the monoidal structure without reference to the category ${}_B^A\mathcal{M}_B$.
    
    The situation changes if we specialize our setup. Assume that the antipode of $A$ is invertible and that $AB^+ \subseteq B^+A$. Then, by Lemma \ref{LemABvsBA}, $\pi_B(A)$ is a quotient Hopf algebra of $A$. In this case consider the full subcategory ${}^{\pi_B}\mathcal{M}_0$ of ${}^{\pi_B}\mathcal{M}_B$ consisting of objects $V$, such that $v\triangleleft b = \varepsilon(b)v$. This category is equivalent to ${}^{\pi_B}\mathcal{M}$ and thus monoidal with monoidal product given by the tensor product over $k$. Further consider the full subcategory ${}_B^{A}\mathcal{M}_0$ of ${}_B^A\mathcal{M}_B$ consisting of objects $\mathcal{F}$ such that $\mathcal{F}B^+ \subseteq B^+\mathcal{F}$.

    \begin{proposition}
        \label{PropMonoidalTakeuchi}
        Assume that the antipode of $A$ is bijective and that $AB^+ \subseteq B^+A$. If $A$ is faithfully flat as a right $B$-module, then the functor $\Phi$ from Theorem \ref{ThmTakeuchi} restricts to a monoidal equivalence ${}_B^A\mathcal{M}_0\rightarrow {}^{\pi_B}\mathcal{M}_0$. The monoidal structure on this restriction is given by
        \[
            \Phi(\mathcal{F}\otimes_B \mathcal{G})\rightarrow \Phi(\mathcal{F})\otimes \Phi(\mathcal{G}), \quad [f\otimes_B g]\mapsto [f]\otimes [g].
        \]
    \end{proposition}
    
    \begin{proof}
        See \cite[Proposition 4.1]{Bua16}.
    \end{proof}

    \subsection{Schneider's dual equivalence}\label{SecSchneider}

    In this section, we demonstrate that the results of the previous section carry over to a dualized setup. Let $U$ be a Hopf algebra and let $H$ be a right coideal subalgebra, that is $H$ is a subalgebra such that $\Delta(H) \subseteq H\otimes U$. As in the previous section, we obtain a quotient coalgebra $C := U/UH^+ \cong U\otimes_H k$, which is naturally a left $U$-module coalgebra.

    The following Hopf--Galois-type isomorphism will be useful later.

    \begin{lemma}
        \label{LemHopfGalois}
        The canonical map
        \[
            \chi\colon U\otimes_H U \rightarrow C\otimes U, \quad x\otimes_H y \mapsto x_{(-1)}\otimes x_{(0)}y
        \]
        is an isomorphism.
    \end{lemma}

    \begin{proof}
        The inverse is given by
        \[
            C\otimes U \rightarrow U\otimes_H U, \quad [x]\otimes y \mapsto x_{(1)} \otimes_H S(x_{(2)})y.
        \]
        One checks that both maps are well-defined.
    \end{proof}

    We would now like to describe an equivalence of categories analogous to \ref{ThmTakeuchi}. To that end, we consider the coinvariant functor
    \[
        \coinv{C}(-)\colon{}_U^C\mathcal{M}^C \rightarrow {}_H\mathcal{M}^C, \quad M \mapsto \coinv{C}M = \{m\in M \mid m_{(-1)}\otimes m_{(0)} = 1\otimes m\}.
    \]
    As well as the induction functor
    \[
        U\otimes_H-\colon{}_H\mathcal{M}^C\rightarrow {}_U^C\mathcal{M}^C, \quad V\mapsto U\otimes_H V.
    \]
    Here, the left $C$-coaction on $U\otimes_H V$ is given by applying the left $C$-coaction on $U$ to the first factor, and the right $C$-coaction is given by the formula
    \[
        x\otimes_H v \mapsto (x_{(1)}\otimes_H v_{(0)}) \otimes x_{(2)}v_{(1)}.
    \]
    By a routine argument one shows that the above functors are mutually adjoint.

    \begin{proposition}
        \label{PropAdjunctionCoinvariantsInduction}
        The induction functor $U\otimes_H-\colon {}_U^C\mathcal{M}^C\rightarrow {}_H\mathcal{M}^C$ is left adjoint to the coinvariant functor $\coinv{C}(-)\colon {}_U^C\mathcal{M}^C\rightarrow {}_H\mathcal{M}^C$. This adjunction is exhibited the unit, and counit given by
        \[
            \eta_V\colon V \rightarrow \coinv{C}(U\otimes_H V), \quad v \mapsto 1\otimes_H v
        \]
        for $V \in {}_H\mathcal{M}^C$ and
        \[
            \varepsilon_M \colon U\otimes_H\coinv{C}M \rightarrow M, \quad x\otimes_H m \mapsto xm
        \]
        for $M \in {}_U^C\mathcal{M}^C$.
    \end{proposition}

    \begin{remark}
        \label{RemDomainCoinvariantsInduction}
        The functors $U\otimes_H-$ and $\coinv{C}(-)$ can also be viewed as functors between ${}_H\mathcal{M}$ and ${}_U^C\mathcal{M}$, and they still form a pair of adjoint functors with respect to those domains.
    \end{remark}

    Note that if $U$ is faithfully flat as a left $H$-module, we have an equivalence of categories
    \[
        {}_H\mathcal{M}_H^U \simeq {}_H\mathcal{M}^C.
    \]
    This equivalence is given by the same functors $\Phi$ and $\Psi$ as in the previous section, by changing their respective handedness, that is, in one direction we have $M\mapsto M/MH^+$ and in the other direction $V\mapsto V\cotimes{C} U$. Using this equivalence it follows that $H = \coinv{C}U$ (compare Corollary \ref{CorQHS}). This observation allows us to show that the functors $\coinv{C}(-)$ and $U\otimes_H -$ are equivalences.

    \begin{theorem}
        \label{ThmSchneiderVariationI}
        Assume that $U$ is faithfully flat as a left- and right $H$-module then $U\otimes_H-$ and $\coinv{C}(-)$ are mutually inverse equivalences of categories. This assertion holds for both pairs of categories $({}_H\mathcal{M}^C, {}_U^C\mathcal{M}^C)$ and $({}_H\mathcal{M}, {}_U^C\mathcal{M})$.
    \end{theorem}

    \begin{proof}
        Since $H = \coinv{C}U$, the assumptions of \cite[Theorem 3.7]{Sch90} are satisfied. Moreover, by Lemma \ref{LemHopfGalois} also condition of (1) in the same theorem holds. Thus, the induction functor is an equivalence between ${}_H\mathcal{M}$ and ${}_U^C\mathcal{M}$. The assertion that this equivalence also holds between ${}_U^C\mathcal{M}^C$ and ${}_H\mathcal{M}^C$ follows by Theorem \ref{ThmAdjunctionAndEquivalence} using the same argument as in the proof of Theorem \ref{ThmTakeuchi}.
    \end{proof}

    To end our discussion of categorical equivalences, we would like to prove a monoidal version of the above result, similar to Proposition \ref{PropMonoidalTakeuchi}. To this end observe that the category ${}_U^C\mathcal{M}^C$ is monoidal with monoidal product given by $-\cotimes{C}-$. Indeed, if $M, N \in {}_U^C\mathcal{M}^C$, then for any elements $m\cotimes{C} n \in M\cotimes{C} N$ and $x\in U$, we have
    \begin{align*}
        \Delta_M(x_{(1)}m)\otimes x_{(2)}n &= x_{(1)}m_{(0)}\otimes x_{(2)}m_{(1)}\otimes x_{(3)}n = x_{(1)}m \otimes x_{(2)}n_{(-1)}\otimes x_{(3)}n_{(0)}\\
        &= x_{(1)}m \otimes {}_N\Delta(x_{(2)}n).
    \end{align*}
    That is, the left $U$-action on $M\otimes N$ restricts to the cotensor product. Moreover, one can show that this left action is compatible with the left- and right $C$-coactions, in other words, we have $M\cotimes{C} N \in {}_U^C\mathcal{M}^C$. Since $C$ is a left $U$-module coalgebra it holds that $C \in {}_U^C\mathcal{M}^C$, and lastly, by checking that the associators and unitors of $({}^C\mathcal{M}^C, \cotimes{C}, C)$ are $U$-linear one shows that the monoidal structure of ${}^C\mathcal{M}^C$ extends to ${}_U^C\mathcal{M}^C$.
    
    Again, it is not obvious how to describe the monoidal structure on ${}_H\mathcal{M}^C$, which is induced by the one on ${}_U^C\mathcal{M}^C$ via the equivalence in \ref{ThmSchneiderVariationI}. We therefore restrict to the full subcategory ${}_H\mathcal{M}^0$ consisting of objects $V \in {}_H\mathcal{M}^C$, where the right $C$-coaction is given by $v \mapsto v\otimes [1]$. This subcategory is equivalent to ${}_H\mathcal{M}$. Thus, if $H$ is a sub bialgebra, then it is monoidal with monoidal product given by the tensor product over $k$. As the next result shows, this monoidal structure is the same as the one induced by the equivalence. In the following, we denote the full subcategory of ${}_U^C\mathcal{M}^C$ consisting of objects $M$ such that $\coinv{C}M \in {}_H\mathcal{M}^0$ by ${}_U^C\mathcal{M}^0$.

    \begin{proposition}
        \label{PropMonEquivalenceSupZero}
        Let $M \in {}_U^C\mathcal{M}^0$ and $N\in {}_U^C\mathcal{M}^C$, then
        \[
            \coinv{C}(M\cotimes{C}N) = \coinv{C}M\otimes \coinv{C}N.
        \]
        In particular, ${}_U^C\mathcal{M}^0$ is a monoidal subcategory of ${}_U^C\mathcal{M}^C$, and moreover, if $H$ is a sub bialgebra of $U$ and $U$ is faithfully flat as a left and right $H$-module, then $\coinv{C}(-)$ restricts to a monoidal equivalence between $({}_U^C\mathcal{M}^0, \cotimes{C}, C)$ and $({}_H\mathcal{M}^0,\otimes, k)$.
    \end{proposition}

    \begin{proof}
        Recall the left $C$-coaction on $M\cotimes{C} N$ is given by $m\cotimes{C}n\mapsto m_{(-1)}\otimes m_{(0)}\cotimes{C}n$. Thus, the inclusion ``$\supseteq$'' is immediate. For the converse inclusion let $u \in \coinv{C}(M\cotimes{C}N)$. Fix a basis $\{w_i\}_i$ of $N$. This lets us write
        \[
            u = \sum_i m_i\otimes w_i
        \]
        for uniquely determined $m_i \in M$, and we have
        \[
            \sum_i {}_M\Delta(m_i)\otimes w_i = {}_{M\cotimes{C}N}\Delta(u) = [1]\otimes u = \sum_i[1]\otimes m_i\otimes w_i.
        \]
        It follows that $\Delta(m_i) = [1]\otimes m_i$ for all $i$, and we have $u\in \coinv{C}M\otimes N$. Next, fix a basis $\{v_j\}_j$ of $\coinv{C}M$. We can similarly write
        \[
            u = \sum_jv_j\otimes n_j
        \]
        for uniquely determined $n_j \in N$, and we have
        \[
            \sum_jv_j\otimes {}_N\Delta(n_j) = \sum_j\Delta_M(v_j) \otimes n_j = \sum_j v_j\otimes [1] \otimes n_j.
        \]
        Therefore, we have $n_j \in \coinv{C}N$ for all $j$ and $u \in \coinv{C}M\otimes \coinv{C}N$.
    \end{proof}

    \subsection{Quantum homogeneous spaces as finitary duals of coalgebras}\label{SecInducedQHS}

    Let $U$ be a Hopf algebra and $H\subseteq U$ a right coideal subalgebra and $C := U\otimes_H k$ as in the previous section. In this section we will explain how to obtain a principal quantum homogeneous space as the finitary dual of $C$. To begin, we fix a tensor subcategory $\mathcal{C}\subseteq {}_U\mathcal{M}$ of finite dimensional $U$-modules, that is a full subcategory of finite dimensional $U$-modules such that $k \in \mathcal{C}$, and which is closed under taking direct sums, tensor products and duals \cite[p. 165]{MS99}. We consider the finitary dual of $U$ with respect to $\mathcal{C}$, the subspace of $U^\ast$ which consists of all matrix coefficients of modules in $\mathcal{C}$ (see Appendix \ref{SecFinDuals})
    \[
        A := U_\mathcal{C}^\circ = \{c_{f,v}^M \mid M \in \mathcal{C}\}.
    \]
    There is a natural left $U$-action on $A$ given by $x\triangleright a = a_{(2)}(x)a_{(1)}$, for $x \in U$ and $a \in A$. Let $B := {}^HA := \{a \in A\mid h\triangleright a = \varepsilon(h)a \text{ for all }h\in H\}$ denote the \textit{infinitesimal invariants} of $A$ with respect to the restriction of this action to $H$. We argue that we can identify $B$ with the finitary dual $C_\mathcal{C}^\circ$ of $C$ with respect to $\mathcal{C}$. The latter is defined as the subspace of $C^\ast$ consisting of all functionals $f\colon C\rightarrow k$, such that there is a left $U$-submodule $N\subseteq C$ such that $C/N \in \mathcal{C}$ and $f\vert_N \equiv 0$.

    \begin{lemma}
        \label{LemFinDualofC}
        The space infinitesimal invariants $B$ is a left coideal subalgebra of $A$. In addition, we have an isomorphism of $k$-algebras
        \[
            \iota\colon C_\mathcal{C}^\circ \rightarrow B, \quad \iota(f)(x) = f(x\otimes_H 1).
        \]
    \end{lemma}

    \begin{proof}
        We only show the second assertion. The first one is well known (see for instance \cite[Theorem 2.2]{MS99}).
        To see that the map $\iota$ is well-defined, let $f \in C_\mathcal{C}^\circ$ and let $N\subseteq C$ be a $U$-submodule of finite codimension such that $f\vert_N = 0$ and $C/N \in \mathcal{C}$. Denote by $\overline{f}$ the induced linear map $C/N \rightarrow k$, then $\iota(f)$ can be written as the matrix coefficient $c_{\overline{f},1\otimes_H 1 + N}^{C/N}$, which shows that the image of $\iota$ lies in $A$. Next let $h \in H$ and $x\in U$ be arbitrary, then
        \begin{align*}
            (h\triangleright \iota(f))(x) &= (\iota(f)_{(2)}(h)\iota(f)_{(1)})(x) = \iota(f)(xh) = f(xh\otimes_H 1) = \varepsilon(h)\iota(f)(x)
        \end{align*}
        and thus even $\iota(f) \subseteq B$ for all $f \in C_\mathcal{C}^\circ$. The inverse is given by the map $B\rightarrow C_\mathcal{C}^\circ$ which sends an element $b\in B$ to the map $x\otimes_H 1\mapsto b(x)$. This is well-defined, since
        \[
            b(xh) = (h\triangleright b)(x) = \varepsilon(h)b(x)
        \]
        for all $b \in B, x\in U$ and $h \in H$.
    \end{proof}

    As in Section \ref{SecTakeuchi} we consider the quotient map $\pi_B \colon A\rightarrow A/B^+A$.

    \begin{theorem}
        \label{ThmMuellerSchneider}
        Assume that the antipode of $A$ is bijective. Then the following hold
        \begin{enumerate}[label = (\arabic*)]
            \item If $H$ is a sub bialgebra, then $AB^+ \subseteq B^+A$. In particular, $\pi_B(A)$ is a quotient Hopf algebra.
            \item Suppose that $H$ is \textit{$\mathcal{C}$-semisimple}, in the sense that the restriction of every $V \in \mathcal{C}$ is completely reducible as an $H$-module. Then $A$ is faithfully flat as a left- and right $B$-module.
        \end{enumerate}
    \end{theorem}

    \begin{proof}
        \textit{(1)}: For any $a \in A$ and $b \in B$ we have
        \begin{align*}
            h\triangleright (a_{(1)}bS(a_{(2)})) &= a_{(1)(2)}(h_{(1)})b_{(2)}(h_{(2)})S(a_{(2)})_{(2)}(h_{(3)})a_{(1)(1)}b_{(1)}S(a_{(2)})_{(1)}\\
            &= a_{(2)}(h_{(1)})b_{(2)}(h_{(2)})S(a_{(3)})(h_{(3)})a_{(1)}b_{(1)}S(a_{(4)})\\
            &= (a_{(2)}S(a_{(3)}))(h)a_{(1)}bS(a_{(4)})\\
            &= \varepsilon(h)a_{(1)}bS(a_{(2)}).
        \end{align*}
        We see that $B$ is closed under the left adjoint action. In particular, if $b \in B^+$, then $ab = a_{(1)}bS(a_{(2)})a_{(3)} \in B^+A$, thus $AB^+ \subseteq B^+A$. The assertion that $\pi_B(A)$ is a quotient Hopf algebras now follows from \ref{LemABvsBA}.

        For \textit{(2)} see \cite[Theorem 2.2]{MS99}.
    \end{proof}

    We can summarize the previous result by saying that if the antipode of $A$ is bijective and $H$ is a sub bialgebra which is $\mathcal{C}$-semisimple, then the finitary dual $C_\mathcal{C}^\circ$ is a principal quantum homogeneous space.

    \begin{remark}
        The condition that the antipode of $A$ is bijective is for instance implied if the antipode of $U$ is bijective, and the tensor category $\mathcal{C}$ is closed under duals with respect to $S^{-1}$ (see Section \ref{SecFinDuals}).
    \end{remark}

    \subsection{Admissible pairings}

    We let $U, H, C, A, B$ and $\pi_B$ be as in the previous section. In this section, we discuss pairings between objects $\mathcal{F} \in {}_B^A\mathcal{M}_B$ and $M\in {}_U^C\mathcal{M}^C$, as well as between objects $V\in {}^{\pi_B}\mathcal{M}_B$ and $T\in {}_H\mathcal{M}^C$, that are compatible with the various (co)module structures as described in the following. 

    \begin{definition}
        \label{DefAdmissiblePairings}
        \begin{enumerate}[label = (\arabic*)]
            \item Let $V \in {}^{\pi_B}\mathcal{M}_B$ and $T\in {}_H\mathcal{M}^C$. A pairing $\langle-,-\rangle\colon V\times T \rightarrow k$ is called \textit{admissible} if
            \begin{equation}\label{eq:AdmissibleI}
                \langle v, ht\rangle = v_{(-1)}(h)\langle v_{(0)},t\rangle, \quad \langle vb,t\rangle = \langle v, t_{(0)}\rangle b(t_{(1)})
            \end{equation}
            for all $v \in V, t\in T, h\in H$ and $b \in B$.
            \item Let $\mathcal{F} \in {}_B^{A}\mathcal{M}_B$ and $M\in {}_U^C\mathcal{M}^C$. Then a pairing $\langle -,-\rangle\colon \mathcal{F}\times M \rightarrow k$ is called \textit{admissible} if
            \begin{equation}\label{eq:AdmissibleII}
                \begin{gathered}
                \langle bf, m\rangle = b(m_{(-1)})\langle f,m_{(0)}\rangle, \quad \langle fb,m\rangle = \langle f, m_{(0)}\rangle b(m_{(1)})\\
                \langle f, Xm\rangle = f_{(-1)}(X)\langle f_{(0)}, m\rangle
                \end{gathered}
            \end{equation}
            for all $f \in \mathcal{F}, b \in B, m \in M$ and $X \in U$.
        \end{enumerate}
    \end{definition}

    In the definition above, we make sense of and expression like $b(c)$ for $b \in B$ and $c\in C$, using the identification $B\cong C_\mathcal{C}^\circ$ of \ref{LemFinDualofC}, and an expression $\pi_B(a)(h)$ is understood as $a(h)$ for $a \in A$ and $h\in H$. The value of $a(h)$ only depends on $\pi_B(a)$, since
    \[
        b(h) = b(1h) = (h\triangleright b)(1) = \varepsilon(h)b(1) = 0            
    \]
    for all $b \in B^+$.

    \begin{proposition}
        \label{PropTranslAdmissiblePairings}
        Let $\mathcal{F} \in {}_B^A\mathcal{M}_B, M \in {}_U^C\mathcal{M}^C$, and write $V := \mathcal{F}/B^+\mathcal{F}, T := \coinv{C}M$. Then for every admissible pairing $\langle -,-\rangle\colon \mathcal{F}\times M\rightarrow k$, the induced pairing $\langle -, -\rangle^\flat \colon V\times T \rightarrow k$ given by $\langle [f], t\rangle^\flat = \langle f, t \rangle$ is admissible. Conversely, any admissible pairing $\langle-,-\rangle\colon V\times T\rightarrow k$ between arbitrary objects $V \in {}^{\pi_B}\mathcal{M}_B$ and $T \in {}_H\mathcal{M}^C$ extends to an admissible pairing
        \[
            \langle-,-\rangle^\sharp\colon A\cotimes{\pi_B }V\times U\otimes_HT \rightarrow k, \quad \langle a\cotimes{\pi_B} v, x\otimes_H t\rangle^\sharp = a(x)\langle v,t\rangle.
        \]
        If the antipode of $A$ is bijective, $A$ is faithfully flat as a right $B$-module and $U$ is faithfully flat as a left- and right $H$-module, then $\langle-,-\rangle \mapsto \langle-,-\rangle^\flat$ is a bijection from the set of admissible pairings between $\mathcal{F}$ and $M$ and the set of admissible pairings between $V$ and $T$.
    \end{proposition}

    \begin{proof}
        Let $\langle -,-\rangle\colon \mathcal{F}\times M\rightarrow k$ be an admissible pairing. Let $f \in \mathcal{F}, t \in T$ and $b \in B$, then
        \begin{align*}
            \langle b^+f, t \rangle = b^+(t_{(-1)})\langle f, t_{(0)}\rangle = b^+([1])\langle f, t\rangle = 0.
        \end{align*}
        Thus, the pairing $\langle -, -\rangle^\flat$ is well-defined. Admissibility follows directly from the admissibility of $\langle -,-\rangle$.
        
        For the other direction, we start with a pairing $\langle -, -\rangle\colon V\times T \rightarrow k$, then its extension is well-defined since
        \begin{align*}
            \langle a\cotimes{\pi_B} v, Xh\otimes t\rangle &= a(Xh)\langle v, t\rangle = a_{(1)}(X)a_{(2)}(h)\langle t,v\rangle = a_{(1)}(X) [a_{(2)}](h)\langle v,t\rangle\\
            &= a(X)v_{(-1)}(h)\langle v_{(0)},t \rangle = a(X)\langle v, ht\rangle = \langle a\cotimes{\pi_B}v,X\otimes ht\rangle.
        \end{align*}
        For all $a\cotimes{\pi_B}v \in A\cotimes{\pi_B}V, X \in U, h \in H$ and $t \in T$. One then checks that $\langle-,-\rangle^\sharp$ is admissible.
        
        For the last assertion, we can use that under our assumptions $\mathcal{F} \cong A\cotimes{\pi_B}V$ and $M \cong U\otimes_H T$ (see Theorem \ref{ThmSchneiderVariationI} and Theorem \ref{ThmTakeuchi}). Thus, the extension $\langle-,-\rangle^\sharp$ of an admissible pairing $\langle-,-\rangle$ induces an admissible pairing between $\mathcal{F}$ and $M$. The resulting mapping is inverse to $\langle-,-\rangle \mapsto \langle-,-\rangle^\flat$, arguing as follows: If we start with a pairing $\langle -,-\rangle$ between $\mathcal{F}$ and $M$, then we have for all $f \in \mathcal{F}$ and $t\in T$
        \begin{align*}
            \langle f, t\rangle^{\flat, \sharp} &= \langle f_{(-1)}\cotimes{\pi_B} [f_{(0)}], 1\otimes_H t\rangle^{\flat,\sharp} = f_{(-1)}(1)\langle [f_{(0)}], t\rangle^\flat = \langle [f], t\rangle^\flat = \langle f, t\rangle
        \end{align*}
        which implies that $\langle-,-\rangle^{\flat, \sharp} = \langle-,-\rangle$, since admissible pairings are determined by their restriction $\mathcal{F}\times T\rightarrow k$. If we start with a pairing $\langle-,-\rangle^\flat\colon V\times T\rightarrow k$, then using the same computation, its extension satisfies
        \begin{align*}
            \langle f, t\rangle^\sharp = \langle [f], t\rangle  
        \end{align*}
        for all $f \in \mathcal{F}$ and $t\in T$. Thus, the correspondence is bijective.
    \end{proof}

    For a pairing between vector spaces $\langle-,-\rangle\colon V\times T\rightarrow k$ we write
    \[
        W^\perp := \{t \in T \mid \langle w,t \rangle = 0 \text{ for all } w\in W\}
    \]
    for a subset $W\subseteq V$ and similarly
    \[
        S^\perp = \{v \in V\mid \langle v,s\rangle = 0 \text{ for all } s \in S\}
    \]
    for a subset $S\subseteq T$.

    \begin{lemma}
        \label{LemStabilityUnderOrthogonalComplements}
        Let $\mathcal{F} \in {}_B^A\mathcal{M}_B$ and $M\in {}_U^C\mathcal{M}^C$, and let $\langle-,-\rangle\colon \mathcal{F}\times M \rightarrow k$ be an admissible pairing. Then for any subobject $N \subseteq M$ in ${}_U^C\mathcal{M}^C$, the orthogonal complement $N^\perp$ is a subobject of $\mathcal{F}$ in ${}_B^A\mathcal{M}_B$. Moreover, if $B$ separates the elements of $C$, then also for every subobject $\mathcal{E} \subseteq \mathcal{F}$, the orthogonal complement $\mathcal{E}^\perp$ is a subobject of $M$.
    \end{lemma}

    \begin{proof}
        For the first assertion let $f \in N^\perp, b \in B$ and $n \in N$, then
        \begin{align*}
            \langle fb, n\rangle = \langle f,n_{(0)}\rangle b(n_{(1)}) = 0
        \end{align*}
        since $n_{(0)}\otimes n_{(1)} \in N\otimes C$, thus $fb \in N^\perp$. Similarly, $bf \in N^\perp$. Also for any $f \in N^\perp, x \in U$ and $n\in N$
        \begin{align}
            \langle f_{(-1)}(x)f_{(0)}, n\rangle = \langle f, xn\rangle = 0 
        \end{align}
        since $xn \in N$, thus $f_{(-1)}\otimes f_{(0)} \in A\otimes N^\perp$. For the second assertion let $m\in \mathcal{E}^\perp, b \in B$ and $e\in \mathcal{E}$, then
        \[
            \langle e, b(m_{(1)})m_{(0)}\rangle = \langle eb,m\rangle = 0
        \]
        since $eb \in \mathcal{E}$. It follows that $b(m_{(1)})m_{(0)}\in \mathcal{E}^\perp$ for all $b \in B$, since $B$ separates the elements of $C$, we have $m_{(0)}\otimes m_{(1)}\in \mathcal{E}^\perp\otimes C$. By the same argument, we also have $m_{(-1)}\otimes m_{(0)}\in C\otimes \mathcal{E}^\perp$. Lastly, if $x \in U$, then for all $e\in \mathcal{E}$
        \[
            \langle e, xm\rangle = e_{(-1)}(x)\langle e_{(0)},m\rangle = 0
        \]
        since $e_{(-1)}\otimes e_{(0)}\in A\otimes \mathcal{E}$, and thus also $xm \in \mathcal{E}^\perp$.
    \end{proof}

    \begin{lemma}
        \label{LemOrthogonalComplements}
        Let $V \in {}^{\pi_B}\mathcal{M}_B, T\in {}_H\mathcal{M}^C$ and let $\langle-,-\rangle^\flat\colon W\times T \rightarrow k$ be an admissible pairing. Then the following assertions for the extended pairing
        \[
            \langle -,-\rangle\colon(A\cotimes{\pi_B}V)\times (U\otimes_H T) \rightarrow k
        \]
        (see Proposition \ref{PropTranslAdmissiblePairings}) hold.
        \begin{enumerate}[label = (\arabic*)]
            \item For any subobject $S\subseteq T$ in ${}_H\mathcal{M}^C$, one has $(U\otimes_H S)^\perp = A\cotimes{\pi_B} S^\perp$.
            \item Assume that the antipode of $A$ is bijective, that $A$ is faithfully flat as a left $B$-module, and that $U$ is faithfully flat as a left- and right $H$-module. Further assume that $B$ separates the elements of $C$. Let $W\subseteq V$ be a subobject in ${}^{\pi_B}\mathcal{M}_B$. Using Lemma \ref{LemStabilityUnderOrthogonalComplements} and Theorem \ref{ThmSchneiderVariationI} we can write $(A\cotimes{\pi_B} W)^\perp = U\otimes_H S$ for some subobject $S\subseteq T$ in ${}_H\mathcal{M}^C$. If $W^{\perp\perp} = W$ and $S^{\perp\perp} = S$ (note that this is automatic if $V$ and $T$ are finite dimensional and $\langle-,-\rangle^\flat$ is non-degenerate), then $(A\cotimes{\pi_B} W)^\perp = U\otimes_H W^\perp$.
        \end{enumerate}
    \end{lemma}

    \begin{proof}
        \textit{(1)}: The inclusion $A\cotimes{\pi_B} S^\perp \subseteq (U\otimes_H S)^\perp$ follows by direct computation. For the other inclusion let $a\cotimes{\pi_B} v \in (U\otimes_H S)^\perp$ and write
        \[
            a\cotimes{\pi_B} v = \sum_{i=1}^Na_i\otimes v_i
        \]
        such that $a_1, \dots, a_N$ is linear independent. This lets us choose $X_1, \dots, X_n \in U$ such that $a_i(X_j) = \delta_{ij}$ for all $1\le i,j\le N$. Now for all $1\le i \le N$ and $s \in S$
        \begin{align*}
            \langle v_i, s\rangle^\flat &= \sum_{j=1}^Na_j(X_i)\langle v_j, s\rangle = \langle a\cotimes{\pi_B} v, X_i\otimes_H s\rangle = 0.
        \end{align*}
        That is $v_i \in S^\perp$ for all $1\le i \le N$, and therefore $a\cotimes{\pi_B} v \in A\cotimes{\pi_B}S^\perp$.

        \textit{(2)}: On the one hand we have
        \[
            A\cotimes{\pi_B}W \subseteq (A\cotimes{\pi_B}W)^{\perp\perp} = (U\otimes_H S)^\perp\overset{(1)}{=} A\cotimes{\pi_B} S^\perp.
        \]
        Therefore, we have $W \subseteq S^\perp$. On the other hand one has $U\otimes_H W^\perp \subseteq (A\cotimes{\pi_B}W)^\perp = U\otimes_H S$ by direct computation. Again using \textit{(1)} this gives 
        \[
            A\cotimes{\pi_B} S^\perp = (U\otimes_H S)^\perp \subseteq (U\otimes_H W^\perp)^\perp = A\cotimes{\pi_B} W^{\perp\perp} = A\cotimes{\pi_B}W.
        \]
        Thus, also $S^\perp \subseteq W$, and it follows that $W^\perp = S^{\perp\perp} = S$.
    \end{proof}

    \section{Equivariant codifferential calculi} \label{SecEquivariantCocalculi}

    In this section, we consider codifferential calculi, that are equipped with a compatible left action of a Hopf algebra. Thus, we fix a Hopf algebra $U$ and a left $U$-module coalgebra $C$ over $U$.
    
    \subsection{First- and higher-order equivariant codifferential calculi}

    We start with basic definitions.

    \begin{definition}
        \begin{enumerate}[label = (\arabic*)]
            \item A first-order codifferential calculus $(\mathscr{W}_1, \delta)$ over $C$ is defined to be \textit{$U$-equivariant} if $\mathscr{W}_1\in {}_U^C\mathcal{M}^C$ and $\delta$ is $U$-linear.
            \item A (higher-order) codifferential calculus $(\mathscr{W}_\bullet, \delta)$ over $C$ is called \textit{$U$-equivariant} if for each $n\in \Z_{\ge 0}$ the component $\mathscr{W}_n$ is an object in ${}_U^C\mathcal{M}^C$, and the coproduct as well as the differential $\delta$ are $U$-linear.
        \end{enumerate}
    \end{definition}

    The easiest example of an equivariant codifferential calculus is given by the universal one.

    \begin{example}
        \label{ExplUniversalFOCCEquivariant}
        Let $\mathscr{W}_1^u$ be the universal first-order codifferential calculus. Since $C$ is a $U$-module coalgebra, the image of the comultiplication is a $U$-submodule of $C\otimes C$, thus $\mathscr{W}_1^u$ is a quotient $U$-module of $C\otimes C$. Moreover, one can check that the $U$-action is compatible with the $C$-coactions and that $\delta^u$ is $U$-linear. Thus, the universal first-order codifferential calculus is $U$-equivariant. This equivariance extends to the higher-orders. Since ${}_U^C\mathcal{M}^C$ is a monoidal category with monoidal structure as in the discussion above \ref{PropMonEquivalenceSupZero} it follows that $(\mathscr{W}_n^u)^\mathrm{max} = (\mathscr{W}_1^u)^{\cotimes{C} n} \in {}_U^C\mathcal{M}^C$ for all $n\ge 0$. Moreover, recall that in the case of the universal codifferential the degree two differential is given by
        \[
            \delta_2^u\colon \mathscr{W}_1^u\cotimes{C}\mathscr{W}_1^u \rightarrow \mathscr{W}_1^u, \quad w_1\cotimes{C}w_2 \mapsto [\delta^u(w_1)\otimes \delta^u(w_2)]
        \]
        which is a $U$-linear map. It follows that also for every $n\ge 2$ the differential $\delta_n^u$ is $U$-linear since it is given by the formula \eqref{eq:MaximalProlongationDifferential}. Since the coproduct is given by deconcatenation, it is $U$-linear as well, therefore the maximal prolongation of the universal first-order codifferential calculus is $U$-equivariant.
    \end{example}

    It follows directly from the definition, that for any $U$-equivariant codifferential calculus $(\mathscr{W}_\bullet, \delta)$ over $C$ the associated first-order codifferential calculus $\underline{\mathscr{W}_1}(\mathscr{W}_\bullet)$ is $U$-equivariant. The converse also holds.

    \begin{proposition}
        \label{PropEquivarianceOfProlongation}
        Let $(\mathscr{W}_1,\delta)$ be a $U$-equivariant first-order codifferential calculus. Then its maximal prolongation is also $U$-equivariant.
    \end{proposition}

    \begin{proof}
        Consider the embedding $\mathscr{W}_\bullet^\mathrm{max}\hookrightarrow T_C^c(\mathscr{W}_1)$. As in Example \ref{ExplUniversalFOCCEquivariant}, we see that each $(\mathscr{W}_1)^{\cotimes{C} n}$ is a $U$-submodule and that the coproduct of $T_C^c(\mathscr{W}_1)$ is $U$-linear. Thus, it suffices to show that every $\mathscr{W}_n^\mathrm{max}$ is a $U$-submodule of $(\mathscr{W}_1)^{\cotimes{C} n}$ and that the differential is $U$-linear. In fact, we only need to check this for $\mathscr{W}_2^\mathrm{max}$ and $\delta_2$ by \eqref{eq:MaximalProlongationDegn} and \eqref{eq:MaximalProlongationDifferential}. Recall that $\mathscr{W}_2^\mathrm{max}$ is defined as the maximal sub $C$-bicomodule such that the map
        \[
            \check{\delta}\colon \mathscr{W}_1\cotimes{C}\mathscr{W}_1\rightarrow \mathscr{W}_1^u, \quad w\cotimes{C}v\mapsto [\delta(w)\otimes \delta(v)]
        \]
        factors through the canonical embedding $\mathscr{W}_1 \hookrightarrow \mathscr{W}_1^u$. As in the previous example this map is $U$-linear. Note also that the embedding $\varphi\colon\mathscr{W}_1 \rightarrow \mathscr{W}_1^u, w \mapsto [w_{(-1)}\otimes \delta(w_{(0)})]$ is $U$-linear. Thus, if $M$ is any sub $C$-bicomodule such that $\check{\delta}(M)\subseteq \varphi(\mathscr{W}_1)$, then also $\check{\delta}(UM)\subseteq \varphi(\mathscr{W}_1)$. Moreover, $UM$ is again a sub $C$-bicomodule. For instance, the assertion that $UM$ is closed under the left $C$-coaction holds, since
        \[
            (xm)_{(-1)}\otimes (xm)_{(0)} = x_{(1)}m_{(-1)}\otimes x_{(2)}m_{(0)} \in C\otimes UM
        \]
        for all $x \in U$ and $m \in M$. In particular, it follows that $\mathscr{W}_2^\mathrm{max}$ is a $U$-submodule, and the induced map $\delta_2\colon \mathscr{W}_2^\mathrm{max}\rightarrow \mathscr{W}_1$ is $U$-linear.
    \end{proof}

    \subsection{Classification of equivariant first-order codifferential calculi} \label{SecHermissionClassification}

    In this section, we specialize the setup to the case where $C = U\otimes_H k$, for a right coideal subalgebra $H\subseteq U$. Our goal is to show that $U$-equivariant first-order codifferential calculi are all classified by quantum tangent spaces if $U$ is faithfully flat as a left and right $H$-module.

    \begin{definition}
        \label{DefQuantumTangentSpace}
        A \textit{quantum tangent space} is a subspace $T\subseteq C^+ = \ker(\varepsilon)$ such that
        \begin{equation}
            \Delta(T) \subseteq (T\oplus k[1])\otimes C, \quad HT \subseteq T.
        \end{equation}
    \end{definition}

    Note that the condition on the image of $T$ under the coproduct can be equivalently phrased as demanding that $T$ is a right $C$-subcomodule of $C^+$, where the latter is endowed with the coaction
    \begin{equation}\label{eq:ReducedRightCoaction}
        C^+ \rightarrow C^+\otimes C, \quad c\mapsto c_{(1)}^+\otimes c_{(2)}.
    \end{equation}
    Moreover, since $T$ is closed under the left $H$-action one can deduce that $T$ is an object in ${}_H\mathcal{M}^C$. Indeed, we have compatibility of the $H$-action and the right $C$ action
    \begin{align*}
        (hc)_{(1)}^+\otimes (hc)_{(2)} &= h_{(1)}c_{(1)}\otimes h_{(2)}c_{(2)} - \varepsilon(h_{(1)})\varepsilon(c_{(1)})[1]\otimes h_{(2)}c_{(2)}\\
        &= h(c_{(1)}\otimes c_{(2)}) - \varepsilon(c_{(1)})h_{(1)}[1]\otimes h_{(2)}c_{(2)}\\ &= h(c_{(1)}^+\otimes c_{(2)})
    \end{align*}
    for all $h \in H$ and $c \in C$, where we used the fact that $h[1] = [h] = \varepsilon(h)[1]$.

    If we apply the induction functor $U\otimes_H -$ to $T$, we get the following.

    \begin{proposition}
        \label{PropTangentspaceFOCC}
        Let $T \subseteq C$ be a quantum tangent space, and write $\mathscr{W}_1^T := U\otimes_H T$. Then 
        \[
            \delta^T\colon \mathscr{W}_1^T \rightarrow C, \quad x\otimes_H t \mapsto xt
        \]
        is a coderivation. Moreover, if $U$ is flat as a left $H$-module then $(\mathscr{W}_1^T,\delta^T)$ is a $U$-equivariant first-order codifferential calculus.
    \end{proposition}

    \begin{proof}
        Let $x \in U$ and let $t \in T$. Then
        \begin{align*}
            \Delta(\delta^T(x\otimes_H t)) &= \Delta(xt) = x_{(1)}t_{(1)}\otimes x_{(2)}t_{(2)}\\
            &= x_{(1)}t_{(1)}^+\otimes x_{(2)}t_{(2)} + x_{(1)}\varepsilon(t_{(1)})[1]\otimes x_{(2)}t_{(2)}\\
            &= \delta^T(x_{(1)}\otimes_H t_{(1)}^+)\otimes x_{(2)}t_{(2)} + [x_{(1)}]\otimes \delta^T(x_{(2)}\otimes_H t)\\
            &= (\delta^T\otimes \id)\Delta_{\mathscr{W}_1^T}(x\otimes_H t) + (\id \otimes \delta^T){}_{\mathscr{W}_1^T}\Delta(x\otimes_H t).
        \end{align*}
        Now assume that $U$ is flat as a left $H$-module. Consider the map
        \[
            \overline{\chi}\colon U\otimes_H C \rightarrow C\otimes C, \quad x\otimes_H c, \quad x \otimes_H c \mapsto [x_{(1)}]\otimes x_{(2)}c
        \]
        and observe that this is simply the map $\chi\otimes_H \id_k$ where $\chi$ is the canonical map $U\otimes_H U \rightarrow C\otimes U$ from Lemma \ref{LemHopfGalois}. Hence, it is an isomorphism. By flatness, we have an embedding $\mathscr{W}_1^T \hookrightarrow U\otimes_H C$. Therefore, also the map
        \[
            \mathscr{W}_1^T \rightarrow C\otimes C, \quad x\otimes_H t \mapsto x_{(1)}\otimes x_{(2)}t
        \]
        is injective. However, this map is precisely the map which is required to be injective in Definition \ref{DefFOCC}, since $x_{(1)}\otimes x_{(2)}t = (\id \otimes\delta^T){}_{\mathscr{W}_1^T}\Delta(x\otimes_H t)$ for all $x\in U$ and $t \in T$. Lastly, equivariance follows directly by the definition of $\mathscr{W}_1^T$ and $\delta^T$.
    \end{proof}

    We can now state our classification result.

    \begin{theorem}
        \label{ThmClassificationEquivariantFOCC}
        Assume that $U$ is faithfully flat as a left- and right $H$ module. Then the assignment $T \mapsto (\mathscr{W}_1^T, \delta^T)$ from Proposition \ref{PropTangentspaceFOCC} induces a bijection
        \begin{gather*}
            \{\text{quantum tangent spaces } T\subseteq C\}\\ \overset{1:1}{\longleftrightarrow} \{\text{equivariant first-order codifferential calculi } (\mathscr{W}_1, \delta)\}/\cong.
        \end{gather*}
    \end{theorem}

    \begin{proof}
        Given an equivariant first-order codifferential calculus $(\mathscr{W}_1, \delta)$ set $T := \delta(\coinv{C}\mathscr{W}_1)$. By the universal property of the universal first-order codifferential calculus \ref{PropUniversalFOCC}, $T$ only depends on the isomorphism class of $\mathscr{W}_1$. Moreover, $T$ is a quantum tangent space. Indeed, we have $HT \subseteq T$, since the $C$-coinvariants are closed under the $H$-action and $\delta$ is $U$-linear. Since the image of any coderivation is contained in $C^+$, we in particular have $T\subseteq C^+$. Lastly, by
        \begin{align*}
            \delta(w)_{(1)}^+\otimes \delta(w)_{(2)} &= w_{(-1)}^+\otimes \delta(w_{(0)}) + \delta(w_{(0)})^+\otimes w_{(1)} = [1]^+\otimes \delta(w) + \delta(w_{(0)})\otimes w_{(1)}\\
            &= \delta(w_{(0)})\otimes w_{(1)}
        \end{align*}
        for all $w\in \coinv{C}\mathscr{W}_1$, the restriction of $\delta$ to $\coinv{C}\mathscr{W}_1$ is right $C$-colinear, and $T$ is a right $C$-subcomodule of $C^+$. To show the assignment is a bijection, note that on the one hand, by the categorical equivalence of Corollary \ref{ThmSchneiderVariationI}, we get an isomorphism in ${}_U^C\mathcal{M}^C$
        \[
            U\otimes_H T = U\otimes_H \delta(\coinv{C}\mathscr{W}_1) \cong \mathscr{W}_1
        \]
        which is concretely given by $x\otimes_H \delta(w)\mapsto xw$. Since $\delta$ is $U$-linear, this map is also compatible with the differentials, and hence an isomorphism of $U$-equivariant first-order codifferential calculi. On the other hand, given a quantum tangent space $T\subseteq C$, then $\coinv{C}(U\otimes_H T) = 1\otimes_H T$, which again follows from the categorical equivalence, and thus $\delta^T(\coinv{C}\mathscr{W}_1^T) = T$.
    \end{proof}

    \begin{example}
        \label{ExplTangentSpaceofUniversalFOCC}
        Consider the universal first-order codifferential calculus $\mathscr{W}_1^u$ over $C$. As an object in ${}_U^C\mathcal{M}^C$ we have
        \[
            \mathscr{W}_1^u \cong C\otimes C^+.
        \]
        Here, the $C$-bicomodule structure on $C\otimes C^+$ is given as in \ref{RemUniversalFOCCKhalkhali}, and the left $U$-action is simply the diagonal action.
        Since the left coaction is given by applying the coproduct to the first factor, the space of left $C$-coinvariants of $C\otimes C^+$ is $C^+$. Thus, the quantum tangent space of $\mathscr{W}_1^u$ is $C^+$, and we have an isomorphism of first-order codifferential calculi $\mathscr{W}_1^u\cong U\otimes_H C^+$. Explicitly, we have an isomorphism
        \[
            U\otimes_H C^+ \rightarrow C\otimes C^+, \quad x\otimes d \mapsto [x_{(1)}]\otimes x_{(2)}d
        \]
        which is given by restricting the isomorphism $\overline{\chi}$ in the proof of Proposition \ref{PropTangentspaceFOCC}. In particular, its inverse is the map given by $[x]\otimes d \mapsto x_{(1)}\otimes_H S(x_{(2)})d$. Combining this isomorphism with the isomorphism $\mathscr{W}_1^u\cong C\otimes C^+$, the isomorphism between $U\otimes_H C^+$ and $\mathscr{W}_1^u$ is given by
        \[
            U\otimes_H C^+ \rightarrow \mathscr{W}_1^u, \quad x\otimes_H d \mapsto [[x_{(1)}]\otimes x_{(2)}d]
        \]
        with inverse
        \[
            \mathscr{W}_1^u \rightarrow U\otimes_H C^+, \quad [[x]\otimes d]\mapsto x_{(1)}\otimes_H S(x_{(2)})d - \varepsilon(d)x\otimes_H [1].
        \]
    \end{example}

    To finish the section, we look at the special case where $H$ is the ground field. The result we obtain dualizes a result of Woronowicz \cite{Wor89} about covariant first-order differential calculi on Hopf algebras.

    \begin{corollary}
        \label{CorClassificationEquivariantFOCCHopf}
        For any Hopf algebra $U$, equivariant first-order codifferential calculi on $U$ are, up to isomorphism, in bijective correspondence with subspaces $T\subseteq U^+$ such that
        \[
            \Delta(T)\subseteq (T\oplus k1)\otimes U.
        \] 
        More precisely, given a subspace $T$ as above, then $U\otimes T$ is a first-order codifferential calculus with codifferential
        \[
            \delta^T\colon U\otimes T \rightarrow U, \quad x\otimes t\mapsto xt.
        \]
    \end{corollary}

    \subsection{Computation of higher-order equivariant codifferential calculi} \label{SecHigherOrderEquivariantCalculi}

    We let $U, H$ and $C$ be the same as in the previous section. We additionally assume that $H$ is a subbialgebra and that $U$ is faithfully flat as a left and right $H$-module.

    As we know from Theorem \ref{ThmClassificationEquivariantFOCC} an equivariant first-order codifferential calculus $\mathscr{W}_1$ over $C$ is uniquely determined by its quantum tangent space $T$. The goal of this section is to give a description of the maximal prolongation of $\mathscr{W}_1$, in the case where the right $C$-coaction on $T$ is trivial, that is, when $T \in {}_H\mathcal{M}^0$ (see the discussion around Proposition \ref{PropMonEquivalenceSupZero}). The key step in computing the maximal prolongation is to compute $\mathscr{W}_2^\mathrm{max}$. Recall from section \ref{SecMaximalProlongation} that this space is determined by the map
    \[
        \mathscr{W}_1\cotimes{C}\mathscr{W}_1 \rightarrow \mathscr{W}_1^u, \quad w_1\cotimes{C}w_2 \mapsto [\delta(w_1)\otimes \delta(w_2)].
    \]
    Let us assume that $\mathscr{W}_1 = U\otimes_H T$ with $T\in {}_H\mathcal{M}^0$. By Example \ref{ExplTangentSpaceofUniversalFOCC}, we can identify $\mathscr{W}_1^u$ with $U\otimes_H C^+$, and using the monoidal equivalence from Proposition \ref{PropMonEquivalenceSupZero} we can identify $\mathscr{W}_1\cotimes{C}\mathscr{W}_1$ with $U\otimes_H (T\otimes T)$. Under these identifications, the map above becomes
    \[
        \hat{\delta}\colon U\otimes_H (T\otimes T) \rightarrow U\otimes_H C^+, \quad 1\otimes_H ([x]\otimes t)\mapsto x_{(1)}\otimes_H S(x_{(2)})t.
    \]
    Here it suffices to specify the values of elements of the form $1\otimes_H (t_1\otimes t_2)$ by $U$-linearity. In the following we write $\check{\delta}_0([x]\otimes t) := \check{\delta}(1\otimes_H ([x]\otimes t))$.

    \begin{proposition}
        \label{PropDegTwoRelationsSupZero}
        Write $\check{R} := \check{\delta}_0^{-1}(U\otimes_H T)$, then under the identification $\mathscr{W}_1\cotimes{C}\mathscr{W}_1 \cong U\otimes_H (T\otimes T)$, the subspace $\mathscr{W}_2^\mathrm{max}$ corresponds to $U\otimes_H \check{R}$.
    \end{proposition}

    \begin{proof}
        We have to show that $U\otimes_H \check{R}$ is the largest $C$-sub-bicomodule of $U\otimes_H (T\otimes T)$ which gets mapped to $U\otimes_H T$ under $\check{\delta}$. Indeed, it follows by construction that $U\otimes_H \check{R}$ gets mapped into $U\otimes_H T$ under $\check{\delta}$. To show that it is maximal, let $L \subseteq U\otimes_H (T\otimes T)$ be a $C$-sub-bicomodule which gets mapped to $U\otimes_H T$. Arguing as in the end of the proof of Proposition \ref{PropEquivarianceOfProlongation} we can assume without loss of generality that $L$ is a $U$-submodule, and thus a subobject of $U\otimes_H (T\otimes T)$ in ${}_U^C\mathcal{M}^C$. Write $W := \coinv{C}L$. By Proposition \ref{PropAdjunctionCoinvariantsInduction} and Theorem \ref{ThmSchneiderVariationI}, the map
        \[
            U\otimes_H W \rightarrow L, \quad u\otimes_H w \mapsto uw
        \]
        is an isomorphism in ${}_U^C\mathcal{M}^C$. Since $W \subseteq \coinv{C}(U\otimes_H (T\otimes T)) = 1\otimes_H (T\otimes T)$, we can identify $W$ as a subspace of $T\otimes T$. Under this identification, the above map is just the identity, and we have $L = U\otimes_H W$, and by assumption $\check{\delta}(W) \subseteq \check{\delta}(L) \subseteq U\otimes_H T$. Thus, we have $W\subseteq \check{R}$ and therefore $L\subseteq U\otimes_H \check{R}$.
    \end{proof}

    For the next result, recall the notion of a quadratic coalgebra $\mathfrak{C}(T,\check{R})$, corresponding to a finite dimensional vector space $T$ and a subspace $\check{R}\subseteq T\otimes T$ (see Section \ref{SecQuadraticCoalgebras}). Note that if $T \in {}_H\mathcal{M}^0$ and $\check{R}$ is an $H$-submodule of $T\otimes T$, then $\mathfrak{C}(T,\check{R})$ is a coalgebra in the category ${}_H\mathcal{M}^0$.

    \begin{theorem}
        \label{ThmMaxProlongationViaQuadraticCoalgs}
        Let $\mathscr{W}_1$ be an equivariant first-order codifferential calculus such that its quantum tangent space $T$ is an object in ${}_H\mathcal{M}^0$, and let $\check{R}$ be as in Proposition \ref{PropDegTwoRelationsSupZero}. Then
        \[
            \mathscr{W}_\bullet^\mathrm{max} \cong U\otimes_H \mathfrak{C}(T,\check{R})
        \]
        as coalgebras in ${}_U^C\mathcal{M}^C$.
    \end{theorem}

    \begin{proof}
        For simplicity, assume that $\mathscr{W}_1 = U\otimes_H T$. First we show that $\coinv{C}\mathscr{W}_\bullet^\mathrm{max} \cong \mathfrak{C}(T,\check{R})$ as objects in ${}_H\mathcal{M}^C$. In degree zero and one, it is immediate that $\coinv{C}\mathscr{W}_0^\mathrm{max} = \coinv{C} C \cong k$, and $\coinv{C}\mathscr{W}_1^\mathrm{max} = \coinv{C}(U\otimes_H T) \cong T$. Hence, these coinvariant subspaces agree with $\mathfrak{C}^{(0)}(T,\check{R})$, respectively $\mathfrak{C}^{(1)}(T,\check{R})$. For $n\ge 2$, combining Proposition \ref{PropDegTwoRelationsSupZero} with the definition of $\mathscr{W}_n^\mathrm{max}$, given by \eqref{eq:MaximalProlongationDegn}, and Lemma \ref{PropMonEquivalenceSupZero} we get
        \begin{align*}
            \coinv{C}(\mathscr{W}_n^\mathrm{max}) &= \coinv{C}\left(\bigcap_{i+j+2=n}(U\otimes_H T)^{\cotimes{C}i}\cotimes{C}(U\otimes_H \check{R})\cotimes{C}(U\otimes_H T)^{\cotimes{C}j}\right)\\
            &= \bigcap_{i+j+2=n}\coinv{C}\left((U\otimes_H T)^{\cotimes{C}i}\cotimes{C}(U\otimes_H \check{R})\cotimes{C}(U\otimes_H T)^{\cotimes{C}j}\right)\\
            &\cong \bigcap_{i+j+2=n}T^{\otimes i}\otimes \check{R}\otimes T^{\otimes j}\\
            &= \mathfrak{C}^{(n)}(T,\check{R}).
        \end{align*}
        This shows that ${}^{\text{co C}}\mathscr{W}_\bullet^\mathrm{max}\cong \mathfrak{C}(T,\check{R})$ as objects in ${}_H\mathcal{M}^C$, and thus by Theorem \ref{ThmSchneiderVariationI} $\mathscr{W}_\bullet^\mathrm{max} \cong U\otimes_H \mathfrak{C}(T,\check{R})$ as objects in ${}_U^C\mathcal{M}^C$. Lastly, this isomorphism is also an isomorphism of coalgebras, since the comultiplication on $\coinv{C}\mathscr{W}_\bullet^\mathrm{max}$ induced by $\mathscr{W}_\bullet^\mathrm{max}$ is simply given by deconcatenation of tensors, which agrees with the comultiplication on $\mathfrak{C}(T,\check{R})$.
    \end{proof}

    \section{Relation to differential calculi and Cartan pairs}\label{SecRelFOCCFODCFOVC}

    In this section we discuss how first-order codifferential calculi relate to other types of algebraically defined differential structures.

    \subsection{Correspondence between first-order differential- and codifferential calculi}

    Let $C$ be a coalgebra and fix a subalgebra $A\subseteq C^\ast$ of the dual algebra of $C$. Recall that every first-order differential calculus on $A$ is a quotient of the universal one $\Omega_u^1$, which is given as the kernel of the multiplication map $\mu\colon A\otimes A \rightarrow A$ with differential given by $d_u(a) = 1\otimes a - a\otimes 1$ \cite{Wor89}. We use this observation to define a first-order differential calculus on $A$ in terms of a subbimodule of $\Omega_u^1$ determined by a first-order codifferential calculus on $C$.

    \begin{lemma}
        \label{LemFOCCtoFODC}
        Let $\mathscr{W}_1$ be a first-order codifferential calculus on $C$, then the subspace
        \[
            N^{(1)} := \mathrm{span}_k\{ad_u(b) \mid a(w_{(-1)})b(\delta(w_{(0)})) = 0 \text{ for all }w\in \mathscr{W}_1\}
        \]
        is sub bimodule. In particular, we obtain a first-order differential calculus $\Omega^1 := \Omega_u^1/N^{(1)}$, determined by $\mathscr{W}_1$.
    \end{lemma}

    \begin{proof}
        Let $ad_u(b) \in N^{(1)}, c \in A$ and $w\in \mathscr{W}_1$ then
        \begin{align*}
            (ca)(w_{(-1)})b(\delta(w_{(0)})) &= c(w_{(-2)})a(w_{(-1)})b(\delta(w_{(0)})) = 0
        \end{align*}
        and thus, $cad_u(b) \in N^{(1)}$. Observe that, since $d_u$ is a derivation
        \[
            ad_u(b)c = ad_u(bc) - abd_u(c)
        \]
        and since $\delta$ is a coderivation
        \begin{align*}
            &\phantom{=} a(w_{(-1)})(bc)(\delta(w_{(0)})) - (ab)(w_{(-1)})c(\delta(w_{(0)}))\\
            &= a(w_{(-1)})b(\delta(w_{(0)})_{(1)})c(\delta(w_{0})_{(2)}) - a(w_{(-2)})b(w_{(-1)})c(\delta(w_{(0)}))\\
            &= a(w_{(-2)})b(w_{(-1)})c(\delta(w_{(0)})) + a(w_{(-1)})b(\delta(w_{(0)}))c(w_{(1)}) - a(w_{(-2)})b(w_{(-1)})c(\delta(w_{(0)}))\\
            &= 0.
        \end{align*}
        Thus, also $ad_u(b)c\in N^{(1)}$.
    \end{proof}

    The resulting first-order differential calculus $\Omega^1 = \Omega_u^1/N^{(1)}$ pairs in a canonical way with $\mathscr{W}_1$:
    \[
        \langle-,-\rangle\colon \Omega^1\times \mathscr{W}_1\rightarrow k, \quad \langle ad(b), w\rangle = a(w_{(-1)})b(\delta(w_{(0)}))
    \]
    where $d$ denotes the differential of $\Omega^1$.
    Moreover, this pairing has the following properties
    \begin{equation}\label{eq:fodcpairing}
        \langle a\omega, w\rangle = a(w_{(-1)})\langle \omega, w_{(0)}\rangle, \quad \langle \omega a,w\rangle = \langle \omega, w_{(0)}\rangle a(w_{(1)}), \quad \langle d(a),w\rangle = a(\delta(w)).
    \end{equation}

    \begin{definition}
        \label{DefBimodPairing}
        Let $M\in {}^C\mathcal{M}^C$ and $\mathscr{F} \in {}_A\mathcal{M}_A$. We call pairing $\langle -, -,\rangle\colon\mathcal{F}\times M \rightarrow k$ a \textit{bimodule pairing} if it satisfies the first two identities in \eqref{eq:fodcpairing}.
    \end{definition}

    By considering the pairing between the corresponding universal  first-order calculi, we can also pass from first-order differential calculi to first-order codifferential calculi.

    \begin{proposition}
        \label{PropFODCtoFOCC}
        Assume that $A$ separates the elements of $C$ (for instance, if $A = C^\ast$). Let $N^{(1)}\subseteq \Omega_u^1$ be a sub bicomodule. Then $\mathscr{W}_1 := (N^{(1)})^\perp$ is a sub-bicomodule of $\mathscr{W}_1^u$. In particular, it is a first-order codifferential calculus with codifferential given by the restriction of $\delta^u$.
    \end{proposition}

    \begin{proof}
        Let $w \in \mathscr{W}_1$. We have to show that $w_{(-1)}\otimes w_{(0)} \in C\otimes \mathscr{W}_1$ and $w_{(0)}\otimes w_{(1)}\in \mathscr{W}_1\otimes C$. Since $A$ separates the elements of $C$ by assumption, this is equivalent to $c(w_{(-1)})w_{(0)}, c(w_{(1)})w_{(0)} \in \mathscr{W}_1$ for all $c\in A$. Indeed, for any $c\in A$ and $\sum_i a_id_u(b_i)\in N^{(1)}$, we have
        \begin{align*}
            \sum_i\langle c(w_{(-1)})a_id_u(b_i), w_{(0)}\rangle &= \sum_i c(w_{(-2)})a_i(w_{(-1)})b_i(\delta(w_{(0)}))\\
            &= \sum_i (ca_i)(w_{(-1)})b_i(\delta(w_{(0)}))\\
            &= \sum_i\langle ca_id_u(b_i),w\rangle\\
            &= 0 
        \end{align*}
        using that $N^{(1)}$ is a sub bimodule. Moreover, using that $\delta$ is a coderivation
        \begin{align*}
            \sum_i\langle c(w_{(1)})a_id_u(b_i), w_{(0)}\rangle &= \sum_i a_i(w_{(-1)})b_i(\delta(w_{(0)}))c(w_{(1)})\\
            &= \sum_i (a_i(w_{(-1)})b_i(\delta(w_{(0)})_{(1)})c(\delta(w_{(0)})_{(2)})\\
            &\phantom{== \sum_i} - a_i(w_{(-2)})b_i(w_{(-1)})c(\delta(w_{(0)})))\\
            &= \sum_i \langle a_id_u(b_ic) - a_ib_id_u(c), w\rangle\\
            &= \sum_i\langle a_id_u(b_i)c, w\rangle\\
            &= 0
        \end{align*}
    \end{proof}

    As a consequence of Proposition \ref{PropFODCtoFOCC}, we can make the following observation: First note the categories of first-order (co)differential calculi are equivalent to the poset categories of sub $A$-bimodules, respectively sub $C$-bicomodules. Next, if we assume that $A$ separates the elements of $C$, then taking orthogonal complements gives functors between those poset categories, that form a Galois connection (see \cite[IV.5 Theorem 1]{Mac98}). In particular, the constructions in Lemma \ref{LemFOCCtoFODC} and Proposition \ref{PropFODCtoFOCC} give rise to a pair of adjoint functors.

    \subsection{The codifferential prolongation}
    
    We fix a coalgebra $C$, a subalgebra $A\subseteq C^\ast$, and a first-order codifferential calculus $\mathscr{W}_1$. We let $\Omega^1 := \Omega_u^1/N^{(1)}$ denote the first-order differential calculus corresponding to $\mathscr{W}_1$ as in Lemma \ref{LemFOCCtoFODC}.

    In this section, we present a construction which extends the first-order differential calculus $\Omega^1$ to a full differential calculus, incorporating the maximal prolongation of $\mathscr{W}_1$. Firstly, note that we have a canonical bimodule pairing between $(\Omega^1)^{\otimes_A n}$ and $(\mathscr{W}_1)^{\cotimes{C} n}$ for all $n\ge 0$. For $n = 0$, we use the pairing $A\times C \rightarrow k, \langle a, c\rangle = a(c)$, and for $n\ge 2$, the pairing is constructed using the following:

    \begin{lemma}
        \label{LemTensorPairing}
        Let $\mathcal{F}, \mathcal{G}\in {}_A\mathcal{M}_A, M, N \in {}^C\mathcal{M}^C$, and let
        \[
            \langle-,-\rangle_1\colon \mathcal{F}\times M\rightarrow k, \quad \langle-,-\rangle_2\colon \mathcal{G}\times N\rightarrow k
        \]
        be bimodule pairings. Then
        \[
            \langle-,-\rangle\colon (\mathcal{F}\otimes_A \mathcal{G})\times (M\cotimes{C}N)\rightarrow k, \quad (f\otimes_A g, m\cotimes{C} n)\mapsto \langle f, m\rangle_1\langle g,n\rangle_2.
        \]
        Is a bimodule pairing.
    \end{lemma}

    \begin{proof}
        For any $f \in \mathcal{F}, g \in \mathcal{G}$,  $m\cotimes{C} n \in M\cotimes{C}N$ and $a\in A$, we have by the properties of cotensor products, and \eqref{eq:fodcpairing}
        \begin{align*}
            \langle fa,m\rangle_1\langle g,n\rangle_2 &= \langle f, m_{(0)}\rangle_1 a(m_{(1)})\langle g,n\rangle_2 = \langle f, m\rangle_1 a(n_{(-1)})\langle g,n_{(0)}\rangle_2\\
            &= \langle f,m\rangle_1\langle a g,n\rangle_2.     
        \end{align*}
        This shows that the pairing is well-defined. The fact that it is a bimodule pairing follows directly from the definition of the $A$-actions on $\mathcal{F}\otimes_A \mathcal{G}$, the $C$-coactions on $M\cotimes{C}N$.
    \end{proof}

    We use the pairing defined in Lemma \ref{LemTensorPairing} to define an ideal of $T_A(\Omega^1)$. For each $n\ge 2$ consider the following subbimodule of $(\Omega^1)^{\otimes_A n}$ (the fact that the space defined below is indeed a subbimodule follows since the given pairing is a bimodule pairing)
    \[
        N_\text{codiff}^{(n)} := (\mathscr{W}_n^\mathrm{max})^\perp.
    \]
    If we let $\Omega_\mathrm{codiff}^{(n)} := (\Omega^1)^{\otimes_A n}/N_\mathrm{codiff}^{(n)}$, we get an induced pairing
    \begin{equation}\label{eq:IndTensorPairing}
        \Omega_\mathrm{codiff}^{(n)}\times \mathscr{W}_n^\mathrm{max} \rightarrow k.
    \end{equation}
    Note that with this choice, $\mathscr{W}_n^\mathrm{max}$ separates the elements of $\Omega_\mathrm{codiff}^{(n)}$.

    \begin{theorem}
        \label{ThmCodifferentialProlongation}
        Let $N_\mathrm{codiff}^\bullet := \bigoplus_{n\ge 2} N_\mathrm{codiff}^{(n)}$. Then $N_\mathrm{codiff}^\bullet$ is a two-sided ideal of $T_A(\Omega^1)$ and the resulting quotient algebra $\Omega_\mathrm{codiff}^\bullet$ is a differential calculus over $A$ with differential determined by the equation
        \begin{equation}\label{eq:DiffPairing}
            \langle d(\omega), w\rangle = \langle \omega,\delta(w)\rangle
        \end{equation}
        Where $\omega \in (\Omega^1)^{\otimes_A n}/N_\mathrm{codiff}^{(n)}$ and $w \in \mathscr{W}_n^\mathrm{max}$ with $n\ge 0$.
        Moreover, the extension of the pairing \eqref{eq:IndTensorPairing} $\Omega_\mathrm{codiff}^\bullet\times \mathscr{W}_\bullet^\mathrm{max} \rightarrow k$ is compatible with the algebra and coalgebra structures, in the sense that
        \begin{equation}\label{eq:AlgCoalgPairing}
            \langle \omega\eta, w\rangle = \langle\omega,w_{(1)}\rangle\langle \eta,w_{(2)}\rangle
        \end{equation}
        for all $\omega, \eta \in \Omega_\mathrm{codiff}^\bullet$ and $w\in \mathscr{W}_\bullet^\mathrm{max}$.
    \end{theorem}

    \begin{proof}
        We first argue that \eqref{eq:AlgCoalgPairing} holds for the pairing between $T_A(\Omega^1)$ and $T_C^c(\mathscr{W}_1)$ induced by Lemma \ref{LemTensorPairing}. Let $\omega \in (\Omega^1)^{\otimes_A m}, \eta \in (\Omega^1)^{\otimes_A n}$, and $w \in \mathscr{W}_1^{\cotimes{C} l}$ for $m,n,l \ge 0$. Note that \eqref{eq:AlgCoalgPairing} holds if $l\neq m+n$. Indeed, in this case both sides of the equation are zero. For the right-hand side this follows directly, since the pairing is only non-zero for elements of the same degree, and for the left-hand side recall that the coproduct of $w$ can be written as
        \[
            w_{(1)}\otimes w_{(2)} = \sum_{i=0}^lw_{(1)}^{l-i}\otimes w_{(2)}^i
        \]
        where $w_{(1)}^{l-i}\otimes w_{(2)}^i \in \mathscr{W}_1^{\cotimes{C} l-i}\otimes \mathscr{W}_1^{\cotimes{C} i}$ for $0\le i \le l$. Thus, let $l = m+n$ if $m = 0$, then $\omega = a$ for some $a\in A$ and since $\langle-,-\rangle$ is a bimodule pairing
        \[
            \langle a\eta, w\rangle = \langle a, w_{(-1)}\rangle \langle \eta, w_{(0)}\rangle = \langle a, w_{(1)}^0\rangle \langle \eta,w_{(2)}^{n}\rangle = \langle a,w_{(1)}\rangle \langle \eta, w_{(2)}\rangle.
        \]
        By symmetry, we also get \eqref{eq:AlgCoalgPairing} if $n = 0$. Thus let $m,n > 0$. We can assume without loss of generality, that $\omega = \omega_1\otimes_A \dots \otimes_A \omega_m$ and $\eta = \omega_{m+1}\otimes_A\dots\otimes_A \omega_{m+n}$, and if we write $w = w_1\cotimes{C}\dots\cotimes{C} w_{m+n}$, we get
        \begin{align*}
            \langle \omega\eta,w\rangle &= \langle \omega_1, w_1\rangle \dots \langle\omega_{m+n}, w_{m+n}\rangle = \langle \omega, w_1\cotimes{C} \dots \cotimes{C} w_m\rangle\langle \eta, w_{m+1}\cotimes{C}\dots \cotimes{C} w_{m+n}\rangle\\
            &= \langle \omega, w_{(1)}^m\rangle\langle \eta, w_{(2)}^n\rangle\\
            &= \langle \omega, w_{(1)}\rangle\langle \eta, w_{(2)}\rangle.
        \end{align*}

        Next, we show that $N_\mathrm{codiff}^\bullet$ is a two-sided ideal. Indeed, by construction of $\mathscr{W}_\bullet^\mathrm{max}$ (see \eqref{eq:MaximalProlongationDegn}), we have that $\mathscr{W}_{n+1} \subseteq \mathscr{W}_1\cotimes{C} \mathscr{W}_n^\mathrm{max}, \mathscr{W}_n^\mathrm{max}\cotimes{C}\mathscr{W}_1$ for all $n\ge 2$. Thus, $N_\mathrm{codiff}^\bullet$ is closed under left- and right multiplication with generators of $T_A(\Omega^1)$, and therefore it is a two-sided ideal. Note also, that \eqref{eq:AlgCoalgPairing} still holds for the induced pairing between $\Omega_\mathrm{codiff}^\bullet$ and $\mathscr{W}_\bullet^\mathrm{max}$.
        
        We are left with showing that $\Omega_\mathrm{codiff}^\bullet$ admits a differential. We will argue as follows: First we show that $N_\mathrm{max}^{(2)} \subseteq N_\mathrm{codiff}^{(2)}$, where
        \[
            N_\mathrm{max}^{(2)} = \mathrm{span}_{A\otimes A^{op}}\{d(a)\otimes_A d(b) \mid ad_u(b) \in N^{(1)}\}
        \]
        This will imply that $\Omega_\mathrm{codiff}^\bullet$ is a quotient algebra of the maximal prolongation of $\Omega^1$ which is given by $\Omega_\mathrm{max}^\bullet = T_A(\Omega^1)/\langle N_\mathrm{max}^{(2)}\rangle$ (see for instance \cite[Lemma 1.32]{BM20}). We will then show that the image of $N_\mathrm{codiff}^\bullet$ in $\Omega_\mathrm{max}^\bullet$ is a differential ideal, and that the induced differential on $\Omega_\mathrm{codiff}^\bullet$ satisfies \eqref{eq:DiffPairing}. Let $ad_u(b) \in N^{(1)}$ and $w_1\cotimes{C} w_2 \in \mathscr{W}_2^\mathrm{max}$. Then by definition of $\mathscr{W}_2^\mathrm{max}$, we have $[\delta(w_1)\otimes \delta(w_1)] = [w_{(-1)}\otimes \delta(w_{(0)})]$ for some $w\in \mathscr{W}_1$, and thus
        \begin{align*}
            \langle d(a)\otimes_A d(b), w_1\cotimes{C} w_2\rangle &= \langle d(a), w_1\rangle \langle d(b),w_2\rangle = \langle ad_u(b), [\delta(w_1)\otimes \delta(w_2)]\rangle\\
            &= \langle ad_u(b), [w_{(-1)}\otimes \delta(w_{(0)})]\rangle\\
            &= a(w_{(-1)})b(\delta(w_{(0)}))\\
            &= 0
        \end{align*}
        Where the last equality holds by definition of $N^{(1)}$ (see Lemma \ref{LemFOCCtoFODC}). It follows that $\Omega_\mathrm{codiff}^\bullet$ is a quotient of $\Omega_\mathrm{max}^\bullet$. In particular, we get a well-defined pairing $\Omega_\mathrm{max}^\bullet\times \mathscr{W}_\bullet^\mathrm{max}$ which moreover satisfies \eqref{eq:AlgCoalgPairing}. We show that this pairing satisfies \eqref{eq:DiffPairing} by induction on the degree $n$ of a homogeneous element $\omega \in \Omega_\mathrm{max}^n$. For $n = 0$, this follows from \eqref{eq:fodcpairing}. For $n = 1$, we have for all $ad(b) \in \Omega^1, w_1\cotimes{C} w_2 \in \mathscr{W}_2^\mathrm{max}$ and with $w := \delta_2(w_1\cotimes{C} w_2)$
        \begin{align*}
            \langle d(ad(b)),w_1\cotimes{C}w_2\rangle &= \langle d(a)\wedge d(b), w_1\cotimes{C}w_2\rangle = \langle d(a),w_1\rangle \langle d(b),w_2\rangle\\
            &= a(\delta(w_1))b(\delta(w_2))\\
            &= \langle ad_u(b), [\delta(w_1)\otimes \delta(w_2)]\rangle\\
            &= \langle ad_u(b), [w_{(-1)}\otimes \delta(w_{(0)})]\rangle\\
            &= a(w_{(-1)})b(\delta(w_{(0)}))\\
            &= \langle ad(b), \delta(w_1\cotimes{C} w_2)\rangle
        \end{align*}
        Where we denoted by $-\wedge -$ the product of $\Omega_\mathrm{max}^\bullet$. Now let $n\ge 2$ and let $\omega \in \Omega^1, \eta\in \Omega_\mathrm{max}^{n-1}$ and $w \in \mathscr{W}_{n+1}$. Then using induction and \eqref{eq:dgcoalgII} we get
        \begin{align*}
            \langle d(\omega\wedge \eta),w \rangle &= \langle d(\omega)\wedge \eta, w\rangle - \langle \omega \wedge d(\eta), w\rangle\\
            &= \langle d(\omega), w_{(1)}\rangle \langle \eta, w_{(2)}\rangle - \langle \omega, w_{(1)}\rangle\langle d(\eta),w_{(2)}\rangle\\
            &= \langle d(\omega), w_{(1)}^2\rangle\langle\eta, w_{(2)}^{n-1}\rangle - \langle \omega, w_{(1)}^1\rangle \langle d(\eta), w_{(2)}^{n}\rangle\\
            &= \langle \omega, \delta(w_{(1)}^2)\rangle\langle \eta, w_{(2)}^{n-1}\rangle - \langle \omega, w_{(1)}^1\rangle\langle \eta, \delta(w_{(2)}^n)\rangle\\
            &= \langle \omega, \delta(w)_{(1)}^1\rangle\langle\eta,\delta(w)_{(2)}^{n-1}\rangle\\
            &= \langle \omega\wedge \eta, \delta(w)\rangle.
        \end{align*}
        From the assertion that \eqref{eq:DiffPairing} holds for $\Omega_\mathrm{max}^\bullet$ and $\mathscr{W}_\bullet^\mathrm{max}$, it follows immediately that the differential on $\Omega_\mathrm{max}^\bullet$ descends to $\Omega_\mathrm{codiff}^\bullet$ and that \eqref{eq:DiffPairing} also holds for $\Omega_\mathrm{codiff}^\bullet$.
    \end{proof}

    \begin{definition}
        \label{DefCodifferentialProlongation}
        We call the differential calculus $\Omega_\mathrm{codiff}^\bullet$ constructed in Theorem \ref{ThmCodifferentialProlongation} the \textit{codifferential prolongation} of $\Omega^1$.
    \end{definition}

    \subsection{Differential calculi over quantum homogeneous spaces arising as finitary duals of coalgebras}

    It is natural to ask whether the codifferential prolongation of a first-order differential calculus coming from a first-order codifferential calculus agrees with the maximal prolongation. In this section we will give a sufficient condition in the framework of quantum homogeneous spaces.

    We let $U$ be a Hopf algebra, $H\subseteq U$ a right coideal subalgebra, $C := U\otimes_H k$, as well as $A := U_\mathcal{C}^\circ$ for a tensor category $\mathcal{C}$ of finite dimensional $U$-modules which is closed under taking duals.

    \begin{lemma}
        Let $\mathscr{W}_1$ be a $U$-equivariant first-order codifferential calculus over $C$, then the pairing
        \[
            \Omega_u^1\times \mathscr{W}_1\rightarrow k, \quad \langle ad_u(b),w\rangle = a(w_{(-1)})b(\delta(w_{(0)}))
        \]
        where $\Omega_u^1$ is the universal first-order differential calculus over $B$, is admissible (see Definition \ref{DefAdmissiblePairings}). Moreover, the first-order differential calculus $\Omega^1 = \Omega_u^1/\mathscr{W}_1^\perp$ is $A$-covariant.
    \end{lemma}

    \begin{proof}
        First note that the universal first-order differential calculus is $\Omega_u^1$ is $A$-covariant. In particular, we have $\Omega_u^1 \in {}_B^A\mathcal{M}_B$. We have already seen that the above is a bimodule pairing \eqref{eq:fodcpairing}. Let $X\in U, a,b \in B$ and $w \in \mathscr{W}_1$, then 
        \begin{align*}
            \langle ad_u(b),Xw\rangle &= a((Xw)_{(-1)})b(\delta((Xw)_{(0)})) = a(X_{(1)}w_{(-1)})b(X_{(2)}\delta(w_{(0)}))\\
            &= a_{(1)}(X_{(1)})b_{(1)}(X_{(2)})a_{(2)}(w_{(-1)})b_{(2)}(\delta(w_{(0)}))\\
            &= a_{(1)}b_{(1)}(X)\langle a_{(2)}d_u(b_{(2)})w\rangle.
        \end{align*}
        This shows admissibility, recalling that the left $A$-coaction on $\Omega_u^1$ can be written as $ad_u(b)\mapsto a_{(1)}b_{(1)}\otimes a_{(2)}d_u(b_{(2)})$. By Lemma \ref{LemStabilityUnderOrthogonalComplements} the orthogonal complement of $\mathscr{W}_1$ is a subobject of $\Omega_u^1$ and thus $\Omega^1 \in {}_B^A\mathcal{M}_B$ as a quotient object.
    \end{proof}

    In order to proceed, we use the classification results for first-order (co)differential calculi. On this note we also assume for the rest of this section that $U$ is faithfully flat as a left and right $H$-module, and that $A$ is faithfully flat as a right $B$-module (a criterion for the latter is given in Theorem \ref{ThmMuellerSchneider}). Recall that by \cite[Theorem 2]{Her02} any $A$-covariant first-order differential calculus $\Omega^1$ is isomorphic to one of the form $A\cotimes{\pi_B} B^+/I$ for some subobject $I\subseteq B^+$, where $B^+$ is an object in ${}^{\pi_B}\mathcal{M}_B$ via right multiplication and $\pi_B(A)$-coaction given by
    \[
        b\mapsto \pi_B(b_{(1)})\otimes b_{(2)}^+.
    \]
    The differential of $A\cotimes{\pi_B} B^+/I$ is given by
    \[
        B \rightarrow A\cotimes{\pi_B} B^+/I, \quad b\mapsto b_{(1)}\otimes p(b_{(2)}^+)
    \]
    where $p\colon B^+\rightarrow B^+/I$ denotes the canonical projection. For more details on the construction consult \cite{Her02}.
    For the universal first-order differential calculus $\Omega_u^1$, we have an explicit isomorphism
    \[
        A\cotimes{\pi_B} B^+\rightarrow  \Omega_u^1, \quad a\cotimes{\pi_B}b \mapsto aS(b_{(1)})d_u(b_{(2)}).
    \]
    Indeed, the above map is an isomorphism of first-order differential calculus, which follows by checking that it is inverse to the canonical map
    \[
        \Omega_u^1 \rightarrow A\cotimes{\pi_B} B^+, \quad ad_u(b)\mapsto ad(b) = ab_{(1)}\otimes b_{(2)}^+.
    \]
    Under this identification, the pairing between $\Omega_u^1$ and a $U$-equivariant first-order codifferential calculus $\mathscr{W}_1$, such that $\mathscr{W}_1 \cong U\otimes_H T$ for some quantum tangent space $T\subseteq C^+$ takes the form
    \begin{align*}
        \langle a\cotimes{\pi_B} b, X\otimes_H t\rangle &= \langle aS(b_{(1)})d_u(b_{(2)}), X\otimes_H t\rangle = (aS(b_{(1)}))(X_{(1)})b_{(2)}(X_{(2)}t)\\
        &= a(X_{(1)})(S(b_{(1)}))(X_{(2)})b_{(2)}(X_{(3)}t) = a(X_{(1)})b_{(1)}(S(X_{(2)}))b_{(2)}(X_{(3)}t)\\
        &= a(X_{(1)})b(S(X_{(2)})X_{(3)}t) = a(X)b(t)
    \end{align*}
    where $a\cotimes{\pi_B}b\in A\cotimes{\pi_B}B^+, X \in U$ and $t\in T$. Thus, the pairing $A\cotimes{\pi_B} B^+\times U\otimes_H T \rightarrow k$ is induced by the pairing $B^+\otimes T \rightarrow k$ in the sense of Proposition \ref{PropTranslAdmissiblePairings}. Using Lemma \ref{LemOrthogonalComplements} we can write $\mathscr{W}_1^\perp = A\cotimes{\pi_B}I^{(1)}$ with $I^{(1)} := T^\perp$, and the induced first-order differential calculus takes the form $\Omega^1 = A\cotimes{\pi_B} B^+/I^{(1)}$.

    The following gives a criterion for when it is possible to recover $T$ from $I^{(1)}$.

    \begin{lemma}
        \label{LemTangentSpaceFromIdeal}
        In the situation above, assume that $B$ separates the elements of $C$, and that the quantum tangent space $T$ is finite dimensional. Then $T = (I^{(1)})^\perp$, where the orthogonal complement is taken with respect to the extended pairing $B^+\times C^+ \rightarrow k$. In particular
        \[
            B^+/I^{(1)}\times T \rightarrow k, \quad \langle b + I^{(1)}, t\rangle = b(t)
        \]
        is a non-degenerate pairing between finite dimensional vector space.
    \end{lemma}
    
    \begin{proof}
        Let $t_1, \dots, t_n$ be a basis of $T$. Since $B$ separates the points of $C$, we can find $b^1, \dots, b^N \in B^+$ such that $b^i(t_j) = \delta_{ij}$ for all $1\le i,j\le N$. Note that $b^1 + I^{(1)},\dots, b^N + I^{(1)}$ is a dual basis of $t_1,\dots, t_N$ with respect to the pairing $B^+/I^{(1)}\times T\rightarrow k$. Suppose there is some element $s \in (I^{(1)})^\perp\setminus T$. Then the set $t_1,\dots, t_N, s$ is linear independent, and thus
        \[
            \tilde{s} := s - \sum_{i = 1}^Nb^i(s)t_i
        \]
        is non-zero. However, we still have $\tilde{s} \in (I^{(1)})^\perp$, and also
        \[
            b^j(\tilde{s}) = b^j(s) - b^j(s) = 0
        \]
        for all $1\le j \le N$. In other words, the element $\tilde{s}$ is annihilated by all $b \in B^+$, which contradicts the assumption that $B$ separates the points of $C$.
    \end{proof}

    If we are in the case where $H$ is a sub bialgebra of $U$ and $T\in {}_H\mathcal{M}^0$ (see the discussion around Proposition \ref{PropMonEquivalenceSupZero}), then the codifferential prolongation of $\Omega^1$ indeed agrees with its maximal prolongation.
    
    \begin{theorem}
        \label{ThmCodiffProlongationisMaximalProlongation}
        Assume that $H$ is a sub bialgebra of $U$, that the antipode of $A$ is invertible, and that $B$ separates the elements of $C$. Let $T\subseteq C^+$ be a finite dimensional quantum tangent space such that $T\in {}_H\mathcal{M}^0$, let $\mathscr{W}_1 := U\otimes_H T$ be the corresponding first-order codifferential calculus, and let $\Omega^1 = \Omega_u^1/\mathscr{W}_1^\perp$ be the corresponding first-order differential calculus. Then we have $\Omega_\mathrm{max}^\bullet = \Omega_\mathrm{codiff}^\bullet$.
    \end{theorem}

    \begin{proof}
        Recall that the codifferential prolongation of $\Omega^1$ is given by
        \[
            \Omega_\mathrm{codiff}^\bullet = B\oplus \Omega^1 \oplus \bigoplus_{n\ge 2}(\Omega^1)^{\otimes_B n}/N_\mathrm{codiff}^{(n)}
        \]
        where $N_\mathrm{codiff}^{(n)}$ is the orthogonal complement of $\mathscr{W}_n^\mathrm{max}$ with respect to the pairing
        \[
            (\Omega^1)^{\otimes_B n}\times (\mathscr{W}_1)^{\cotimes{C}n} \rightarrow k.
        \]
        Thus, we have to check that
        \[
            N_\mathrm{codiff}^{(n)} = N_\mathrm{max}^{(n)} = \sum_{i+j+2 = n} (\Omega^1)^{\otimes_B i} \otimes_B N_\mathrm{max}^{(2)}\otimes_B (\Omega^1)^{\otimes_B j}
        \]
        for all $n\ge 2$. Since $T \in {}_H\mathcal{M}^0$, we also have $V := B^+/I^{(1)} \in {}^{\pi_B}\mathcal{M}_0$, and we can apply the monoidal equivalences between $({}_U^C\mathcal{M}^0,\cotimes{C},C)$ and $({}_H\mathcal{M}^0,\otimes, k)$, as well as $({}_B^A\mathcal{M}_0, \otimes_B, B)$ and $({}^{\pi_B}\mathcal{M}_0, \otimes, k)$ (see Proposition \ref{PropMonoidalTakeuchi} and Proposition \ref{PropMonEquivalenceSupZero}). In particular, since $\Omega^1 \cong A\cotimes{\pi_B} V$ we have identifications $(\Omega^1)^{\otimes_B n} \cong (A\cotimes{\pi_B} V)^{\otimes_B n} \cong A\cotimes{\pi_B} V^{\otimes n}$ and $\mathscr{W}_1^{\cotimes{C} n} \cong U\otimes_H T^{\otimes n}$. Under these identifications, the induced pairing $\langle-,-\rangle^\flat\colon V^{\otimes n}\times T^{\otimes n}\rightarrow k$ (in the sense of Proposition \ref{PropTranslAdmissiblePairings}) can be written as
        \[
            \langle v^1\otimes \dots \otimes v^n, t_1\otimes \dots\otimes t_n\rangle^\flat = \langle v^1, t_1\rangle \dots \langle v^n,t_n\rangle.
        \]
        Indeed, using the strong monoidal structure of $\Phi$ (Proposition \ref{PropMonoidalTakeuchi}) and the counit of the adjunction $\Phi \dashv \Psi$ (Theorem \ref{ThmTakeuchi}), we obtain an isomorphism
        \begin{gather*}
            \Phi((A\cotimes{\pi_B}V)^{\otimes_B n})\rightarrow V^{\otimes n},\\ (a^1\cotimes{\pi_B} v^1)\otimes_B \dots \otimes_B (a^n\cotimes{\pi_B} v^n) + B^+(A\cotimes{\pi_B} V) \mapsto \varepsilon(a^1)v^1\otimes \dots \otimes\varepsilon(a^n)v^n
        \end{gather*}
        and if we write $v^i = b^i + I^{(1)}$ for suitable $b^i \in B^+$, then the preimage of $v^1\otimes \dots \otimes v^n$ is given by $(b_{(1)}^1\cotimes{\pi_B}((b_{(2)}^1)^+ + I^{(1)}))\otimes_B \dots \otimes_B (b_{(1)}\cotimes{\pi_B} ((b_{(2)}^n)^+ + I^{(1)}))$. It follows now
        \begin{align*}
            &\phantom{=} \langle (b_{(1)}^1\cotimes{\pi_B}((b_{(2)}^1)^+ + I^{(1)}))\otimes_B \dots \otimes_B (b_{(1)}\cotimes{\pi_B} ((b_{(2)}^n)^+ + I^{(1)})),\\
            &\hspace{9cm} (1\otimes_H t_1)\cotimes{C} \dots \cotimes{C} (1\otimes_H t_n)\rangle\\
            &= \langle b_{(1)}^1\cotimes{\pi_B} ((b_{(2)}^1)^+ + I^{(1)}), 1\otimes_H t_1\rangle \dots \langle b_{(1)}^n\cotimes{\pi_B} ((b_{(2)}^n)^+ + I^{(1)}),1\otimes_H t_n\rangle\\
            &= b_{(1)}^1(1)\langle (b_{(2)}^1)^+ + I^{(1)}, t_1\rangle \dots b_{(1)}^n(1)\langle (b_{(2)}^n)^+ + I^{(1)}, t_n\rangle\\
            &=  \langle v^1, t_1\rangle \dots \langle v^n,t_n\rangle.
        \end{align*}
        In particular, by Lemma \ref{LemTangentSpaceFromIdeal}, the pairing between $V^{\otimes n}$ and $T^{\otimes n}$ is non-degenerate. We proceed by showing that $N_\mathrm{max}^{(2)} = N_\mathrm{codiff}^{(2)}$. Let $w_1\cotimes{C} w_2 \in (N_\mathrm{max}^{(2)})^\perp$, then for every $ad_u(b) \in N^{(1)}$
        \begin{align*}
            \langle ad_u(b), [\delta(w_1)\otimes \delta(w_2)]\rangle &= a(\delta(w_1))b(\delta(w_2)) = \langle d(a),w_1\rangle\langle d(b),w_2\rangle\\
            &= \langle d(a)\otimes_Bd(b),w_1\cotimes{C} w_2\rangle\\
            &= 0.
        \end{align*}
        Thus, we have $[\delta(w_1)\otimes \delta(w_2)] \in (N^{(1)})^\perp = \mathscr{W}_1^{\perp\perp}$, where orthogonal complements are taken with respect to the pairing $\Omega_u^1\times \mathscr{W}_1^u\rightarrow k$. By identifying $\Omega_u^1$ with $A\cotimes{\pi_B} B^+$ we get, using Lemma \ref{LemTangentSpaceFromIdeal} and Lemma \ref{LemStabilityUnderOrthogonalComplements}
        \[
            \mathscr{W}_1^{\perp\perp} = (A\cotimes{\pi_B}I^{(1)})^\perp = U\otimes_H (I^{(1)})^\perp = U\otimes_H T.
        \]
        This implies that $[\delta(w_1)\otimes \delta(w_2)] \in \mathscr{W}_1$ and thus $w_1\cotimes{C} w_2 \in \mathscr{W}_2^\mathrm{max}$. Since $N_\mathrm{max}^{(2)} \subseteq N_\mathrm{codiff}^{(2)}$ we also have
        \[
        (\mathscr{W}_2^\mathrm{max})^{\perp\perp} = (N_\mathrm{codiff}^{(2)})^\perp \subseteq (N_\mathrm{max}^{(2)})^\perp \subseteq \mathscr{W}_2^\mathrm{max}
        \]
        and since $\mathscr{W}_2^\mathrm{max} \subseteq (\mathscr{W}_2^\mathrm{max})^{\perp\perp}$ all the above inclusions are equalities, which intern implies $N_\mathrm{codiff}^{(2)} = N_\mathrm{max}^{(2)}$ using the fact that the induced pairing $V\otimes V\times T\otimes T\rightarrow k$ is a non-degenerate pairing between finite dimensional vector spaces and Lemma \ref{LemStabilityUnderOrthogonalComplements}. For the higher degrees, we first observe that we can identify $\mathscr{W}_2^\mathrm{max}$ with $U\otimes_H \check{R}$ for some subspace $\check{R}\subseteq T\otimes T$ and $N_\mathrm{max}^{(2)} = A\cotimes{\pi_B }R$ with $R = \check{R}^\perp$ again using Lemma \ref{LemOrthogonalComplements}. We then have for arbitrary $n \ge 2$
        \[
            \mathscr{W}_n^\mathrm{max} \cong U\otimes_H \bigcap_{i+j+2 = n} T^{\otimes i}\otimes \check{R} \otimes T^{\otimes j}.
        \]
        Thus, we can identify $N_\mathrm{codiff}^{(n)} = (\mathscr{W}_n^\mathrm{max})^\perp$ with
        \[
            A\cotimes{\pi_B } \left(\bigcap_{i+j+2 = n} T^{\otimes i}\otimes \check{R} \otimes T^{\otimes j}\right)^\perp = A\cotimes{\pi_B} \left(\sum_{i+j+2=n} V^{\otimes i}\otimes R \otimes V^{\otimes j}\right).
        \]
        The right-hand side is precisely the image of $N_\mathrm{max}^{(n)}$ under the isomorphism $(\Omega^1)^{\otimes_B n} \cong A\cotimes{\pi_B }V^{\otimes n}$, and we conclude that also $N_\mathrm{max}^{(n)} = N_\mathrm{codiff}^{(n)}$.
    \end{proof}

    The proof of the above theorem gives an explicit description of the maximal prolongation of $\Omega^1$, which conversely allows us to recover the maximal prolongation of $\mathscr{W}_1$ if we happen to know $\Omega_\mathrm{max}^\bullet$.

    \begin{corollary}
        \label{PorMaxProlongationQuadraticAlgebras}
        In the situation of Theorem \ref{ThmCodiffProlongationisMaximalProlongation} we have
        \[
            \Omega_\mathrm{max}^\bullet \cong A\cotimes{\pi_B} \mathfrak{A}(V,\check{R}^\perp).
        \]
        Here, $\check{R}$ is the subobject of $T\otimes T$ such that $\mathscr{W}_2^\mathrm{max} \cong U\otimes_H \check{R}$, and $\mathfrak{A}(V,\check{R}^\perp)$ is the quadratic dual algebra of $\mathfrak{C}(T,\check{R})$.
    \end{corollary}

    \subsection{Cartan pairs associated to first-order codifferential calculi}

    Besides first-order differential calculi there is also the notion of a Cartan pair, which serves as a replacement for vector fields in noncommutative geometry. The latter concept is due to Borowiec \cite{Bor96, Bor24}. In this section we show that a first-order codifferential calculus gives rise to a Cartan pair.

    \begin{definition}
        \label{DefVectorCalculus}
        A right \textit{Cartan pair} over a $k$-algebra $A$ is a pair $(\mathscr{X}^1, \mathscr{L})$ consisting of an $A$-bimodule $\mathscr{X}^1$ and a linear map $\mathscr{L}\colon \mathscr{X}^1\rightarrow \End_k(A), X\mapsto \mathscr{L}_X$ such that $\mathscr{L}$ is injective and
        \begin{gather}\label{eq:CPAnchorR}
            \mathscr{L}_{aX}(b) = a\mathscr{L}_X(b)\\
            \mathscr{L}_X(ab) = \mathscr{L}_X(a)b + \mathscr{L}_{Xa}(b)
        \end{gather} 
        for all $X \in \mathscr{X}^1$ and $a, b \in A$.
    \end{definition}

    Similarly, one defines left Cartan pairs, replacing \eqref{eq:CPAnchorR} by
    \begin{gather}\label{eq:CPAnchorL}
        \mathscr{L}_{Xa}(b) = \mathscr{L}_X(b)a\\
        \mathscr{L}_{X}(ab) = a\mathscr{L}_X(b) + \mathscr{L}_{bX}(a).
    \end{gather}
    In \cite{Bor96} the map $\mathscr{L}$ is referred to as an action. However, motivated by differential geometry, we find it more fitting to call the latter an anchor morphism as in anchor morphism of a Lie algebroid, or more generally, Lie--Rinehart algebra \cite{Hue21}.

    Let $C$ be a coalgebra and let $A := C^\ast$ be its dual algebra. Then a first-order codifferential calculus on $C$ gives rise to a Cartan pair as follows:

    \begin{proposition}
        \label{PropCPfromFOCC}
        Let $\mathscr{W}_1$ be a first-order codifferential calculus over $C$ with differential $\delta$. Let
        \[
            \mathscr{X}^1 := \mathrm{Hom}^C(C,\mathscr{W}_1)
        \]
        denote the space of right $C$-colinear maps from $C$ to $\mathscr{W}_1$. Define left- and right $A$-action on $\mathscr{X}^1$ via
        \begin{equation}\label{eq:CPLeftAction}
            (a\triangleright X)(c) = a(c_{(1)})X(c_{(2)})
        \end{equation}
        \begin{equation}\label{eq:CPRightAction}
            (X\triangleleft a)(c) = a(X(c)_{(-1)})X(c)_{(0)}
        \end{equation}
        for $a \in A, X \in \mathscr{W}_1$ and $c \in C$. Additionally, we define an anchor morphism $\mathscr{L}\colon \mathscr{X}^1 \rightarrow \mathrm{End}(A)$
        \begin{equation}\label{eq:CPAnchor}
            \mathscr{L}_X(a)(c) = a(\delta(X(c)))
        \end{equation}
        with $a \in A, X\in \mathscr{X}^1$ and $c \in C$. The pair $(\mathscr{X}^1, \mathscr{L})$ is a Cartan pair.
    \end{proposition}

    \begin{proof}
        We first show that \eqref{eq:CPLeftAction} and \eqref{eq:CPRightAction} are respectively left- and right actions. Indeed, let $X \in \mathscr{X}^1, a, b\in A$ and $c \in C$, then
        \[
            (1\triangleright X)(c) = \varepsilon(c_{(1)})X(c_{(2)}) = X(c)        
        \]
        \begin{align*}
            (a\triangleright(b\triangleright X))(c) = a(c_{(1)})b(c_{(2)})X(c_{(3)}) = (ab)(c_{(1)})X(c_{(2)}) = (ab\triangleright X)(c)
        \end{align*}
        and also
        \[
            (X\triangleleft 1)(c) = \varepsilon(X(c)_{(-1)})X(c)_{(0)} = X(c)
        \]
        \begin{align*}
            ((X\triangleleft a)\triangleleft b)(c) &= b((X\triangleleft a)(c)_{(-1)})(X\triangleleft a)(c)_{(0)}\\
            &= a(X(c)_{(-1)})b(X(c)_{(-1)(0)})X(c)_{(0)(0)}\\
            &= (ab)(X(c)_{(-1)})X(c)_{(0)}\\
            &= (X\triangleleft ab)(c).
        \end{align*}
        Moreover, 
        \[
            \mathscr{L}_{aX}(b)(c) = b(\delta((aX)(c))) = a(c_{(1)})b(\delta(X(c_{(2)}))) = (a\mathscr{L}_X(b))(c)
        \]
        and since $\delta$ is a coderivation and $X$ right $C$-colinear
        \begin{align*}
            \mathscr{L}_X(ab)(c) &= (ab)(\delta(X(c))) = a(\delta(X(c))_{(1)})b(\delta(X(c))_{(2)})\\
            &= a(\delta(X(c)_{(0)}))b(X(c)_{(1)}) + a(X(c)_{(-1)})b(\delta(X(c)_{(0)}))\\
            &= a(\delta(X(c_{(1)})))b(c_{(2)}) + b(\delta(a(X(c)_{(-1)})X(c)_{(0)}))\\
            &=  (\mathscr{L}_X(a)b)(c) + b(\delta((X\triangleleft a)(c)))\\
            &= (\mathscr{L}_X(a)b)(c) + \mathscr{L}_{Xa}(b)(c).
        \end{align*}
        To prove injectivity of $\mathscr{L}$, let $X\in \mathscr{X}^1$ such that $\mathscr{L}_X(a) = 0$ for all $a \in A$. We have to show that $X(c) = 0$ for all $c \in C$. By the injectivity property of a first-order codifferential calculus, it is enough to show that $\delta(X(c)_{(0)})\otimes X(c)_{(1)} = 0$ (note that since $\delta$ is a coderivation, injectivity of the map $w\mapsto w_{(-1)}\otimes \delta(w_{(0)})$ is equivalent to injectivity of $w\mapsto \delta(w_{(0)})\otimes w_{(1)}$). Indeed, for all $a,b \in A$, we have by right $C$-colinearity of $X$
        \[
            a(\delta(X(c)_{(0)}))b(X(c)_{(1)}) = a(\delta(X(c_{(1)})))b(c_{(2)}) = (\mathscr{L}_X(a)b)(c) = 0.
        \]
    \end{proof}

    \section{Codifferential calculi on quantized flag manifolds}\label{SecCodifferentialCalculionQuantizedFlags}

    In this section, we will study codifferential calculi on certain quotient module coalgebras of quantized universal enveloping algebras. As an application of the results in the previous section, we show that the antiholomorphic Heckenberger--Kolb calculi on the quantized projective spaces have classical dimension.

    \subsection{Quantized flag manifolds}\label{SecQuantizedFlagManifolds}

    We start by fixing Lie theoretic notation. For more explanation of the below, we refer to introductory texts in Lie theory \cite{Hal03,Hum72}. Let $\mathfrak{g}$ be a complex semisimple Lie algebra of rank $r$ and fix a Cartan subalgebra $\mathfrak{h}\subseteq \mathfrak{g}$. We denote by $\Delta\subseteq \mathfrak{h}^\ast$ the corresponding set of roots of $\mathfrak{g}$ and let $E := \mathrm{span}_\R\Delta$. We rescale the inner product of $E$ induced by the Killing form in such a way that $(\alpha, \alpha) = 2$ for the shortest root. Choose a set of simple roots $\Pi := \{\alpha_1, \dots, \alpha_r\}$, and write $(a_{ij})_{1\le i,j\le n}$ for the Cartan matrix as well as $d_i := \frac{(\alpha_i,\alpha_i)}{2}$. Furthermore, we consider the lattice of integral weights
    \[
        P := \{\lambda \in E \mid \forall 1\le i\le r\colon (\lambda, \alpha_i^\vee) \in \Z\}
    \]
    where $\alpha_i^\vee := 2\alpha_i/(\alpha_i,\alpha_i)$ denotes the coroot of $\alpha_i$.
    
    For $q \in \C\setminus \{-1,0,1\}$, we can define the \textit{quantized universal enveloping algebra} $U_q(\mathfrak{g})$ as the $\C$-algebra with generators $E_i, F_i, K_i^{\pm 1}, 1\le i\le r$ and relations
    \begin{equation}
        K_iK_i^{\pm 1} = K_iK_i^{\pm 1} = 1, \quad K_iK_j = K_jK_i
    \end{equation}
    \begin{equation}
        K_iE_jK_i^{-1} = q_i^{a_{ij}}E_j, \quad K_iF_jK_i^{-1} = q_i^{-a_{ij}}F_j
    \end{equation}
    \begin{equation}
        [E_i,F_j] = \delta_{ij}\frac{1}{q_i - q_i^{-1}}(K_i-K_i^{-1})
    \end{equation}
    \begin{equation}
        \sum_{\ell = 0}^{1-a_{ij}}(-1)^{\ell}
        \left[\begin{matrix}
        1-a_{ij}\\\ell
        \end{matrix}\right]_{q_i}E_i^{1-a_{ij}-\ell}E_jE_i^{\ell}
    \end{equation}
    \begin{equation}
        \sum_{\ell = 0}^{1-a_{ij}}(-1)^{\ell}
        \left[\begin{matrix}
        1-a_{ij}\\\ell
        \end{matrix}\right]_{q_i}F_i^{1-a_{ij}-\ell}F_jF_i^{\ell}
    \end{equation}

    Here we made use of $q$-integers. For an arbitrary complex number $q$ such that $q\notin\{0,1,-1\}$ and $m,n\in \Z_{\ge 0}$, one  defines
    \[
        [n]_q := \frac{q^{n}-q^{-n}}{q-q^{-1}}
    \]
    \[
        [n]_q! := \prod_{i=1}^n [i]_q
    \]
    and
    \[
        \left[\begin{matrix}
            m\\n
        \end{matrix}\right]_q := \frac{[m]_q!}{[n]_q![m-n]_q!}.
    \]

    The quantized universal enveloping algebra $U_q(\mathfrak{g})$ is endowed with the structure of Hopf algebra \cite[Proposition 6.5]{KS97} by letting
    \begin{equation}
        \Delta(E_i) = E_i\otimes K_i + 1\otimes E_i, \quad \Delta(F_i) = F_i\otimes 1 + K_i^{-1}\otimes F_i, \quad \Delta(K_i^{\pm 1}) = K_i^{\pm 1}\otimes K_i^{\pm 1}
    \end{equation}
    \begin{equation}
        \varepsilon(E_i) = \varepsilon(F_i) = 0, \quad \varepsilon(K_i^{\pm}) = 1
    \end{equation}
    \begin{equation}
        S(E_i) = -E_iK_i^{-1}, \quad S(F_i) = -K_iF_i, \quad S(K_i^{\pm 1}) = K_i^{\mp 1}.
    \end{equation}
    We remark that by \cite[Proposition 6.6]{KS97}, the antipode of $U_q(\mathfrak{g})$ is invertible.

    For a fixed subset of simple roots $S\subseteq \Pi$, we consider the Levi subalgebra $U_q(\mathfrak{l}_S)$, which is the subalgebra of $U_q(\mathfrak{g})$ generated by the set $\{E_i, F_i, K_j^{\pm 1}\mid \alpha_i\in S, 1\le j \le r\}$. The following enables us to use the tools developed in section \ref{SecEquivariantCocalculi}.
    
    \begin{proposition}
        The quantized universal enveloping algebra $U_q(\mathfrak{g})$ is faithfully flat as a left- and right $U_q(\mathfrak{l}_S)$-module.
    \end{proposition}

    \begin{proof}
        From the definition of the coproduct and antipode on $U_q(\mathfrak{g})$ it follows that $U_q(\mathfrak{l}_S)$ is a Hopf subalgebra. It also follows that $U_q(\mathfrak{g})$ is generated by group-like and skew primitive elements, thus by \cite[Corollary 5.1.14]{Rad12} it follows that $U_q(\mathfrak{g})$ is pointed. In particular, the coradical $U_q(\mathfrak{g})_0$ is cocommutative and the assertion follows from the main result in \cite{Mas91}.
    \end{proof}

    For simplicity, we will from now on assume that $q$ is not a root of unity. Let $\mathcal{C}$ denote the full subcategory of the category of finite dimensional $U_q(\mathfrak{g})$-modules $M$ that admit a decomposition
    \[
        M = \bigoplus_{\lambda \in P} M[\lambda]
    \]
    where $M[\lambda] := \{m\in M \mid \forall 1\le i\le r \colon K_im = q^{(\alpha_i,\lambda)}m\}$ denotes the weight space of the weight $\lambda$. Note that $\mathcal{C}$ is a tensor subcategory of ${}_{U_q(\mathfrak{g})}\mathcal{M}$. Indeed, it is immediate that $\mathcal{C}$ is closed under direct sums, and for tensor products, this follows directly from the way how the coproducts of the $K_i$'s are defined. The category $\mathcal{C}$ is also closed under duals with respect to $S$ by verifying that
    \[
        M^\ast = \bigoplus_{\lambda \in P} M^\ast[-\lambda].
    \]
    Note that one gets the same decomposition if one considers the dual of $M$ with respect to $S^{-1}$. In other words, $\mathcal{C}$ is also closed under taking dual with respect to $S^{-1}$.

    \begin{definition}
        \label{DefQuantizedFlagManifold}
        We define the $q$-deformed coordinate ring $\mathcal{O}_q(G)$ as
        \[
            \mathcal{O}_q(G) := U_q(\mathfrak{g})_\mathcal{C}^\circ.
        \]
        In addition, for a subset $S\subseteq \Pi$ we define the \textit{quantized flag manifold}
        \[
            \mathcal{O}_q(G/L_S) := {}^{U_q(\mathfrak{l}_S)}\mathcal{O}_q(G) = \{b\in \mathcal{O}_q(G)\mid \forall X\in U_q(\mathfrak{l}_S)\mid b_{(1)}\langle b_{(2)}, X\rangle = \varepsilon(X)b\}.
        \]
    \end{definition}

    Since the antipode of $U_q(\mathfrak{g})$ is invertible and $\mathcal{C}$ is closed under taking duals with respect to $S$ and $S^{-1}$, the antipode of $\mathcal{O}_q(G)$ is bijective as well (see the discussion in Appendix \ref{SecFinDuals}). Moreover, $\mathcal{O}_q(G/L_S)$ is a left coideal subalgebra which $O_q(G)$ is faithfully flat over. The last assertion follows from Theorem \ref{ThmMuellerSchneider} using that every $U_q(\mathfrak{l}_S)$-module that arises as the restriction of a module in $\mathcal{C}$ completely reducible. In this situation, we consider the quotient coalgebra
    \[
        U_q^c(\mathfrak{g}/\mathfrak{l}_S) := U_q(\mathfrak{g})\otimes_{U_q(\mathfrak{l}_S)} \C_\mathrm{triv}.
    \]
    By Lemma \ref{LemFinDualofC} we have $\mathcal{O}_q(G/L_S)\cong U_q^c(\mathfrak{g}/\mathfrak{l}_S)_\mathcal{C}^\circ$. Here, the latter even agrees with the full finitary dual $U_q^c(\mathfrak{g}/\mathfrak{l}_S)^\circ$. Indeed, the quotient $U_q^c(\mathfrak{g}/\mathfrak{l}_S)$ is spanned by the images under the projection $U_q(\mathfrak{g})\twoheadrightarrow U_q^c(\mathfrak{g}/\mathfrak{l}_S)$ of monomials in $E_i$'s and $F_i$'s for $1\le i\le r$. If $v$ is the image of any such monomial in $U_q^c(\mathfrak{g}/\mathfrak{l}_S)$, then each $K_i$ acts on $v$ via some scalar of the form $q^{(\alpha_i,\lambda)}$, where $\lambda \in P$. Thus, every finite dimensional $U_q(\mathfrak{g})$-quotient of $U_q^c(\mathfrak{g}/\mathfrak{l}_S)$ is already in $\mathcal{C}$.

    \begin{remark}
        Here we use $U_q^c(-)$ in analogy with $\mathcal{O}_q(-)$ on the quantized coordinate ring-side. The superscript $c$ is used to avoid confusions with the quantized universal enveloping algebra, and to emphasize the coalgebra structure.
    \end{remark}

    \subsection{Poincar\'e--Birkhoff--Witt basis for quotient coalgebras}

    We want to give an explicit basis for the quotient coalgebra $U_q^c(\mathfrak{g}/\mathfrak{l}_S)$ of a quantized flag manifold. First let us recall the PBW theorem for $U_q(\mathfrak{g})$. For details, we refer to \cite[Sec. 6.2]{KS97}. Let $W$ be the Weyl group of $\mathfrak{g}$ and denote by $w_0$ its longest element. Recall that $W$ is generated by involutions $s_1, \dots, s_r$ corresponding to the choice of simple roots $\alpha_1, \dots, \alpha_r$. Any choice of a reduced decomposition (this is, the number $n$ is the smallest non-negative integer such that $w_0$ admits a decomposition as below).
    \[
        w_0 = s_{i_n}\dots s_{i_1}
    \]
    gives rise to elements $E_{\beta_1},\dots, E_{\beta_n}, F_{\beta_1},\dots,F_{\beta_n} \in U_q(\mathfrak{g})$ called \textit{root vectors}. In fact, the assignment $i\mapsto \beta_i$ is a bijection between the set $\{1,\dots, n\}$ and the positive roots of $\mathfrak{g}$. By taking certain monomials in the root vectors, we obtain a basis of $U_q(\mathfrak{g})$. We adapt the formulation of \cite[Theorem 24, Chap. 6]{KS97}.

    \begin{theorem}
        \label{ThmPBW}
        The following elements form a vector space basis of $U_q(\mathfrak{g})$:
        \[
            F_{\beta_1}^{l_1}\dots F_{\beta_n}^{l_n}K_1^{m_1}\dots K_r^{m_r}E_{\beta_1}^{t_1}\dots E_{\beta_n}^{t_n}
        \]
        where $l_j, t_j \in \Z_{\ge 0}$ for $1\le j\le n$ and $m_j \in \Z$ for $1\le j \le r$.
    \end{theorem}

    In the sequel, we assume for the sake of simplicity that $S$ is of the form $S\setminus\{\alpha_i\}$, where $\alpha_i$ is a simple root which appears in any positive root with coefficient at most one in the corresponding decomposition into simple roots. In other words, we only consider irreducible (quantized) flag manifolds from now on. Following \cite{HK04}, we write
    \[
        \Delta_S^+ := \mathrm{span}_\Z S\cap \Delta^+, \quad \overline{\Delta_S^+} := \Delta^+\setminus \Delta_S^+.
    \]
    In addition, we write $\overline{\Delta_S^+} = \{\beta_1', \dots, \beta_M'\}$, where $M = |\overline{\Delta_S^+}|$.

    \begin{proposition}
        \label{PropPBWQuotient}
        The images of the elements
        \[
            E_{\beta_1'}^{l_1}\dots E_{\beta_M'}^{l_M}F_{\beta_1'}^{m_1}\dots F_{\beta_M'}^{m_M}, \quad l_j, m_j \in\Z_{\ge 0}
        \]
        under the projection $U_q(\mathfrak{g})\twoheadrightarrow U_q^c(\mathfrak{g}/\mathfrak{l}_S)$ form a basis of $U_q^c(\mathfrak{g}/\mathfrak{l}_S)$.
    \end{proposition}

    \begin{proof}
        See \cite[Proposition 4]{HK04}.
    \end{proof}

    In the case of $\mathfrak{g} = \mathfrak{sl}_{r+1}$, there is a more explicit description for the root vectors. Recall that the standard $A_r$-root system \cite[p. 64]{Hum72} is given by
    \[
        \Delta := \{\alpha_{ij} := \varepsilon_i - \varepsilon_j \in \R^{r+1} \mid 1\le i\neq j \le r+1\}.
    \]
    Here, the $\varepsilon_i$'s denote the standard basis vectors of $\R^{r+1}$. As a subset of simple roots, we choose $\Pi = \{\alpha_i := \alpha_{i,i+1}\mid 1\le i\le r\}$. Corresponding to this choice of simple roots, the positive roots are given by the set $\Delta^+ = \{\alpha_{ij}\mid 1\le i<j\le r+1\}$. Recall also, that the Weyl group of $\mathfrak{sl}_{r+1}$ is given by the symmetric group on $\{1, \dots, r+1\}$, with its standard generators given by the transpositions 
    \[
        s_i := (i,i+1), \quad 1\le i \le r
    \]
    (written in cycle notation).
        
    In the following, we write $[X,Y]_c := XY - cYX$, where $c\in \C$ for the twisted commutator of two elements $X,Y$ of an arbitrary $\C$-algebra $A$.

    \begin{proposition}
        \label{PropAnRootVectors}
        For $1\le i<j \le r+1$ write
        \begin{equation}\label{eq:PropAnRootVectors}
            E_{ji} := [\dots [[E_{j-1},E_{j-2}]_{q^{-1}},E_{j-3}]_{q^{-1}},  \dots, E_i]_{q^{-1}}.
        \end{equation}
        Then the $E_{ji}$'s are precisely the positive root vectors of $U_q(\mathfrak{sl}_{r+1})$ corresponding to the reduced decomposition of the longest element of the Weyl group
        \[
            w_0 = (s_r\dots s_1)(s_r\dots s_2) \dots (s_rs_{r-1})s_r.
        \]
    \end{proposition}

    \begin{proof}
        See \cite[Proposition A.3]{BS25}.
    \end{proof}

    \subsection{The Podle\'s cocalculus}\label{SecPodlesCocalculus}

    For our first example of a codifferential calculus on a quantized flag manifold, we take $\mathfrak{g} := \mathfrak{sl}_2$, and $S = \emptyset$. Thus, the Levi subalgebra is $U_q(\mathfrak{h}) := U_q(\mathfrak{l}_S) = \langle K^{\pm 1}\rangle$. We write $U_q^c(S^2) := U_q(\mathfrak{sl}_2)\otimes_{U_q(\mathfrak{h})}\C_\mathrm{triv}$. Note that the root vectors for $U_q(\mathfrak{sl}_2)$ are simply the generators $E$ and $F$. Thus, we can give an explicit basis for $U_q^c(S^2)$.

    \begin{corollary}
        \label{CorBasisForPodles}
        The elements
        \[
            [E^mF^n], \quad m,n\in \Z_{\ge 0}
        \]
        form a basis of $U_q^c(S^2)$.
    \end{corollary}

    \begin{proof}
        This is a special case of Proposition \ref{PropPBWQuotient}.
    \end{proof}
    
    Consider the following elements $e := [E]$ and $f := [F]$, and let $T := \mathrm{span}_\C\{e,f\}$.

    \begin{lemma}
        \label{LemPodlesTangentSpace}
        The subspace $T$ is a quantum tangent space, and moreover $T \in {}_{U_q(\mathfrak{h})}\mathcal{M}^0$.
    \end{lemma}

    \begin{proof}
        Since $\varepsilon(e) = \varepsilon(f) = 0$, we have $T\subseteq U_q^c(\mathfrak{g}/\mathfrak{l}_S)^+$. In addition, $U_q(\mathfrak{h})T\subseteq T$, since
        \[
            Ke = [KE] = q^2[EK] = q^2e
        \]
        and similarly $Kf = q^{-2}f$. Lastly, it follows at once that $T \in {}_{U_q(\mathfrak{h})}\mathcal{M}^0$ and that $\Delta(T) \subseteq (T\oplus k[1])\otimes C$, since
        \begin{align*}
            \Delta(e) = [E]\otimes [K] + [1]\otimes [E] = e\otimes [1] + [1]\otimes e
        \end{align*}
        and similarly $\Delta(f) = f\otimes [1] + [1]\otimes f$.
    \end{proof}

    From the above, and Proposition \ref{PropTangentspaceFOCC} it follows that $\mathscr{W}_1^q(S^2) := U_q(\mathfrak{sl}_2)\otimes_{U_q(\mathfrak{l}_S)} T$ together with the codifferential $X\otimes_{U_q(\mathfrak{l}_S)}t\mapsto Xt$ is an equivariant first-order codifferential calculus. Our next goal is to compute its maximal prolongation.

    \begin{proposition}
        \label{PropPodlesCorelations}
        Recall the map
        \[
            \check{\delta}\colon T\otimes T \rightarrow U_q(\mathfrak{sl}_2)\otimes_{U_q(\mathfrak{h})} U_q^c(S^2)^+, \quad [x]\otimes t\mapsto x_{(1)}\otimes_{U_q(\mathfrak{h})}S(x_{(2)})t.
        \]
        Then, in the notation of \ref{PropDegTwoRelationsSupZero}, we have $\check{R} = \mathrm{span}_\C\{e\otimes f - q^2f\otimes e\}$.
    \end{proposition}

    \begin{proof}
        Computing the map $\check{\delta}$ on the basis elements of $T\otimes T$ gives
        \begin{align*}
            \check{\delta}(e\otimes e) &= q^{-2}E\otimes_{U_q(\mathfrak{h})}e - q^{-2}\otimes_{U_q(\mathfrak{h})}[E^2]\\
            \check{\delta}(e\otimes f) &= q^2E\otimes_{U_q(\mathfrak{h})} f - q^2\otimes_{U_q(\mathfrak{h})}[EF]\\
            \check{\delta}(f\otimes e) &= F\otimes_{U_q(\mathfrak{h})}e - 1\otimes_{U_q(\mathfrak{h})}[EF]\\
            \check{\delta}(f\otimes f) &= F\otimes_{U_q(\mathfrak{h})} f - 1\otimes_{U_q(\mathfrak{h})} [F^2].
        \end{align*}
        A sufficient condition for an element $w \in U_q(\mathfrak{sl}_2)\otimes_{U_q(\mathfrak{h})}U_q^c(S^2)^+$ to lie in $\mathscr{W}_1^q(S^2)$ is that $(\varepsilon\otimes_{U_q(\mathfrak{h})} \id)(w) \in T$. The values of the basis elements of $T\otimes T$ under the composition $(\varepsilon\otimes_{U_q(\mathfrak{h})}\id)\circ \check{\delta}$ are
        \[
            -q^{-2}[E^2], \quad -q^2[EF], \quad -[EF], \quad -[F^2].
        \]
        Using Corollary \ref{CorBasisForPodles}, we see that the only elements of $T\otimes T$, that can get mapped to $\mathscr{W}_1^q(S^2)$ are scalar multiples of $e\otimes f - q^2f\otimes e$, and indeed
        \[
            \check{\delta}(e\otimes f-q^2f\otimes e) = q^2E\otimes_{U_q(\mathfrak{h})}f - q^2F\otimes_{U_q(\mathfrak{h})}e \in \mathscr{W}_1^q(S^2).
        \]
    \end{proof}

    The last step in computing the maximal prolongation of $\mathscr{W}_1^q(S^2)$ is to determine the quadratic coalgebra $\mathfrak{C}(T,\check{R})$. If we consider the dual vector space $V$ of $T$ equipped with dual basis $\{e^\ast,f^\ast\}$, the orthogonal complement of $R$ in $V\otimes V$ with respect to the pairing
    \[
        V\otimes V\times T\otimes T\rightarrow \C, \quad (v\otimes w, s\otimes t) \mapsto v(s)w(t)
    \]
    is $\mathrm{span}_\C\{e^\ast\otimes e^\ast, f^\ast\otimes f^\ast, e^\ast\otimes f^\ast + q^{-2}f^\ast\otimes e^\ast\}$. The dual quadratic algebra of $\mathfrak{C}(T,\check{R})$ is therefore
    \[
        \Lambda_{q^{-2}}^\bullet(V) := T(V)/\langle e^\ast\otimes e^\ast, f^\ast\otimes f^\ast, e^\ast\otimes f^\ast + q^{-2}f^\ast\otimes e^\ast\rangle.
    \]
    Which has basis $1, e^\ast, f^\ast, e^\ast\wedge f^\ast$, which can be seen by applying Bergman's diamond lemma \cite{Ber77}. With the help of Lemma \ref{LemDualQuadraticAlgebra}, we conclude that $[1], e, f, e\otimes f - q^2 f\otimes e$ forms a basis of $\mathfrak{C}(T,\check{R})$.
    If we denote by $\C_\lambda$ the one-dimensional $U_q(\mathfrak{h})$-module, where $K$ acts via the scalar $q^\lambda$, and by $W(\lambda)$ the induced module $U_q(\mathfrak{sl}_2)\otimes_{U_q(\mathfrak{h})} \C_\lambda$, where we denote by $w_\lambda$, the vector $1\otimes_{U_q(\mathfrak{h})} 1$, and use the obvious identifications
    \[
        U_q(\mathfrak{g})\otimes_{U_q(\mathfrak{h})}\C e \cong W(2), \quad U_q(\mathfrak{g})\otimes_{U_q(\mathfrak{h})} \C f \cong W(-2), \quad U_q(\mathfrak{g})\otimes_{U_q(\mathfrak{h})} \C 1\cong W(0)
    \]
    as well as the identification
    \[
        W(0) \overset{\sim}{\rightarrow} U_q(\mathfrak{g}) \otimes_{U_q(\mathfrak{h})} \C (e\otimes f - q^2f\otimes e), \quad w_0 \mapsto q^{-2}e\otimes f - f\otimes e
    \]
    then the codifferential calculus explicitly looks as follows:
    \begin{equation}\label{eq:BGGresolutionSL2}\begin{tikzcd}
	&& {W(2)} \\
	0 & {W(0)} & \bigoplus & {W(0)} & 0 \\
	&& {W(-2)}
	\arrow["{\delta_E}", from=1-3, to=2-4]
	\arrow[from=2-1, to=2-2]
	\arrow["{-\delta_F}", from=2-2, to=1-3]
	\arrow["{\delta_E}"', from=2-2, to=3-3]
	\arrow[from=2-4, to=2-5]
	\arrow["{\delta_F}"', from=3-3, to=2-4]
    \end{tikzcd}\end{equation}
    where $\delta_F$ is the map $\delta_F\colon W(\lambda) \rightarrow W(\lambda + 2), w_\lambda \mapsto Fw_{\lambda + 2}$, and $\delta_E\colon W(\lambda)\rightarrow W(\lambda-2)$ is defined analogously. Note that \eqref{eq:BGGresolutionSL2} is exactly the complex given in the introduction of \cite{HK072}.

    \subsection{Antiholomorphic cocalculi on quantized projective spaces}\label{SecAntiholomorphicProjectiveCocalculus}

    We proceed by giving family of higher-rank examples. Let $r \ge 1$ and $\mathfrak{g} = \mathfrak{sl}_{r+1}$. We choose the subset of simple roots $S = \{\alpha_2, \dots, \alpha_r\}$, such that the corresponding Levi subalgebra is $U_q(\mathfrak{l}_S) := \langle K_i^{\pm 1},E_j, F_j\colon 1\le i\le r, 2\le j\le r\rangle$. Let us denote $U_q^c(\C P^r) := U_q^c(\mathfrak{g}/\mathfrak{l}_S)$ and $e_i := [E_iE_{i-1}\dots E_1]\in U_q^c(\C P^r)$ for $1 \le i \le r$. We consider the subspace of $U_q^c(\C P^r)$
    \[
            T^{(0,1)} := \mathrm{span}_\C\{e_1, \dots, e_r\}.
    \]
    In what follows, we will view modules over $U_q(\mathfrak{l}_S)$ as modules over $U_q(\mathfrak{sl}_r)$, using the embedding $U_q(\mathfrak{sl}_r)\rightarrow U_q(\mathfrak{l}_S)$ determined by
    \begin{equation}\label{eq:EmbeddingUqslr}
        E_i\mapsto E_{i+1},\quad F_i\mapsto F_{i+1},\quad K_i \mapsto K_{i+1},\quad 1\le i \le r-1.
    \end{equation}
    When dealing with weights, we will refer to weights of the corresponding representation of $U_q(\mathfrak{sl}_r)$. On that note, recall that the fundamental weights $\varpi_1, \dots,\varpi_{r-1}$ of $\mathfrak{sl}_r$ are determined by the relations
    \[
        (\alpha_i^\vee,\varpi_j) = \delta_{ij}.
    \]

    \begin{lemma}
        \label{LemActiononT01}
        For $1\le i\le r$ and $2\le j \le r$, we have
        \begin{enumerate}[label = (\arabic*)]
            \item
            \[
                K_je_i = \begin{cases}
                    qe_i, \quad j=i\\
                    q^{-1}e_i, \quad j = i+1\\
                    e_i, \quad \text{else}.
                \end{cases}
            \]
            \item
            \[
                E_je_i = \begin{cases}
                    e_{i+1}, \quad j = i+1\\
                    0,\quad \text{else}
                \end{cases}
            \]
            \item
            \[
                F_je_i = \begin{cases}
                    e_{i-1},\quad j = i\\
                    0, \quad \text{else}
                \end{cases}
            \]
            where $e_0 := e_{r+1} := 0$.
        \end{enumerate}
    \end{lemma}

    \begin{proof}
        \textit{(1)}: This follows immediately from the commutation relations 
        \[
            K_iE_i = q^2E_iK_i, K_{i\pm 1}E_i = q^{-1}E_iK_{i\pm 1}
        \]
        and $K_jE_i = E_iK_j$ if $|i-j|\ge 2$. The latter follow since the Cartan matrix for $\mathfrak{sl}_{r+1}$ is given by $a_ii = 2, a_{i,i\pm 1} = -1$ and $a_{ij} = 0$ for $|i-j| \ge 2$, and the fact that $(\alpha,\alpha) = 2$ for all roots of $\mathfrak{sl}_{r+1}$.

        \textit{(2)}: Recall that for $U_q(\mathfrak{sl}_{r+1})$, the Serre relations among the $E_i's$ are explicitly given by
        \[
            E_{i\pm1}^2E_i -[2]_qE_{i\pm 1}E_iE_{i\pm 1} + E_iE_{i\pm 1} = 0
        \]
        and
        \[
            E_iE_j = E_jE_i
        \]
        if $|i-j| \ge 2$. Thus, it is clear that $E_{i+1}e_i = e_{i+1}$ for $1\le i \le r-1$ and $E_je_i = 0$ for $j > i+1$. For $j = i$, we get
        \begin{align*}
            E_ie_i &= [E_i^2E_{i-1}\dots E_1] = [[2]_qE_iE_{i-1}E_i \dots E_1] - [E_{i-1}E_i^2E_{i-2}\dots E_1] = 0.
        \end{align*}
        If $j = i-1$, then
        \begin{align*}
            E_{i-1}e_i &= [E_{i-1}E_i\dots E_1] = \frac{1}{[2]_q}\left([E_{i-1}^2E_iE_{i-2}\dots E_1] + [E_iE_{i-1}E_{i-2}\dots E_1]\right)\\
            &= \frac{1}{[2]_q}E_iE_{i-1}e_{i-1}\\
            &= 0
        \end{align*}
        and lastly, if $j < i-1$, then
        \[
            E_je_i = E_i \dots E_{j+2}E_je_{j+1} = 0
        \]
        by the case $j = i-1$.

        \textit{(3)}: Since $E_i$ and $F_j$ commute if $i\neq j$, it is clear that $F_je_i = 0$ if $j>i$. For $j=i$, we compute using \textit{(1)}
        \begin{align*}
            F_ie_i &= [F_iE_i\dots E_1] = [E_iF_iE_{i-1}\dots E_1] - [[E_i,F_i]E_{i-1}\dots E_i]\\
            &= -\frac{1}{q-q^{-1}}(K_i-K_i^{-1})[E_{i-1}\dots E_1]\\
            &= -\frac{1}{q-q^{-1}}\left(q^{-1}e_{i-1} + qe_{i-1}\right) = e_{i-1}.
        \end{align*}
        Lastly, for $j < i$, we get
        \[
            F_je_i = E_i\dots E_{j+1}F_je_j = E_i\dots E_{j+1}e_{j-1} = 0
        \]
        by \textit{(2)}.
    \end{proof}

    In fact, the elements $e_1, \dots, e_r \in U_q^c(\C P^r)^+$ form a linear independent set.

    \begin{lemma}
        \label{LemRootVectorsInCPr}
        For every $2\le j \le r+1$, the image of $E_{ji}$ (see \ref{PropAnRootVectors}) in $U_q^c(\C P^r)$ under the canonical projection is $e_{j-1}$. In particular the set $\{e_1,\dots, e_r\}$ is linearly independent.
    \end{lemma}

    \begin{proof}
        Solely for the sake of this proof, we let $e_0 := [1]$. Fix $2\le j \le r+1$. We claim that
        \begin{equation}\label{eq:ActionOfRootVectors}
            E_{j,j-i}e_{j-i-1} = e_{j-1}
        \end{equation}
        for all $1\le i < j$. Indeed, if $i = 1$ and $j \ge 3$, then
        \[
            E_{j,j-1}e_{j-2} = E_{j-1}e_{j-2} = e_{j-1}
        \]
        by Lemma \ref{LemActiononT01}, and if $i = 1$ and $j = 2$, then we see directly
        \[
            E_{21}e_0 = E_1[1] = e_1.
        \]
        If $i\ge 1$ and $i+1 < j$, then by induction on $i$
        \begin{align*}
            E_{j,j-i-1}e_{j-i-2} &= [E_{j,j-i},E_{j-i-1}]_{q^{-1}}e_{j-i-2}\\
            &= E_{j,j-1}E_{j-i-1}e_{j-i-2} -q^{-1}E_{j-i-1}E_{j,j-i}e_{j-i-2}\\
            &= E_{j,j-i}e_{j-i-1} = e_{j-1}.
        \end{align*}
        Here, the last step follows by induction, the equality $E_{j-i-1}e_{j-i-2} = e_{j-i-1}$ follows from the case $i = 1$, since $E_{j-i-1} = E_{j-i,j-i-1}$, and $E_{j,j-i}e_{j-i-2} = 0$, since $E_{j,j-i}$ is a sum of monomials in $E_{j-i}, \dots, E_{j-1}$ by \eqref{eq:PropAnRootVectors}, which annihilate $e_{{j-i-2}}$. In the case $j-i-2 = 0$, the latter assertion follows directly from the relations of $U_q^c(\C P^r)$ and for $j-i-2 > 0$ this is a consequence of Lemma \ref{LemActiononT01}.

        We conclude $[E_{j,1}] = E_{j,1}[1] = e_{j-1}$ for all $2\le j \le r+1$. The last assertion follows from Proposition \ref{PropPBWQuotient}.
    \end{proof}

    From Lemma \ref{LemActiononT01} and \ref{LemRootVectorsInCPr} it follows that $T^{(0,1)}$ is an irreducible $U_q(\mathfrak{l}_S)$-submodule of $U_q^c(\C P^r)$. If we restrict the action to $U_q(\mathfrak{sl}_r)$ along the map \eqref{eq:EmbeddingUqslr}, then $T^{(0,1)}\cong V_{\varpi_{r-1}}$, the irreducible $U_q(\mathfrak{sl}_r)$-module of highest weight $\varpi_{r-1}$.

    \begin{proposition}
        \label{PropTangentspaceCPr}
        The subspace $T^{(0,1)}$ is a quantum tangent space such that $T^{(0,1)} \in {}_{U_q(\mathfrak{l}_S)}\mathcal{M}^0$.
    \end{proposition}

    \begin{proof}
        We have already seen that $U_q(\mathfrak{l}_S)T^{(0,1)}\subseteq T^{(0,1)}$ (Lemma \ref{LemActiononT01}). To see the coproduct property, we compute the coproducts
        \[
            \Delta(E_i\dots E_1)
        \]
        for $1\le i \le r$. Denote by $U_q(\mathfrak{h})$ the subalgebra generated by $K_j^{\pm 1}$ for $1\le j \le r$. Then in the quotient $U_q(\mathfrak{sl}_{r+1})\otimes_{U_q(\mathfrak{h})}\C_\mathrm{triv}$, the coproducts evaluate to
        \[
            \sum_{j = 0}^i\sum_{\sigma \in \mathrm{Sh}(j,i-j)}\lambda(j,\sigma)[E_{\sigma(1)}\dots E_{\sigma(j)}]\otimes [E_{\sigma(j+1)}\dots E_{\sigma(i)}]
        \]
        here $\mathrm{Sh}(j,i-j)$ denotes the set of $(j,i-j)$-shuffels, that is, permutations $\sigma \in S_{i}$ such that $\sigma(1)< \dots < \sigma(j)$ and $\sigma(j+1)< \dots < \sigma(i)$. If we pass further to the quotient $U_q^c(\C P^r)$ then the only summands that survive are $[E_i\dots E_1]\otimes [1]$ and $[1]\otimes [E_i\dots E_1]$, and one can check that the corresponding coefficients $\lambda(j,\sigma)$ are both equal to one. Thus,
        \[
            \Delta(e_i) = e_i\otimes [1] + [1]\otimes e_i
        \]
        for all $1\le i \le r$ and the assertion follows immediately. 
    \end{proof}

    \begin{proposition}
        \label{PropCPrCorelation}
        In the notation of Proposition \ref{PropDegTwoRelationsSupZero} the subset $\check{R}\subseteq T^{(0,1)}\otimes T^{(0,1)}$ is given by
        \[
            \mathrm{span}_\C\{e_i\otimes e_j - qe_j\otimes e_i \mid 1\le i < j\le r\}.
        \]
    \end{proposition}

    \begin{proof}
        The crucial observation is that $\check{\delta}$ is $U_q(\mathfrak{l}_S)$-linear, and therefore in particular $U_q(\mathfrak{sl}_r)$-linear, if we restrict the action along the embedding $U_q(\mathfrak{sl}_r)\rightarrow U_q(\mathfrak{l}_S)$ described above. Note that $T^{(0,1)}$ is an irreducible $U_q(\mathfrak{sl}_r)$-module of weight $\varpi_{r-1}$. Thus, the tensor product with itself decomposes as in \cite[p. 300]{OnVin90} (note that the decomposition of tensor products of finite dimensional type one representation of $U_q(g)$ into irreducible components is the same as for the Lie algebra $\mathfrak{g}$, as remarked in \cite[Sec. 7.2]{KS97}):
        \[
            T^{(0,1)}\otimes T^{(0,1)} \cong V_{2\varpi_r}\oplus V_{\varpi_{r-2}}.
        \]
        Here $V_\lambda$ denotes the irreducible $U_q(\mathfrak{sl}_r)$-module of highest weight $\lambda$. The irreducible components are spanned by their (up to scalar multiples) unique lowest weight vectors. More precisely, we have
        \[
            V_{2\varpi_{r-1}} = U_q(\mathfrak{l}_S)(e_1\otimes e_1), \quad V_{\varpi_{r-2}} = U_q(\mathfrak{l}_S)(e_1\otimes e_2 - qe_2\otimes e_1).
        \]
        Indeed, it is clear that $e_1\otimes e_1$ is a lowest weight vector, and for $e_1\otimes e_2-qe_2\otimes e_1$ it is immediate that it is annihilated by $F_j$ for $j\ge 3$ and for $F_2$, we first compute
        \begin{align*}
            F_2e_2 &= [F_2E_2E_1] = \underbrace{[E_2F_2E_1]}_{= 0} - [[E_2,F_2]E_1] = \frac{1}{q-q^{-1}}\left([(K_2^{-1}-K_2)E_1]\right)\\
            &= \frac{1}{q-q^{-1}}\left(q[E_1]-q^{-1}[E_1]\right)\\
            &= e_1.
        \end{align*}
        Consequently,
        \begin{align*}
            F_2\cdot(e_1\otimes e_2 - qe_2\otimes e_1) &= K_2^{-1}e_1\otimes e_1 - qe_1\otimes e_1 = qe_1\otimes e_1 - qe_1\otimes e_1 = 0.
        \end{align*}
        Next, we compute $\check{\delta}(e_1\otimes e_2 - qe_2\otimes e_1)$. Using that
        \[
            (\id \otimes S)(\Delta(E_1)) = E_1\otimes K_1^{-1} - 1\otimes E_1K_1^{-1}
        \]
        we find that
        \begin{align*}
            \check{\delta}(e_1\otimes e_2) &= E_1\otimes_{U_q(\mathfrak{l}_S)}  K_1^{-1}e_2 - 1\otimes E_1K_1^{-1}e_2\\
            &= q^{-1}E_1\otimes_{U_q(\mathfrak{l}_S)}e_2 - q^{-1}\otimes_{U_q(\mathfrak{l}_S)}[E_1E_2E_1]
        \end{align*}
        and using that
        \begin{gather*}
            (\id\otimes S)(\Delta(E_2E_1)) = E_2E_1\otimes K_1^{-1}K_2^{-1} - E_2\otimes E_1K_1^{-1}K_2^{-1} + E_1\otimes K_1^{-1}E_2K_2^{-1}\\ + 1\otimes E_1K_1^{-1}E_2K_2^{-1}
        \end{gather*}
        as well as
        \[
            [E_2E_1^2] = [[2]_qE_1E_2E_1 - E_1^2E_2] = [2]_q[E_1E_2E_1]
        \]
        we see that
        \begin{align*}
            \check{\delta}(e_2\otimes e_1) &= \eta - E_2\otimes_{U_q(\mathfrak{l}_S)}[E_1K_1^{-1}K_2^{-1}E_1] + 1\otimes_{U_q(\mathfrak{l}_S)}[E_1K_1^{-1}E_2K_2^{-1}E_1]\\
            &= \eta - q^{-1}\otimes_{U_q(\mathfrak{l}_S)}[E_2E_1^2] + 1\otimes_{U_q(\mathfrak{l}_S)}[E_1E_2E_1]\\
            &= \eta + (1-q^{-1}[2]_q)\otimes_{U_q(\mathfrak{l}_S)}[E_1E_2E_1]\\
            &= \eta -q^{-2}\otimes_{U_q(\mathfrak{l}_S)}[E_1E_2E_1].
        \end{align*}
        Where $\eta \in U_q(\mathfrak{sl}_{r+1})\otimes_{U_q(\mathfrak{l}_S)} T^{(0,1)}$. Combining these two computations immediately gives 
        \[
            \check{\delta}(e_1\otimes e_2 - qe_1\otimes e_2) \in U_q(\mathfrak{sl}_{r+1})\otimes_{U_q(\mathfrak{l}_S)}T^{(0,1)}.
        \]
        We claim that on the other hand $\check{\delta}(e_1\otimes e_1) \notin U_q(\mathfrak{sl}_{r+1})\otimes_{U_q(\mathfrak{l}_S)} T^{(0,1)}$. Indeed, we have
        \begin{align*}
            (\varepsilon\otimes_{U_q(\mathfrak{l}_S)}\id)(\check{\delta}(e_1\otimes e_1)) &= S(E_1)e_1 = -[E_1K_1^{-1}E_1] = -q^{-2}[E_1^2] \notin T^{(0,1)}
        \end{align*}
        here the conclusion was made using the description of the basis of $U_q^c(\C P^r)$ from Proposition \ref{PropPBWQuotient} and Proposition \ref{PropAnRootVectors}.

        By $U_q(\mathfrak{l}_S)$-linearity of $\check{\delta}$, and since $T^{(0,1)}\otimes T^{(0,1)}$ only has two non-isomorphic irreducible components we must have $\check{R} = U_q(\mathfrak{l}_S)(e_1\otimes e_2-qe_2\otimes e_1)$. Because
        \[
            E_{i+1}\cdot(e_i\otimes e_j - qe_j\otimes e_i) = e_{i+1}\otimes e_j - qe_j\otimes e_{i+1}
        \]
        for $1\le i < j-1$ and
        \[
            E_{j+1}\cdot(e_i\otimes e_j - qe_j\otimes e_i) = e_i \otimes e_{j+1} - qe_{j+1}\otimes e_i 
        \]
        for $1\le i \le j\le r-1$, we find that
        \[
            \mathrm{span}_\C\{e_i\otimes e_j - qe_j\otimes e_i\mid 1\le i< j\le r\} \subseteq \check{R}.
        \]
        Finally, since $\dim_\C V_{\varpi_{r-2}} = \binom{r}{2}$, and the set on the left-hand side is linearly independent, the assertion follows.
    \end{proof}

    Let us consider the first-order covariant differential calculus on $\mathcal{O}_q(\C P^r)$ induced by
    \[
        \mathscr{W}_{(0,1)}^q(\C P^r) := U_q(\mathfrak{sl}_{r+1})\otimes_{U_q(\mathfrak{l}_S)} T^{(0,1)}.
    \]
    We may write this induced first-order differential calculus as
    \[
        \Omega_q^{(0,1)}(\C P^r) := \mathcal{O}_q(SL_{r+1})\cotimes{\pi_{O_q(\C P^r)}} V^{(0,1)}
    \]
    with $V^{(0,1)} = \mathcal{O}_q(\C P^r)^+/I^{(0,1)}$, and $I^{(0,1)}$ the orthogonal complement of $T^{(0,1)}$ with respect to the non-degenerate pairing
    \[
        \mathcal{O}_q(\C P^r)^+\times U_q^c(\C P^r)^+\rightarrow \C.
    \]
    By \cite[Section 7 Theorem 2]{HK04}, this first-order differential calculus is indeed isomorphic to the antiholomorphic part of the first-order Heckenberger--Kolb calculus on $\mathcal{O}_q(\C P^r)$. As a consequence of Proposition \ref{PropCPrCorelation} we obtain a new proof that the maximal prolongation of $\Omega_q^{(0,1)}(\C P^r)$ has classical dimension.
    
    \begin{corollary}
        \label{CorClassDim}
        The maximal prolongation of $\Omega_q^{(0,1)}(\C P^r)$ is isomorphic to
        \[
            \mathcal{O}_q(SL_{r+1})\cotimes{\pi_{\mathcal{O}_q(\C P^r)}} V^{(0,\bullet)}
        \]
        where $V^{(0,\bullet)} := T(V^{(0,1)})/\langle e^i\otimes e^j + q^{-1}e^j\otimes e^i\mid 1\le i\le j \le r\rangle$, and $e^1, \dots, e^r$ is the dual basis of $e_1, \dots, e_r$ with respect to the non-degenerate pairing $V^{(0,1)}\times T^{(0,1)}\rightarrow \C$. If we denote by $\wedge$ the product of $V^{(0,\bullet)}$ then the degree $n$ part for each $n \ge 0$ has basis
        \[
            \{e^{i_1}\wedge \dots \wedge e^{i_n}\mid 1\le i_1 < \dots < i_n \le r\}.
        \] 
        In particular,
        \[
            \dim_\C \Omega_q^{(0,n)}(\C P^r) := \dim_\C V^{(0,n)} = \binom{r}{n}.
        \]
        Thus, the dimension of $\Omega_q^{(0,\bullet)}(\C P^r)$ is classical.
    \end{corollary}

    \begin{proof}
        By Corollary \ref{PorMaxProlongationQuadraticAlgebras}, the maximal prolongation of $\Omega_q^{(0,1)}(\C P^r)$ can be written as
        \[
            \mathcal{O}_q(SL_{r+1})\cotimes{\pi_{\mathcal{O}_q(\C P^r)}} \mathfrak{A}(V^{(0,1)},\check{R}^\perp).
        \]
        The orthogonal complement of $\check{R}$ given in Proposition \ref{PropCPrCorelation} is computed to be
        \[
            \{e^i\otimes e^j + q^{-1}e^j\otimes e^i\mid 1\le i\le j \le r\}.
        \]
        Thus, $\mathfrak{A}(V^{(0,1)},\check{R}^\perp) = V^{(0,\bullet)}$. Since the relations of $V^{(0,\bullet)}$ are given by anti-commutation up to the scalar $q^{-1}$, the statement about the basis of $V^{(0,\bullet)}$ follows from a standard application of Bergman's diamond lemma \cite{Ber77}.
    \end{proof}

    \appendix

    \section{Quadratic algebras and coalgebras} \label{SecQuadraticCoalgebras}

    We start by defining quadratic algebras and coalgebras. On that note recall that the tensor algebra $T(V)$ of a vector space $V$ admits a coalgebra structure given by deconcatenation of tensors. Observe that this is a special case of the cotensor coalgebra \ref{DefRelTensorCoalg} where we take $C = k$. We denote this coalgebra by $T^c(V)$.

    \begin{definition}
        \label{DefQuadraticAlgebras}
        Let $V$ be a finite dimensional vector space, and $R\subseteq V\otimes V$ a subspace.
        \begin{enumerate}[label = (\arabic*)]
            \item The \textit{quadratic algebra} determined by $V$ and $R$ is
            \[
                \mathfrak{A}(V,R) := T(V)/\langle R\rangle.
            \]
            \item The \textit{quadratic coalgebra} determined by $V$ and $R$ is the graded subcoalgebra $\mathfrak{C}(V,R)$ of $T^c(V)$ with homogeneous components given by
            \[
                \mathfrak{C}^{(0)}(V,R) = k1, \quad \mathfrak{C}^{(1)}(V,R) = V, \quad \mathfrak{C}^{(n)}(V,R) = \bigcap_{i+j+2=n}V^{\otimes i}\otimes R\otimes V^{\otimes j}, \ (n\ge 2).
            \]
        \end{enumerate}
    \end{definition}

    Note that also quadratic algebras are graded. Explicitly, the homogeneous components of a quadratic algebra $\mathfrak{A}(C,R)$ are given by
    \[
        \mathfrak{A}^{(0)}(C,R) = k1, \quad \mathfrak{A}^{(1)}(C,R) = V, \quad \mathfrak{A}^{(n)} = V^{\otimes n}/\left(\sum_{i+j+2}V^{\otimes i}\otimes R\otimes V^{\otimes j}\right).
    \]

    Quadratic algebras and coalgebras are dual to each other, as made precise by the following.
        
    \begin{definition}
        \label{DefQuadraticDual}
        The quadratic dual algebra of a quadratic coalgebra $\mathfrak{C}(V,R)$ is the quadratic algebra $\mathfrak{A}(V^\ast,R^\perp)$, where $V^\ast$ is the dual space of $V$ and $R^\perp$ is the orthogonal complement of $R$ with respect to the pairing
        \[
            (V^\ast)^{\otimes 2}\times V^{\otimes 2} \rightarrow k, \quad (v^\ast\otimes w^\ast, v,w) = v^\ast(v)w^\ast(w).
        \]
    \end{definition}

    The dual quadratic algebra is chosen in such a way, that one obtains an induced non-degenerate algebra-coalgebra pairing
    \[
        \langle-,-\rangle\colon \mathfrak{A}(V^\ast,R^\perp)\times \mathfrak{C}(V,R) \rightarrow k.
    \]
    Here, algebra-coalgebra pairing means that
    \[
        \langle ab,c\rangle = \langle a,c_{(1)}\rangle\langle b,c_{(2)}\rangle, \quad \langle 1,c\rangle = \varepsilon(c)
    \]
    for all $a,b\in \mathfrak{A}(V^\ast,R^\perp)$ and $c\in \mathfrak{C}(V,R)$.

    \begin{lemma}
        \label{LemDualQuadraticAlgebra}
        Let $\mathfrak{C}(V,R)$ be a quadratic coalgebra and let $\mathfrak{A}(V^\ast,R^\perp)$ be its quadratic dual algebra. Then for each $n\ge 0$, the pairing
        \[
            \mathfrak{A}^{(n)}(V^\ast,R^\perp)\times \mathfrak{C}^{(n)}(V,R)\rightarrow k, \quad \langle [v_1^\ast\otimes \dots\otimes v_n^\ast],v_1\otimes \dots\otimes v_n\rangle = v_1^\ast(v_1)\dots v_n^\ast(v_n)
        \]
        is well-defined and non-degenerate. In particular, the extension of those pairings to a pairing between $\mathfrak{A}(V^\ast,R^\perp)$ and $\mathfrak{C}(V,R)$ is a non-degenerate algebra-coalgebra pairing, and moreover $\dim_k \mathfrak{A}^{(n)}(V^\ast, R^\perp) = \dim_k \mathfrak{C}^{(n)}(V,R)$ for all $n\ge 0$.
    \end{lemma}

    \begin{proof}
        Non-degeneracy is clear for $n = 0,1,2$. For $n\ge 3$, this directly follows by observing that the orthogonal complement of $\mathfrak{C}^{(n)}(V,R)$ with respect to the pairing $(V^\ast)^{\otimes n}\times V^{\otimes n}\rightarrow k$ is precisely 
        \[
            \tilde{R} := \sum_{i+j+2 =n} (V^\ast)^{\otimes i}\otimes R^\perp\otimes (V^\ast)^{\otimes j}.
        \]
        The inclusion $\tilde{R} \subseteq \mathfrak{C}^{(n)}(V,R)^\perp$ is immediate, and the reverse inclusion follows by showing that $\tilde{R}^\perp \subseteq \mathfrak{C}^{(n)}(V,R)$. Indeed, since $V$ is finite dimensional, and the pairing between $(V^\ast)^{\otimes n}$ and $V^{\otimes n}$ is non-degenerate, the last inclusion implies $\mathfrak{C}^{(n)}(V,R)^\perp \subseteq \tilde{R}^{\perp\perp} = \tilde{R}$.
    \end{proof}

    \section{Finitary duals}\label{SecFinDuals}

    For convenience, we collect some background material on finitary duals of modules over $k$-algebras. Fix an arbitrary $k$-algebra $U$. Recall that if $W$ is a (left) $U$-module, then its linear dual space $W^\ast$ is canonically a right $U$-module with action given by
    \[
        (f\triangleleft u)(w) := f(u\triangleright w).
    \]
    We are interested in a universal submodule of $W^\ast$, given by functionals with a certain finiteness condition.

    \begin{proposition}
        \label{PropFinitaryDual}
        Let $W$ be a $U$-module. For a linear map $f: W \rightarrow k$, the following are equivalent:
        \begin{enumerate}[label = (\roman*)]
            \item There exists a $U$-submodule $K$ of $W$ with $\dim_k(W/K) < \infty$ and $f\vert_K = 0$.
            \item The right $U$-submodule $fU$ of $W^\ast$ generated by $f$ is finite dimensional.
        \end{enumerate}
    \end{proposition}

    \begin{proof}
        $(i)\implies (ii)$: Let $K$ be a submodule of $W$ with $f\vert_K = 0$ such that $\dim_k(W/K) < \infty$. It follows immediately from the definition of the right $U$-module structure on $W^\ast$ that also every $g\in fU$ vanishes on $K$. Thus, the submodule $fU$ can be identified with a subspace of $(W/K)^\ast$, which is finite dimensional by assumption.

        $(ii) \implies (i)$: Suppose $fU$ is finite dimensional. Let
        \[
            K := \{w\in W\mid \forall g \in fU : g(w) = 0\}
        \]
        and note that this is a $U$ submodule of $W$. By choosing a basis $f_1, \dots, f_n$ of $fU$, it follows that
        \[
            K = \bigcap_{i=1}^n\ker(f_i).
        \]
        This description of $K$ gives us an injective map
        \[
            W/K = W/\bigcap_{i=1}^n\ker(f_i) \rightarrow \bigoplus_{i=1}^nW/\ker(f_i)
        \]
        showing that $W/K$ is finite dimensional.
    \end{proof}

    \begin{definition}
        \label{DefFinitaryDual}
        We call the set $W^\circ$ consisting of all $f\in W^\ast$ that satisfy either condition $(i)$ or condition $(ii)$ of \ref{PropFinitaryDual} the \textit{finitary dual} of $W$.
    \end{definition}

    From condition $(ii)$ of \ref{PropFinitaryDual} it follows that $W^\circ$ is a right $U$-submodule of $W^\ast$.

    If we consider $W = U$ with $U$-module structure given by left multiplication, then there is a third characterization of $U^\circ$.

    \begin{definition}
        \label{DefMatrixCoefficient}
        Let $M$ be a $U$-module, $m\in M$ and $\varphi \in M^\ast$. We define
        \[
            c_{\varphi,m}^M\colon U \rightarrow k, \quad u \mapsto \varphi(u\triangleright m).
        \]
        The functional $c_{f,m}^M$ is called a \textit{matrix coefficient} of $M$.
    \end{definition}

    \begin{proposition}
        For a functional $f \in U^\ast$ the following are equivalent:
        \begin{enumerate}[label = (\roman*)]
            \item $f \in U^\circ$.
            \item The exists a finite dimensional $U$-module $M$, an element $m \in M$ and $\varphi\in M^\ast$ with $f = c_{\varphi,m}^M$.
        \end{enumerate}
    \end{proposition}

    \begin{proof}
        $(i)\implies (ii)$: Let $f \in U^\circ$. Choose a $U$-submodule $K$ of $U$ with $\dim_k(U/K) < \infty$. Then $f$ is equal to the matrix coefficient $c_{\overline{f},1+K}^{U/K}$, where $\overline{f}(u + K) = f(u)$.

        $(ii)\implies (i)$: If $f = c_{\varphi, m}^M$ is a matrix coefficient of some finite dimensional $U$-module $M$, then for any $u \in U$, we have $fu = c_{\varphi u, m}^M$. Thus, $fU$ is finite dimensional.
    \end{proof}

    As it is well-known (see for instance \cite[Section 2.5]{Rad12}), the finitary dual of an algebra is a coalgebra.

    \begin{proposition}
        Let $c_{\varphi, m}^M$ be a matrix coefficient of a finite dimensional $U$-module $M$ and let $\{e_1, \dots, e_n\}$ be a basis of $M$ with dual basis $\{e^1, \dots, e^n\}$. Then for all $u, v \in U$
        \[
            c_{\varphi, m}^M(uv) = \sum_{i=1}^nc_{\varphi, e_i}^M(u)c_{e^i,m}^M(v).
        \]
        In particular, the dual map of the multiplication $\mu\colon U\otimes U \rightarrow U$, restricts to a map
        \[
            \Delta^\circ \colon U^\circ \rightarrow U^\circ\otimes U^\circ
        \]
        which endows $U^\circ$ with the structure of a coalgebra, where the counit is given by $f \mapsto f(1)$.
    \end{proposition}

    If $U$ is a Hopf algebra, then we can work out the addition and multiplication of two elements of $U^\ast$ in terms of matrix coefficients. It turns out that
    \[
        c_{\varphi,m}^M\cdot c_{\psi, n}^N = c_{\varphi\otimes \psi, m\otimes n}^{M\otimes N}, \quad c_{\varphi,m}^M + c_{\psi, n}^N = c_{\varphi + \psi, m + n}^{M\oplus N}.
    \]
    Consequently, the finitary dual of $U$ is a bialgebra. Moreover, the dual map of the antipode of $U$ restricts to $U^\circ$, and turns the latter into a Hopf algebra.
    
    It is occasionally useful to only consider matrix coefficients of modules that belong to a tensor subcategory $\mathcal{C}\subseteq {}_U\mathcal{M}$ (see Section \ref{SecInducedQHS}). 

    \begin{definition}
        \label{DefFinitaryDualSubcat}
        Let $\mathcal{C}$ be a tensor subcategory of ${}_U\mathcal{M}$. We let $U_\mathcal{C}^\circ$ be the subset of $U^\circ$ consisting of all matrix coefficients $c_{\varphi, m}^M$, with $M \in \mathcal{C}$.
    \end{definition}

    The subspace $U_\mathcal{C}^\circ$ is in fact a Hopf subalgebra of $U^\circ$. If the antipode of $U$ has an inverse $S^{-1}$, then the dual of a left $U$-module $M$ is also a left $U$-module with respect to
    \[
        (x\triangleright f)(m) = f(S^{-1}(x)\triangleright m).
    \]
    From this observation, it follows that the antipode of $U^\circ$ is invertible if the antipode of $U$ is. Moreover, if we consider $U_\mathcal{C}^\circ$ for a tensor category $\mathcal{C}$, then it admits an invertible antipode if $\mathcal{C}$ is closed under taking duals with respect to the action given by $S^{-1}$.

    \section{Adjoint equivalences}

    Here we collect some results about adjunctions and categorical equivalences. If $\mathcal{C}, \mathcal{D}$ and $\mathcal{E}$ are categories, $F, F' \colon \mathcal{C}\rightarrow \mathcal{D}, G, G'\colon \mathcal{D}\rightarrow \mathcal{E}$ functors, and $\varphi\colon F\rightarrow F', \psi\colon G\rightarrow G'$ natural transformations, we denote by $G\varphi \colon GF\rightarrow GF'$ the natural transformation with components $G(\varphi_c)\colon GFc \rightarrow GF'c$, and by $\psi F\colon GF\rightarrow G'F$ the natural transformation with components $\varphi_{Fc}\colon GFc \rightarrow G'Fc$.

    \begin{definition}
        \label{DefAdjointEquivalence}
        Let $\mathcal{C}, \mathcal{D}$ be categories.
        \begin{enumerate}[label = (\arabic*)]
            \item An \textit{adjunction} between $\mathcal{C}$ and $\mathcal{D}$ consists of functors $F\colon \mathcal{C}\rightarrow \mathcal{D}, U\rightarrow \mathcal{D}\rightarrow \mathcal{C}$ and natural transformations $\eta\colon \id_{\mathcal{C}}\rightarrow UF, \varepsilon\colon FU\rightarrow \id_{\mathcal{D}}$ called \textit{unit} and \textit{counit} respectively such That
            \[
                \varepsilon F \circ F\eta = \id_F, \quad U\varepsilon \circ \eta U = \id_U.
            \]
            In this case we say that $F$ is \textit{left adjoint} to $U$, respectively $U$ is \textit{right adjoint} to $F$ and write $F \dashv U$.
            \item We call an adjunction as above an \textit{adjoint equivalence} if the unit $\eta\colon \id \rightarrow UF$ and counit $\varepsilon\colon FU\rightarrow \id$ are isomorphisms.
        \end{enumerate}
    \end{definition}

    \begin{theorem}
        \label{ThmAdjunctionAndEquivalence}
        Let $F, U$ be a pair of adjoint functors, with unit $\eta\colon \id\rightarrow UF$ and counit $\varepsilon \colon FU\rightarrow \id$. If either $U$ or $F$ is an equivalence, then $\eta$ and $\varepsilon$ are isomorphisms. In other words the adjunction is an adjoint equivalence.
    \end{theorem}

    \begin{proof}
        Assume first that $U$ is an equivalence. Then by \cite[IV.4 Theorem 1]{Mac98}, $U$ is part of an adjoint equivalence $(F',U,\eta',\varepsilon')$. Furthermore, by \cite[IV.1 Corollary 1]{Mac98} there is a natural isomorphism $\phi\colon F\rightarrow F'$ such that $\eta' = U\phi\circ \eta$.  Thus, the natural transformation $\eta$ is an isomorphism. By the triangle identity we have $\varepsilon F \circ \eta = \id$. Thus, since $\eta$ is an isomorphism and $F$ is an equivalence of categories, also $\varepsilon$ is an isomorphism.

        If $F$ is an equivalence, then we can reduce to the first case by viewing $U$ and $F$ as functors between $\mathcal{C}^\mathrm{op}$ and $\mathcal{D}^\mathrm{op}$.
    \end{proof}

    \section{Quotients of Hopf algebras induced by quantum homogeneous spaces}

    In the following we let $A$ be a Hopf algebra, $B \subseteq A$ a left coideal subalgebra and $\pi_B\colon A\rightarrow A/B^+A$ the canonical projections. If one wants the quotient $\pi_B(A)$ to be a Hopf algebra, then there are different assumptions on $B$ one can make. While in the literature it is often assumed that $AB^+ = B^+A$ (for instance in \cite[Section 4.1]{BS25}), in Section \ref{SecTakeuchi} we only assume $AB^+\subseteq B^+A$, which is a priory weaker. In the context of a quantum homogeneous space, they are however equivalent, as the following lemma shows.

    \begin{lemma}
        \label{LemABvsBA}
        Suppose that the antipode of $A$ is bijective. Then the following are equivalent.
        \begin{enumerate}[label = (\arabic*)]
            \item $AB^+\subseteq B^+A$.
            \item $\pi_B(A)$ is a quotient bialgebra.
            \item $\pi_B(A)$ is a quotient Hopf algebra.
        \end{enumerate}
        If $A$ is faithfully flat as a right $B$-module, then the above are equivalent to
        \begin{itemize}[label = (4)]
            \item $AB^+ = B^+A$.
        \end{itemize}
    \end{lemma}

    \begin{proof}
        The assertion that $\pi_B(A)$ is a quotient bialgebra is equivalent to the assertion that $B^+A$ is a two-sided ideal, which is intern equivalent to $AB^+ \subseteq B^+A$. This shows the equivalence of \textit{(1)} and \textit{(2)}. If $\pi_B(A)$ is a Hopf algebra quotient, then it is in particular a bialgebra quotient. If $AB^+ \subseteq B^+A$, then by \cite[Lemma 1.4]{MS99}
        \[
            S(B^+A) = AB^+ \subseteq AB^+A = B^+A.
        \]
        Thus, the antipode of $A$ descends to $\pi_B(A)$, and assertions \textit{(1)}, \textit{(2)} and \textit{(3)} are all equivalent. Clearly, assertion \textit{(4)} implies \textit{(1)}. Now suppose that $A$ is faithfully flat as a right $B$-module. To show the last implication, first note that for all $a \in A$ and $b \in B$
        \begin{align*}
            &(S^{-1}(a_{(2)})ba_{(1)})_{(1)}\otimes \pi_B((S^{-1}(a_{(2)}ba_{(1)}))_{(2)})\\
            &\phantom{===}= S^{-1}(a_{(4)})b_{(1)}a_{(1)} \otimes \pi_B(S^{-1}(a_{(3)})b_{(2)}a_{(2)})\\
            &\phantom{===}= S^{-1}(a_{(4)})b_{(1)}a_{(1)}\otimes \pi_B(S^{-1}(a_{(3)}))\pi_B(b_{(2)})\pi_B(a_{(2)})\\
            &\phantom{===}= S^{-1}(a_{(4)})ba_{(1)}\otimes \pi_B(S^{-1}(a_{(3)})a_{(2)})\\
            &\phantom{===}= S^{-1}(a_{(2)})ba_{(1)}\otimes \pi_B(1).
        \end{align*}
        This implies that $S^{-1}(a_{(2)})ba_{(1)} \in B$, using Corollary \ref{CorQHS}. In particular, if $b\in B^+$, then
        \[
            ba = a_{(3)}S^{-1}(a_{(2)})ba_{(1)}  \in AB^+.
        \]
    \end{proof}

    \bibliographystyle{plain}
    \bibliography{references}

\end{document}